\documentclass [12pt,a4paper,reqno]{amsart}

\usepackage{amsmath,amsfonts,amssymb,amscd,verbatim,delarray,verbatim,calrsfs}

\newcommand{\chapter}{\part}

\theoremstyle{plain}
\newtheorem{theorem}{Theorem}[section]

\newtheorem{Theorem}[theorem]{Theorem}
\newtheorem{corollary}[theorem]{Corollary}
\newtheorem{Corollary}[theorem]{Corollary}

\newtheorem{Proposition}[theorem]{Proposition}
\newtheorem{Lemma}[theorem]{Lemma}

\theoremstyle{definition}

\newtheorem{Definition}[theorem]{Definition}

\newtheorem{Remark}[theorem]{Remark}

\newtheorem{Example}[theorem]{Example}

   \newcommand{\ov}{\overline}
\newcommand{\thmref}[1]{Theorem~\ref{#1}}
\newcommand{\secref}[1]{Section~\ref{#1}}

\newcommand{\lemref}[1]{Lemma~\ref{#1}}
\newcommand{\corref}[1]{Corollary~\ref{#1}}

\newcommand{\remarkref}[1]{Remark~\ref{#1}}
\newcommand{\defnref}[1]{Definition~\ref{#1}}
\newcommand{\propref}[1]{Proposition~\ref{#1}}
\newcommand{\eqqref}[1]{Equation~(\ref{#1})}

\DeclareMathOperator{\PSL}{PSL} \DeclareMathOperator{\SL}{SL}
 \DeclareMathOperator{\GL}{GL}
\DeclareMathOperator{\tr}{tr} 
\DeclareMathOperator{\Tor}{Tor}\DeclareMathOperator{\Diff}{Diff}
\DeclareMathOperator{\TD}{TD}\DeclareMathOperator{\Fx}{\mathrm{Fix}}
\DeclareMathOperator{\mult}{mult}\DeclareMathOperator{\rk}{rk}
\DeclareMathOperator{\invl}{inv}
\DeclareMathOperator{\codim}{codim} 
\DeclareMathOperator{\Aut}{Aut} \DeclareMathOperator{\PAut}{PAut} 
\DeclareMathOperator{\Bim}{Bim} 
\DeclareMathOperator{\Hom}{Hom}
  
\DeclareMathOperator{\NS}{NS} 
\DeclareMathOperator{\Bir}{Bir}  \DeclareMathOperator{\Cr}{Cr} 
\newcommand{\BZ}{\mathbb{Z}}
\newcommand{\BC}{\mathbb{C}}

\newcommand{\BP}{\mathbb{P}}
\newcommand{\BF}{\mathbb{F}}\newcommand{\BQ}{\mathbb{Q}}
\newcommand{\BR}{\mathbb{R}}

\newcommand{\al}{\alpha}

\DeclareMathSymbol{\twoheadrightarrow}  {\mathrel}{AMSa}{"10}
               \def\LL{{\mathcal L}} 
         \def\EE{{\mathcal E}}

\def\OO{{\mathcal O}} \def\II{{\mathcal I}}

\begin{document}

\title[Automorphism groups of $\BP^1-$bundles ]{  Automorphism groups of $\BP^1-$bundles over a non-uniruled base}

\author{Tatiana Bandman }\address{Department of Mathematics,
Bar-Ilan University,
Ramat Gan, 5290002, Israel}
\email{bandman@math.biu.ac.il}
\author{Yuri  G. Zarhin}
\address{Pennsylvania State University, Department of Mathematics, University Park, PA 16802, USA}
\email{zarhin@math.psu.edu}
\thanks{The second named author (Y.Z.) was partially supported by Simons Foundation Collaboration grant   \# 585711.
 Most  of this work was done in January--May 2022 during his stay at the Max-Planck Institut f\"ur Mathematik (Bonn, Germany), whose hospitality and support are gratefully acknowledged.}

\begin{abstract}
 
  In this survey we discuss   holomorphic $\BP^1-$bundles $p: X \to Y$  over a non-uniruled complex compact K\"ahler manifold Y, paying a special attention to the case when $Y$ is a complex torus.   We consider   the groups $\Aut(X)$ and $\Bim(X)$ of its biholomorphic and bimeromorphic automorphisms, respectively,   and  discuss when these  groups are bounded, Jordan,  strongly Jordan, or very Jordan.

 \end{abstract}

\maketitle
\tableofcontents 

\section
{Introduction}\label{introduction}
 In this survey   we  consider the groups $\Aut(X)$ and $\Bim(X)$ of  all biregular and  bimeromorphic self-maps, respectively, for a compact complex  connected K\"ahler manifold $X. $  If $X$ is projective, $\Bim(X)=\Bir(X)$  is  the  group 
of all birational transformations of $X $   (see \cite{Serre56}).  The manifolds we are going to deal with are of special type: $X$ has to be a $\BP^1-$bundle over a non-uniruled   compact complex connected manifold $Y.$

  In general, the  groups  $\Bim(X)$  may be very huge and non-algebraic (for example Cremona group $\mathrm{Cr}_n$ of birational transformation of the $n-$dimensional projective space).  Thus one is tempted to study properties of a group  via its  finite and /or abelian subgroups.    Namely, we are interested in the following
properties of groups.

 \begin{Definition}\label{groups} \begin{enumerate}   \item A group $G$ is called {\sl bounded} if the orders of its finite subgroups are bounded by a universal constant that depends only on $G$  (\cite[Definition 2.9]{Pop}).

\item A group $G$ is called {\sl Jordan} if there is a positive integer $J$ such that
every finite subgroup $B$ of $G$ contains an abelian subgroup $A$  that is normal in $B$ and such that the index $[B:A]\le J.$ The  smallest  such  $J$ is called the  {\sl Jordan constant} of $G,$    denoted by $J_G$.(\cite[Question 6.1] {Serre1}, \cite[Definition 2.1]{Pop},\cite{Pop14}).
\item  A Jordan group $G$ is called {\sl strongly Jordan} \cite{PS14,BZ17} if there is a positive integer $m$  such 
that every finite subgroup of $G$ is generated by at most $m$ elements.

\item  A group  $G$  is  {\sl very Jordan} (\cite{BZ20})   if there exist  a  commutative normal subgroup  $G_0$ of $G$ and a bounded group $F$ that sit in  a short exact sequence\begin{equation}\label{veryjordan}
1\to G_0\to G\to F\to 1.\end{equation}
\end{enumerate}\end{Definition}

In what follows    by {\bf Jorfan Properties}  we mean one of those described in \defnref{groups}.
The study of these properties were inspired  by the following fundamental 
results. 
\begin{Theorem}\label{Jordan}(M.-E.-C.  Jordan  (1878),  \cite {Jordan}, \cite[Theorem 9.9]{Serre}) 
Let $\BC$ be the field of complex numbers.
 Then $\GL_n=\GL(n,\BC)$ is  strongly  Jordan.\end{Theorem}

  \begin{Theorem}(J.-P. Serre  (2009), \cite[Theorem 5.3]{Serre1})
  $ \Cr_2= \Bir(\BP^2)$  is Jordan,  \  $J_{\Cr_2}\leq 2^{10}3^45^27.$\end{Theorem}

It was V.L.  Popov who asked   in \cite{Pop}   a question whether for an algebraic  variety $X$  the groups $\Aut(X)$ and $\Bir(X)$ are Jordan.
  The question  originated  an intensive and fruitful  activity. 
  It was proven that there are vast classes of  manifolds   (varieties)  with Jordan     groups $\Aut(X), \Bim(X),$ and $ \Bir (X), $
 see \secref{Known}. In particular, the Cremona group  $Cr_n=\Bir(\BP^n)$   appeared to be Jordan  for all $n$ (\cite{PS14}  and \cite{Bi}) (this is the   positive answer  to   a question formulated by J.-P.  Serre). 
In \secref{Known} we give  a  glimpse on richness of known facts about Jordan properties of $\Aut(X), \Bim(X)$ or $\Bir(X)$  for various types of varieties  $X.$  We do not pretend to give a complete picture. Our aim is to demonstrate
that the "worst" manifolds from this point of view are the uniruled  but not-rationally connected ones. For example, the group $\Bim(X)$ is not Jordan if $X$ is bimeromorphic to a product of a complex torus of positive algebraic dimension and the projective space $\BP^N,  N>0$  (\cite{Zar14}, \cite{Zar19}).

  In this survey we concentrate on the manifolds of this kind. Namely, our main object of consideration  are   $\BP^1-$bundles  over non-uniruled manifolds,
    i.e.,  triples $(X,p,Y)$
such that \begin{itemize}\item $X,Y$ are compact complex connected  K\"ahler manifolds;\item  $p: X\to Y$   is a holomorphic map from $X$ onto $Y;$
\item $Y$ is not uniruled;\item for every  point $y\in Y$ the fiber $p^*(y)$ is isomorphic to $\BP^1;$ in particular, is   irreducible  and  reduced.
 \end{itemize}

 We say that such  a triple $(X,p,Y)$ has  an {\it almost section }  $D$ if  an irreducible  analytic subset   $D\subset X, \ \codim (D)=1,$ meets a general fiber of $p$ at precisely one point  (see    \defnref{almost}). 
We say that such  a triple $(X,p,Y)$ (or $X$,  or morphism $p$)    is 
    {\it  scarce}, if  $X$ does {\it  not }  admit three distinct  almost  sections 
 $A_1,A_2,A_3$ such that  $A_1\cap A_2=A_1\cap A_3=A_2\cap  A_3 $ 
          (see \defnref{configuration}). We say that  a connected compact complex manifold $Y$ is {\it poor}   (\defnref{poor}) if it contains neither rational curves nor  analytic subsets of $\codim  1.$

The  facts that  we know  about  Jordan properties of $\BP^1-$bundles   $(X,p,Y)$    over non-uniruled 
 K\"ahler manifolds   are presented in the following: 

{\bf {Summary}}

\begin{enumerate}\item   $\Aut(X)$ is always Jordan (\cite {Kim}, for surfaces see also \cite{Zar15}) and even strongly Jordan (see \remarkref{Mundet});
\item If  morphism $p$ is scarce   then $\Aut(X)$  is very Jordan (\thmref{main} of this paper).

\item   If $Y$ is a torus and if $X$ is not a projectivization
of a decomposable vector bundle of rank 2 on $ Y,$ then the group  $\Aut(X)$ is strongly
Jordan  (\cite{kostya}).   
\item If $X,Y$ are  projective, and $X$ is not birational to   $Y\times\BP^1, $  then $\Bir(X)$ is   strongly Jordan  (\cite{BZ17});
\item If  $Y$ is a  poor manifold (see   \defnref{poor})  then  $\Bim (X)=\Aut(X)$ and is very Jordan( \cite{BZ20}).
\item  If   $Y$ is a  complex  torus  and there is no almost section of $p$  then $\Bim(X)$ is Jordan (\cite{kostya}).  In particular,  if $X$  is not the projectivization of a rank 2 vector
bundle on $Y$, then the group $\Bim(X)$ is strongly Jordan.
\item  If   $Y$ is a    complex torus   of positive   algebraic dimension and  $X$ is bimeromorphic  (birational,  if $Y$ is projective) to a direct product  $Y\times\BP^1$  then the group $\Bim (X)$ (respectively, $\Bir(X)$) is not Jordan  (\cite{Zar14,Zar19}).

\item If $Y$ is a complex torus of positive algebraic dimension,  $Y_a$   is its algebraic reduction,   $\LL$ is   the lift to $Y$ of a holomorphic line bundle on $Y_a,$ and $X$ is the projectivization of the rank 2  vector bundle $\LL\oplus\mathbf 1$ then $\Bim (X)$ is not Jordan   (\cite{Zar19}).
\item {\bf Open question.}   Assume that $Y$ is a complex torus of positive algebraic dimension and $X$ has no representation as in previous item. Is $\Bim(X)$ Jordan? 
\end{enumerate}

   Our goal is  to give a  review of the methods used to prove these facts.   The  unpublished previously results are provided with full proves.

  All manifolds are compact complex, and  connected,   if not stated otherwise.   All algebraic varieties are complex,   projective, irreducible, reduced. 
 $\BP^n, \BC^n$ are  complex projective    and affine spaces respectively, 
  $\BP^n_k, \BC^n_k$ are  projective    and affine spaces respectively over  an algebraically closed field $k.  $

    The structure of the survey is as follows. In  \secref{Jgroups} we provide facts and examples concerning bounded, 
 Jordan, and very Jordan groups.  In  \secref{manifolds} we  enumerate  Assumptions  and Notation and remind the notions  related  to manifolds and their maps.  
 In \secref{Known} we give examples of the  known facts about Jordan properties of $\Aut(X), \Bim(X)$ and $\Bir(X)$  for various types of manifolds $X.$  Our aim is to demonstrate a special role of  $\BP^1-$bundles over a non-uniruled base in this field.  In \secref{rationalbundles} we  provide some generalities on maps of  $\BP^1$- bundles.   In  \secref{conics}  we  deal with  the group $\Bim(X)  $ of a non-trivial rational bundle (in particar, projective conic bundle).
   In   Chapter  \ref{directproducts} we deal with  certain $\BP^1-$bundles  over complex tori. We present  a unified approach to   proving results of \cite{Zar14}
 and  \cite{Zar19}.  It is based on sympectic algebra, the highly useful tools for studying line bundles over tori
 and inspired by the work of D. Mumford \cite{Mum66}.
    In Chapter \ref{base} we consider $\BP^1$-bundles  $(X,p,Y)$ with scarce set of sections
  over a non-uniruled K\"{a}hler   base.  It contains certain  generalization and modification of the paper \cite{BZ20}.
 First, in \secref{automorphisms}, for a $\BP^1-$bundle $(X,p,Y)$ we consider the group  $\Aut(X)_p$ of those   automorphisms of $X$ that leave  every fiber of $p$ fixed. In three subsection we describe three different types of such automorphisms.  In \secref{Aut0}, under assumption that $Y$ is K\"{a}hler  and not uniruled and $p$ is scarce,  we  prove that the neutral component  $\Aut_0(X)$  of the complex Lie group $\Aut(X)$ is commutative, hence $\Aut(X)$ is very Jordan.
In \secref{tori} we prove that if $Y$ is poor then $p$ is scarce and $\Aut(X) $  is very Jordan.

{\bf Acknowledgements}. We are deeply grateful to Fr\'ed\'eric Campana, Igor Dolgachev, Lei Ni, Constantin Shramov and Vladimir L. Popov
for helpful stimulating discussions.  Our special thanks go to the referee, whose  numerous  valuable comments helped to improve the exposition.

\chapter{Preliminaries }
\label{preliminaries}

In this chapter we provide some backgrounds:
 properties of Jordan groups, the  Notation and Assumptions and definitions.

 \section {Jordan properties of groups} 
\label{Jgroups}

 In this section we recall the general facts about Jordan properties of  groups. 
The following properties follow easily form the \defnref{groups}

 1) Every finite group is bounded, Jordan, and very Jordan. 

2) Every commutative group is Jordan and very Jordan. 

3) Every finitely generated commutative group is bounded. Indeed, such a group is isomorphic to a  finite direct sum with every summand isomorphic to $\BZ$ or $\BZ/n\BZ$  where  $n$ is positive integer. 

4)  A subgroup of a  Jordan group is Jordan.     A subgroup of a very Jordan group is very Jordan.

5) ``Bounded" implies ``very Jordan", ``very Jordan" implies ``Jordan".

6)    ``Bounded"  implies ``strongly  Jordan."    On the other hand, ``very Jordan" does not imply ``strongly  Jordan." 
 For example, a direct sum of infinitely many copies of 
  $\BZ/2\BZ$ 
 is commutative but has finite subgroups with any given minimal number of generators.

\begin{Example}\label{jordanproperties1}
 The group  $\GL(n,\BZ)$ is bounded. It follows from  the following  Theorem of  Minkowski \cite[Section 9.1]{Serre}):

\begin{Theorem}\label{Minkovsky} (Minkowski, 1887).  If  an element  $a\in \GL(n,\BZ)$ is periodic, and $a={\bf 1}\mod m$   for $ \ m\ge  3,$  then  $ a=1.$\end{Theorem}

It follows that every finite subgroup $H\subset \GL(n,\BZ)$ embeds into $\GL(n,\BZ/{3\BZ})$,
 (there are much more precise bounds,   \cite[Theorem1.1]{Serre20}).
Since every finite subgroup of $\GL(n,\BQ)$ is conjugate to  a subgroup of 
$\GL(n,\BZ)$   (\cite [Lecture 1]{Serre20}),  the group $\GL(n,\BQ)$ is bounded as well.\end{Example}
\begin{Example}\label{jordanproperties2} The multiplicative group   $\BC^*$ of $\BC$ is commutative, very Jordan but not bounded. The same is valid for   the group of translations of a  complex torus of positive dimension.   
\end{Example} 
\begin{Example}\label{jordanproperties3} From \thmref{Jordan} it follows that the group $\GL(n,k)$ is strongly Jordan for every field $k$ of characteristic zero.  Moreover, every linear algebraic group over $k$ is  strongly  Jordan.  On the other hand, $\GL(n,k)$  is obviously not very Jordan if $n \ge 2$.
 \end{Example}

  The following precise values of Jordan constants for groups $\GL(n, \BC)$ were found by  M.J. Collins. 
\begin{Theorem}\label{Collins}(\cite[Theorems {\bf A} and {\bf B}]{Collins}).  For the Jordan constants   of  groups $\GL(n,\BC)$ the following relations hold.

(1) 
 \ \ $J_{\GL(n,\BC)}= (n+1)!$ if  $ n\ge 71$   
or $n=63,65,67,69.$ 

(2) 
\ \ $J_{\GL(n,\BC)}= 60^r\cdot r! $     if $20\le n\le 62$ or  $n=64,66,68,70,$ where $r=[n/2].$

\end{Theorem} 

 The  information on  values of Jordan constants for groups $\GL(n, \BC), \ n< 20,$  is given in  extensive tables provided in the same paper.

 \begin{Example}\label{jordanproperties4}  

 We will use below
analogues of the Heisenberg groups that were used by D. Mumford \cite{Mum66}.  Let 
\begin{itemize}\item
$\mathbf{K}$  be a  finite commutative group of order $N>1;$
\item
 $\mu_N \subset \BC^{*}$  be  the multiplicative group of $N$th roots of unity;
\item
 $\hat{\mathbf{K}}=\mathrm{Hom}(\mathbf{K},\mu_N)$ (the dual of $\mathbf{K}$).\end{itemize}

 The Mumford's theta group $\mathfrak{G}_{\mathbf{K}}$ 
  for  $\mathbf{K}$ 
is a group of matrices  of the type
$$\begin{pmatrix}
1 & \alpha &\gamma \\
0 & 1 & \beta \\
0 & 0 &  1
\end{pmatrix}$$
 where $ \alpha \in\hat{\mathbf{K}},$ \  $\gamma\in  \BC^{*},$  and $ \beta\in\mathbf{K}.$   The product of $ \alpha \in\hat{\mathbf{K}} $ and $ \beta\in\mathbf{K}$  is defined by a certain natural non-degenerate  alternating   bilinear form $e_{\mathbf{K}}$ on 
$\mathbf{H}_{\mathbf{K}}=\mathbf{K}\times\hat{\mathbf{K}}$  with values in $\BC^{*}$
\cite[p. 302]{Zar14}.
 This group may be included into 
 a short exact sequence  
$$1 \to \BC^{*} \to \mathfrak{G}_{\mathbf{K}} \to \mathbf{H}_{\mathbf{K}} \to 1$$
where the image of $\BC^*$  is the center of $\mathfrak{G}_{\mathbf{K}}.$

 Properties of $\mathfrak{G}_{\mathbf{K}}$  \cite[p. 302]{Zar14} imply that it is a {\sl theta group} attached
to the {\sl nondegenerate symplectic pair} $\left(\mathbf{H}_{\mathbf{K}}, e_{\mathbf{K}}\right)$ in a sense  
 of Chapter \ref{directproducts} below. By Theorem \ref{JordanTheta} below, $\mathfrak{G}_{\mathbf{K}}$
   is Jordan and  
$$J_{\mathfrak{G}_{\mathbf{K}}}=\sqrt{\#(\mathbf{H}_{\mathbf{K}})}=N=\#(\mathbf{K}).$$
In particular, if $K=\BZ/N\BZ$ then $J_{\mathfrak{G}_{\BZ/N\BZ}}=N.$  
\end{Example}

\begin{Example}\label{jordanproperties5}  The example of a  non-Jordan group is   given by $\SL(2, \ov \BF_p)$ where $\ov {\BF}_p$  is   the algebraic closure of a  prime finite field $\BF_p$ with   $p$ elements. 

Indeed, if  $q=p^n\ge 4,$ then $\mathrm{SL}(2,  {\BF_q})\subset \mathrm{SL}(2, \ov {\BF_p})$
(Here $\BF_q$ is the $q$-element field.)
  Group   $\mathrm{SL}(2,  {\BF_q})$ is   noncommutative, finite, and has   order $(q^2-1)q$.
Every  normal subgroup  $C\subsetneq \mathrm{SL}(2,  {\BF_q})$
consists of one or two scalars.  thus 
the indices
$$[\mathrm{SL}(2,  {\BF_q}):C] =(q^2-1)q/2 \ \mathrm{ or } \ (q^2-1)q$$ 
 are unbounded when $n$ tends to infinity. 
 Hence 
$\SL(2, \ov {\BF_p})$  is { not} Jordan.

\end{Example}

\begin{Remark}
\label{pJordan}
An analogue of the theorem of Jordan  holds for matrix groups over fields $k$ of 
 prime characteristic $p$ if one considers only finite subgroups, whose order is prime to $p$. 
 On the other hand, there are generalizations of the theorem of Jordan (Brauer-Feit \cite{BF}, Larsen-Pink \cite{LP}) that deal with arbitrary finite subgroups and take into account the order of their Sylow $p$-subgroups.  Their results led to the following definition \cite[Definition 1.6]{Hu}
 (that will be used in Remark \ref{diff}, part 4 below).
 
 A group $G$ is called {\sl $p$-Jordan}  if there exist positive integers $J$ and $e$ such that every finite subgroup $B$ of $G$ contains an abelian $p^{\prime}$-subgroup $A$ that is normal in $B$ and such that the index $[B:A] \le |B_p|^e  J$.  Here $|B_p|$ is the order of a Sylow $p$-subgroup of $B$.

\end{Remark}

\begin{Remark}\label{exactsequences}
Let $G$ be a group. 
 Assume that it may be included into the following exact sequence  of groups

$0\to H\to G\to F\to 0.$

\begin{enumerate}\item If $F$ is bounded and $H$ is bounded then  $G$ is bounded ; 
\item  If $H$ is very  Jordan and $F$ is bounded then  $G$ is  very Jordan; 

\item  If  $F$ is bounded     then  $G$ is Jordan  if and only if  $H$ is Jordan   (\cite[Lemma 2.11]{Pop});

\item  If  $H$ is bounded and $F$   is   strongly Jordan  then  $G$ is Jordan  (\cite[Lemma 2.8]{PS14}).

\item  $G$  being Jordan {\bf does   not} imply  that  $F$   is Jordan ;
\item $F$ and $H$ being Jordan {\bf does   not} imply  that  $G$ is Jordan. 
\end{enumerate}\end{Remark}

   We will need  the following modification of  \cite[Lemma 5.3]{BZ20}.
\begin{Lemma}\label{commutative}   Consider  a short exact sequence  of connected complex Lie groups:
$$0\to A\overset {i}{\rightarrow} B\overset {j}{\rightarrow} D\to 0.$$ 
Here $i$ is a closed holomorphic embedding and $j$ is surjective holomorphic. Assume that $D$ is a complex  torus and $A$ is isomorphic as a complex Lie group  either to $ 
 (\BC^+)^n$ or to $\BC^*.$  Then $B$ is commutative.\end{Lemma}

\begin{proof} The proof of this lemma  coincides verbatim with the proof of \cite[Lemma 5.3]{BZ20}  where the case  $n=1$  is treated.

   {\bf Step 1.}  First, let us prove    that $A$ is a central subgroup  in $B.$ Take any element $b\in  B.$ 
 Define  a holomorphic map $\phi_ b:A\to A, \ \phi_b(a)= bab^{-1}\in A$
for an element $a\in  A.$  Since  it depends holomorphically on $b,$  
we have  a holomorphic map $\xi :B\to \Aut(A),  b  \to \phi_b.$

 Since $A$ is commutative, for every $c\in A$
we have $\phi_{bc}=\phi_b.$ Thus   there is a well defined map $\psi$ fitting into the following commutative diagram \begin{equation}\label{kawamata}
\begin{aligned}
&        &   &                 &  {B}                              &                  &      &             & \notag \\
&        &j  & \swarrow  &                                        &{\searrow} &\xi   &               &  \notag \\
& D&   &                &\stackrel{\psi}{\rightarrow} &                  &       &\Aut(A)&
      \end{aligned},  \end{equation}
The map $\psi=\xi\circ j^{-1}$ is  defined at every point of $D.$  
 It is holomorphic (see, for example, \cite{OV}, \$ 3). 

Since $D$ is a complex torus,  and  $\Aut(A)$ is either $\GL(n,\BC)$ ( if $A= (\BC^+)^n$) or 
consists of  two elements, $id$ and $a \mapsto a^{-1}$
(if $A=\BC^*$), we have $\psi(D)=\{id\}.$  It follows that $A$ is a central subgroup of $B.$

     {\bf Step 2.}   Let us now show that $B$ is commutative.  Consider a  holomorphic  map  $\mathrm{com} :B\times B\to A$  defined by 
$\mathrm{com}(x,y)=xyx^{-1}y^{-1}.$  Since $A$ is a central subgroup of $B, $  similarly to    {\bf Step 1}
we get a holomorphic map  $D\times D\to A.$  It has to be constant since $D$ is a torus and  $A$ is either $ (\BC^+)^n$ or $\BC^*.$\end{proof}

\section{ Complex manifolds}\label{manifolds}

This section   contains preliminaries, Notation, and Assumptions that will be used further on.

By {\sl  (projective) variety} we   mean an algebraic variety that is  Zariski closed subset of a projective space  $\BP^n.$

Let $U\subset\BC^n$  be an open subset.   An analytic subset of $U$
is a closed subset $X\subset U$  such that for any $x\in X$  there exists an open
neighborhood $x \in V \subset U$  and holomorpphic functions $f_1,\dots,f_k: V\to\BC$
such that $X\cap V=\{f_1=0,\dots,f_k=0\}$  (\cite [Definition 1.1.23]{Huy}).

A complex space consist of a Hausdorff topological space
$X$  and a sheaf of rings   $\OO_X$ such that locally $(X,\OO_X)$ is isomorphic to an
analytic subset $Z \subset U\subset \BC^n$   endowed with the sheaf
$\OO_U/\II, $   where $\II$ is a sheaf of holomorphic functions with $Z = Z(\II)$(\cite [Definition 6.2.8]{Huy}).
   By the Chow Theorem   any closed analytic subset of complex projective space  is a  projective  variety.  (\cite[Chapter V,  Section D, 7]{GR},\cite[Proposition13]{Serre56}).

 A  complex manifold  is  a
complex space which is locally modeled on $Z = U\subset \BC^n$ and $\II=\{0\}$ (\cite [Example 6.2.9]{Huy}).
 In particular, it is smooth.

We will use the following

 {\bf Notation  and Assumptions.}
 
  \begin{enumerate} \item$\BZ,\BQ,\BR,\BC$ stand for the ring of integers and fields of rational, real, and complex numbers, respectively.
\item  In what follows, the ground field is $\BC$ if not indicated otherwise.

\item   $\Aut(X)$ stands for the group of   all biholomorphic   (or biregular, if $X$ is projective) automorphisms of  a  complex  manifold $X.$   
 The group $\Aut(X)$ of any complex compact manifold $X$ has a canonical structure of  complex (not necessarily connected) Lie group such that the action map $\Aut(X)\times X \to X$ is holomorphic (Theorem of  Bochner-Montgomery,  \cite{BM}).
\item  $\Aut_0(X)$  stands for the connected identity  component of  $ \Aut(X)$  (as a complex Lie group)  for a compact manifold $X$.
\item If $p:X\to Y$ is a morphism of complex manifolds, then $\Aut(X)_p$ is the subgroup of all  $ f\in \Aut(X)$  such that $p\circ f=p.$
\item For $f\in\Aut(X)$  we denote by $\mathrm{Fix}(f)$   the set of all fixed points of $f.$
\item $\cong$ stands for  `` isomorphic groups" (or isomorphic complex Lie groups if the groups involved are the ones),  and also for isomorphic line bundles; $\sim$ for biholomorphically  isomorphic complex manifolds; $\approx$   for bimeromorphic  or birational complex manifolds.
\item   $id$  stands for identity automorphism,  $\mathbf {I}$  stands for identity matrix.
\item  We say that a subset $U$ of a complex manifold   $X$ is   analytical  Zariski open if $U=X\setminus Z,$  where $Z$ is a closed  analytic subspace of $X.$
\item $\BP^n_{(x_0:...:x_n)}$ stands for  $\BP^n$ with homogeneous coordinates $(x_0:...:x_n).$
\item  $\BC_z, \ov{\BC}_z\sim\BP^1$  is the  complex line (extended complex line, respectively) with coordinate $z$.
\item $\BC^+$ and $\BC^*$ stand for   complex Lie groups $\BC$ and $\BC^*$ with additive and multiplicative group structure, respectively.
\item $\dim (X),  \ \dim_a(X) $ are  the  complex and algebraic dimensions of a compact complex manifold $X$, respectively.
\item By  $\mathbf {pr}$ we  denote the natural projection $Y\times\BP^1\to Y.$
\item For an  element $\psi\in\PSL(2,\BC )$ we denote $\TD(\psi) $  the number $$\TD(\psi):=\frac{\tr (F)^2}{\det ( F)}
,$$ where $F\in \GL(2,\BC)$ is any representative of $\psi,$ and $\tr$ and $\det$ stand for trace and determinant, respectively. 
\item  A rational curve in $X$ is the  image of a non-constant holomorphic map $\BP^1\to X.$

\item $\mathbf 1$   or  $\mathbf 1_Y$  is a trivial holomorphic line bundle   $Y\times\BC$  over  a manifold $Y.$
\item  For a rank 2 holomorphic vector bundle   $\EE$ over  $Y$ we write $ \BP(\EE)$  for  the 
 $\BP^1-$bundle that is the projectivization 
of $\EE.$
\item If $\LL$ is a holomorphic line bundle over $Y$   and $\EE=\LL\oplus \mathbf 1_Y$ then we call $\ov \LL= \BP(\EE)$ {\it the closure }   of $\LL$ viewed as a complex manifold.

\item $\BC  (Z)$ stands for  the  field of rational functions on an irreducible complex projective variety $Z$.  
\item    Let $X,Y$ be two compact connected irreducible reduced analytic complex spaces. A meromorphic map   $f:X\to Y$ relates to every point  $x\in X$  a subset $f(x)\subset Y$
(the image of $x$)  such that the  following conditions are met  \begin{enumerate}\item The graph $G_f:=\{(x,y) \ | y\in f(x)\subset X\times Y\} $ is a connected irreducible
complex analytic subspace of $X\times Y$ with $\dim (G_f)=\dim ( X);$
\item There exist an open dense subset $X_0\subset X$ such, that $f(x)$ consists of one point for every $x\in X_0.$ \end{enumerate} 
\item We say that a property is valid for the  {\it general} point $x\in X$ if it is valid for every point from a certain (analytical)   Zariski open dense subset of $X.$   A  property is valid for the general fiber of a holomorphic map $f: X\to Y$ if it is valid for the fiber  $f^{-1}(y) ,$ for every point 
$y$  of a certain (analytical)   Zariski open dense subset of $Y.$  


\end{enumerate}

 \begin{Definition}\label{covering} Following \cite{Fu81}, we define a covering family of rational curves for a compact complex connected manifold $X$ as a pair of morphisms $p:Z\to T$ and $ \psi:Z\to X$  of compact irreducible complex spaces if the following conditions are satisfied: 
 
 \begin{itemize}
 \item[(1)]
 $\psi $ is surjective;   
 \item[(2)]
   there is a dense analytical  Zariski open subset $U\subset T$ such that for $t\in U,$  the fiber $Z_t=g^{-1}(t)\sim\BP^1$ and $\dim (\psi(Z_t))=1.$
   \end{itemize}\end{Definition}
 
   Manifolds $X$ admitting a covering family with this property are called  {\bf uniruled} (\cite[Definition  2.1, Lemma 2.2] {Fu81}).

\begin{Remark}\label{uni-Kod}
The Kodaira dimension $\kappa(X)=-\infty$ if $X$  is  unuruled compact complex manifold (\cite[Remark, p. 691]{Fu81},\cite[Corollary IV.1.11]{Kollar})
 In low dimensions the converse is true:
\begin{Theorem}\label{Kodaira}(\cite{Mi88}  for projective manifolds, \cite{HP} for  non-projective  ones).
 Let $X$ be a compact K\"{a}hler manifold of
dimension at most 3. Then X is uniruled if   and only if $\kappa(X) = -\infty.$\end{Theorem}\end{Remark}

\begin{Remark}\label{Fujiki}{\bf Fujiki Theorems.} It was proven by A. Fujiki  for a compact connected complex manifold $Y$ that \begin{enumerate}\item
If          $Y$  contains no rational curves   then 
  every meromorphic map $f:X\to Y$  is holomorphic for any complex manifold $X$ (\cite{Fu80}).
 \item  $\Aut_0(Y)$  is isomorphic to a complex torus  $\Tor(Y)$  (unless it is trivial) if  $Y$ is  K\"ahler  and either non-uniruled \cite[Proposition 5.10]{Fu78}
 or has non-negative Kodaira dimension \cite[Corollary 5.11]{Fu78}. 
  \end{enumerate}\end{Remark}

The next statement (\cite[Proposition 1.4]{BZ20} is similar to Lemma 3.1 of \cite{Kim}.
\begin{Proposition}\label{abounded}
Let $X$ be a connected complex compact   K\"{a}hler   manifold  and   $F=\Aut(X)/\Aut_0(X).$ Then the group 
$F$ is bounded.\end{Proposition}
 \begin{Remark}  Lemma 3.1 of Jin Hong Kim,  \cite{Kim}, states the following.
 
 Let $X$ be a normal compact K\"{a}hler variety. Then there exists a positive
integer $l,$ depending only on X, such that for any finite subgroup $G$ of $\Aut(X)$
acting biholomorphically and meromorphically on $X$ we have
$[G : G \cap \Aut_0(X)] \le l.$
 
 We cannot use straightforwardly  this Lemma,  since it is not clear why every finite subgroup of $\Aut(X)/\Aut_0(X)$ should be isomorphic to  $G/(G\cap  \Aut_0(X))$ for some  finite subgroup  $G$ of $\Aut(X).$\end{Remark}
 \begin{Corollary}\label{vj}  Let $X$ be a compact connected complex K\"{a}hler manifold either non-uniruled or with  Kodaira dimension $\kappa(X)\ge 0.$ Then $\Aut(X)$ is very Jordan. \end{Corollary}
 \begin{proof} In view of \propref{abounded} it is sufficient to prove that $\Aut_0(X)$ is commutative.   But this assertion follows from 
  \cite[Proposition 5.10]{Fu78}  if $X$ is non-uniruled and 
 \cite[Corollary 5.11]{Fu78} if  $\kappa(X)\ge 0$ (see \remarkref{Fujiki}). \end{proof}

 In general,  let $Z$ be a  compact complex  connected  K\"{a}hler    manifold. The 
   analogue of the Chevalley  decomposition for algebraic groups is valid for complex Lie group $\Aut_0(Z):$
\begin{equation}\label{Chevalley}
1\to L(Z)\to \Aut_0(Z)\to \Tor(Z)\to 1\end{equation}
where $L(Z) $ is     bimeromorphically     isomorphic to a linear group, and $\Tor(Z)$ is a complex  torus        (\cite[Theorem 5.5]{Fu78},     \cite [Theorem 3.12]{Lie}, \cite[Theorem 3.28]{CP}).

\begin{Remark}\label{LZ} If $L(Z)$  in \eqqref{Chevalley} is not trivial,   $Z$ contains a rational curve.  Moreover, according  to
\cite[Proposition  5.10]{Fu78},  $Z$ is bimeromorphic to a fiber space  whose general fiber is $\BP^1,$ i.e $X$ is {\it uniruled. }\end{Remark}

\chapter{Rational bundles}
In this chapter,    in \secref{Known}, we want to persuade   the reader that uniruled manifolds (in particular, $\BP^1-$bundles) are of special interest from the Jordan properties point of view. To this end we give a very brief and certainly non-complete overview  of known facts in this field.  In \secref{rationalbundles}
 we provide general properties of maps of manifolds endowed with fibration over a non-uniruled base with the general fiber $\BP^1.$   In  \secref{conics}
we deal with projective non-trivial conic bundles.

\section{ Uniruled vs non-uniruled:  Jordan properties  of groups $\Aut(X), \Bim(X), $ and  $\Bir(X)$.}\label{Known}

In order to demonstrate the  special  role of uniruled   manifolds   from Jordan Properties point of vew,   
we 
present samples of   results on Jordan Properties   of  $\Aut(X)$ and $\Bim(X)$ for various types of  compact  complex manifolds   $X.$

 The group $\Aut(X)$ is  known to be Jordan if \begin{itemize}\item
   $X$ is projective ( 
   \cite {MengZhang});
 \item $X$ is a compact complex   K\"{a}hler   manifold    (\cite{Kim}):
\item  $X$ is   a compact complex space in Fujiki's Class $\mathcal C$ ( 
\cite {MPZ},  also \cite{PS19}  for Moishezon 
 threefolds );
 \end{itemize}

\begin{Remark}\label{Mundet} For the group $\Aut(X)$ `` Jordan'' implies ``strongly Jordan''    because:

{\sl  For every compact complex manifold $X$ there is a constant $C=C(X)$ such that every finite subgroup $G\subset\Aut(X)$ may be generated by at most $C$ elements.}

  The proof of this fact one can find in \cite[Theorem  1.3]{MR13}. It is based on the  same property for elementary abelian $p$-groups that was proved  for much wider class of topological spaces  in
\cite{MS}, and the group-theoretic arguments (that, according to the author,  were explained to him by E.  Khukhro and A. Jaikin). 
 Thus the fact is valid   in much more general situation. 
\end{Remark}

  Moreover, the connected identity component  $\Aut_0(X)$ of $\Aut(X)$ is Jordan for every compact  complex space $X$
  (\cite[Theorems 5  and 7]{Pop18}).
    An example of   $X=E,$  where  $E$ is an elliptic curve,  shows that $\Aut(X)$ may be Jordan but not bounded.
  The classification of complex compact surfaces  with bounded automorphisms group was done in \cite{PS21-1}.

    As follows from  \corref{vj}, the  group  $\Aut(X)$ is very Jordan for any  compact connected complex K\"{a}hler non-uniruled manifold   $X. $ 
 For uniruled     manifolds the situation changes:  if $X=E\times \BP^1$   then $\Aut(X)\cong\PSL(2,\BC)\times\Aut(E)$ is neither bounded nor very Jordan.

  The groups  $\Bir (X)$ and $\Bim (X) $ of birational and bimeromorphic transformations, respectively, are    more complicated. 
  Low-dimensional cases  are well understood. Consider the following

   {\bf LIST}

\begin{enumerate}
\item $E$ - an elliptic curve;
\item
 $A_n$ - an abelian variety  of dimension $n;$
 
 \item  $S_b$ - a bielliptic surface;\item  $S_{K1}$ - a surface of Kodaira dimension 1;\item  $S_{K}$ - a  Kodaira  surface (it  is not a   K\"{a}hler   surface).\end{enumerate}

  Here  are examples of  results for low-dimensional cases.\begin{itemize}
   \item     If $X$ is a complex compact surface  with non-negative Kodaira dimension  then $\Bir(X)$ is bounded unless it appears in the  above
{\bf LIST}
  \cite[Theorem1.1]{PS20}. 
 
 \item If $X$ is a projective surface then  $\Bir(X)$ is Jordan if  $X$ is not birational to a product of an elliptic curve and $\BP^1,$  (\cite{Pop}). (The case of $X=\BP^2$ was done earlier by J.-P. Serre, \cite{Serre1}).
\item If $X$ is  birational to a product of an elliptic curve and $\BP^1$ then  $\Bir(X)$ is   not  Jordan   (\cite{Zar14}).

\item   If $X$ is a projective threefold then $\Bir(X)$ is  not Jordan if and only if  $X$ is birational to a direct product $E\times \BP^2$ or $S\times \BP^1,$ where a surface $S$  appears in the   above {\bf LIST}
 \cite {PS18}.

\item The group $\Bim(X)$ is Jordan for any non-uniruled compact complex connected K\"ahler
manifold  of dimension 3 (\cite{PS21},\cite{Golota}).

\item If $X$ is a non-algebraic  uniruled   compact K\"ahler threefold  with  non-Jordan       group $\Bim(X)$ then $X$ is bimeromorphic to $\BP (\EE) $ for a holomorphic rank 2 vector bundle $\EE$ on a two-dimensional complex torus $S$ with $a(S)=1.$ Moreover, if $a(X)=2$ then  $X\approx S\times\BP^1$(\cite{PS19-2}).

\end{itemize}

The following Theorem  for complex projective varieties    was proved by Yu. Prokhorov and  C. Shramov   (for $\dim (X)>3$, assuming a so called  {\it BAB-conjecture} named after A. Borisov,  L. Borisov and V. Alexeev),  and C. Birkar
(who proved this conjecture),   (\cite [Theorem 1.8]{PS14}, \cite{Bi}).

\begin{Theorem}\label{PS14} 
Let $X$  be a projective variety of dimension $n.$  Then the following  hold.
\begin{itemize}    \item [(i)]  The group  $\Bir(X)$ has bounded finite subgroups provided that  $ X$  is non-uniruled
and has irregularity $q(X) = 0.$
\item[(ii)] The group $\Bir(X) $ is Jordan provided that $X$ is non-uniruled.
\item[(iii)]   The group   $\Bir(X)$ is
Jordan provided that $X$ has irregularity $q(X) = 0.$\end{itemize}
\end{Theorem}

  Here $q(X)=\dim_{\BC}H^1(X,\mathcal{O}_X)$ is the irregularity of $X.$   In particurar, the Cremona group   $Cr_n$  of any rank $n$
  is Jordan (\cite{PS16})).    The exact value  $J_{\Cr_2}=7200$  (E. Yasinsky, \cite{Yasinsky}). The Jordan constant for $\Bir(X)$ for a  rationally connected threefold  $X$ may be 
  found   in \cite{PS17}.

  Let us sketch   the proof  of items  (i) and (ii)   of \thmref{PS14}.

   First, using the MMP(Minimal Model Program) the authors 
 reduce the problem to consideration of the group   $\PAut(X_m),$
 where $X_m$ is a special (relatively minimal) model of $X$ and 
 $\PAut(Z)$ stands for the group of birational selfmaps  of a variety $Z$ that are   isomorphisms in codimension 1.  This means that $f\in \PAut(X_m)$ moves   a divisor to a divisor  and induces an automorphism $f_*=\psi(f)$  of the  finitely generated abelian group
 $\mathrm{NS}^W(X_m)=\mathrm Cl(X_m)/\mathrm {Cl}^0(X_m),$ were  $\mathrm {Cl}(X_m)$ stands for group of Weil divisors on $X_m$ modulo linear equivalence, and $\mathrm {Cl}^0(X_m)$ consists of those ones that are algebraically equivalent to zero.  

   Thus there is a short exact sequence 
\begin{equation}\label{PSex3}
0 \longrightarrow    G_i    \overset{i}\longrightarrow        G  \overset{\psi}\longrightarrow   \Aut(\mathrm{NS}^W(X)), \end{equation}

 where $G_i=\ker (\psi)$  acts on each of  equivalence  classes of $\mathrm {Cl}(X_m).$     Since $\mathrm {NS}^W(X_m)$ is finitely generated abelian group,  $\Aut(\mathrm{NS}(X))$ is bounded. 

 Take a very ample divisor $L$ and denote by $\mathrm Cl_L(X_m)$  the equivalence class containing $L. $  It is an abelian variety of dimension $q(X_m)=q(X).$

   Let $G_L$ be  the kernel of the action   of $ G_i $   on $\mathrm {Cl}_L(X_m).$   
     Then there is a short exact sequence 
 \begin{equation}\label{PSex4}
0 \longrightarrow        {G_L}  \longrightarrow        G_i \longrightarrow        G_{ab}\end{equation} 
 
 where $G_{ab} \subset \Aut(\mathrm {Cl}_L(X_m))$ is a subgroup of automorphisms  (as a variety, but not as a group)  of abelian variety  $\mathrm {Cl}_L(X_m).$     The group   $\Aut(\mathrm Cl_L(X_m))$ 
 is strongly Jordan. 
 Let $V$ be  a linear space of sections of $L$   and $ \BP(V)$ its projectivization.  Let $F_L$ be the subgroup of those 
linear transformations  of  the projective space $\BP(V)$   that preserve $X_m\subset \BP(V).$  Since $F_L$  is a linear group and    $X$ (and $X_m$) are {\bf non-uniruled},  $F_L$ has to be finite   (see \remarkref{LZ}). 
  Thus $G_L\subset F_L$ is finite.
 
  Therefore\begin{itemize} \item If $q(X)=0,$  then $G_{ab}$ is trivial and $\Bir(X_m)$ is bounded (see \remarkref{exactsequences},(1)).
\item If $q(x)>0 $ then  $G_i$ is Jordan   (see \remarkref{exactsequences},(4))  and $\Bir(X)_m$ is Jordan     (see \remarkref{exactsequences},(3)). 
 \end{itemize}

\begin{Remark}
\label{diff}  \ \ \  
\newline  {\bf 1.} One can ask similar  questions about the  group $\Diff (M)$ of all diffeomorphisms  of a smooth manifold $M.$
 There was the  Conjecture of E. Ghys  (1997):

{\it If $M$ is a  compact smooth manifold, then $\Diff (M)$ is Jordan.}

It was answered negatively by B. Csik\'{o}s, L. Pyber, E.Szab\'{o}  in \cite{CPS},
whose approach was based on an algebraic geometry construction from \cite{Zar14}
(see also Chapter \ref{directproducts} below).

In works of  J. Winkelmann \cite{W}  and  V. Popov    \cite{Pop15}    it was proven that  there is a connected 
 non-compact  Riemann surface $M$
 such that 
$\Aut(M)$ contains 
an isomorphic copy of every finitely presented (in particular, every finite) group $G$.
In particular, $\Diff(M)$ is not Jordan.
 B. Zimmerman \cite{Zim} proved that if $M$ is compact and $\dim(M)\le 3$  then  $\Diff(M)$  is  Jordan. 
 The Jordan properties of $\Diff(M)$ were deeply studied by  I. Mundet i Riera  (\cite {MR10}, \cite {MR15}, \cite {MR16}, \cite {MR17},
\cite {MR17b}, \cite {MR18}).  It was proven there, in particular, that $\Diff(M)$  is Jordan if $M$ is one of the following:

(1) open acyclic manifolds, 

(2) compact manifolds (possibly with boundary) with nonzero
Euler characteristic, 

(3) homology spheres.

{\bf 2.} The question on Jordan properties for algebraic groups over various fields was considered in \cite{Pop18},\cite {MengZhang}, and \cite {ShVb}
(see also \cite{BZ17}).

{\bf 3.}  Jordan properties  of $\Aut (X)$ and $\Bir (X)$ for  open subsets of certain projective   $\BP^1-$bundles   were considered in \cite{BZ15}, \cite{BZ18}.

{\bf 4.}  In the case of algebraic varieties $X$ over algebraically closed fields  of prime  characteristic  $p$ one should not expect the Jordan properties to hold (see Example \ref{jordanproperties5}). However, there are  analogues of several important results over $\BC$ that deal instead with $p$-Jordan properties (see  \remarkref{pJordan}) of $\Aut (X)$ and $\Bir (X)$ (\cite{Hu}, \cite{CS},  \cite{Kuz}).
On the other hand, it is known that the Cremona group 
of rank $2$
over a {\sl finite} field is Jordan \cite{PS21-c}.

\end{Remark}

     For compact complex manifolds,  roughly speaking, from Jordan properties point of view the uniruled varieties   are the worst and may be divided in several categories.

  First,    manifolds   $X$ that are rationally connected (or with $q(X)=0 $).   For projective varieties, thanks to \thmref{PS14},  $\Bir(X)$ is Jordan.

Second,   manifolds that are fibered over a non-uniruled base  $Y$   with rationally connected fibers, with $q(X)\ne 0,$   that 
are not  bimeromorphic (birational) to a direct product $Y\times \BP^N.$     In many special cases $\Bim(X)$ (or $\Bir(X)$) is Jordan. 
Moreover, $\Aut(X)$ appears often  to be very Jordan.  
  We discuss some of these special cases in Chapter{  \ref{base}.

Third,  $X$ is isomorphic (bimeromorphic) to the direct product $Y\times \BP^N.$ If $Y$ is a torus, and  $a(Y)>0$ then $\Bir(Y)$ is not Jordan. This case is subject of Chapter   \ref{directproducts}.

 \section{ Rational bundles}\label{rationalbundles}
 
 In this section we provide some  useful  about  $\BP^1-$ bundles and their morphisms.   We start with slightly more general construction.

\begin{Definition}\label{rationalbundles1}
 We say that  a triple  $(X,p,Y)$   is a  {\it  rational bundle} over $Y$ if \begin{itemize}
 \item  $X,Y$ are compact connected   complex     manifolds  endowed with a holomorphic surjective map   $p:X\to Y;$

    \item  for a general $y\in Y$  the fiber $p^*(y)$ is reduced and isomorphic to $\BP^1$  (where {\sl general} means outside a proper  analytic subset of 
    $Y$, see {\bf Notation and Assumptions} (20)).

 \item  If $\dim  ( P_y)=1, $  where $P_y:=p^{-1}(y),$   for every $y\in Y$  we call $(X,p,Y)$   an {\it equidimensional rational bundle}  over $Y.$  \end{itemize}
\end{Definition}

 If for an open subset $U\subset Y$  and for every   $y\in U$     the fiber $P_y\sim \BP^1$  then, by a theorem of W. Fischer and  H. Grauert  (\cite{FG}),  $p^{-1}(U)\subset X$ is a holomorphically locally trivial fiber bundle over $U.$   If $U=X$ then triple  $(X,p,Y)$  is   a   $\BP^1$-bundle over $Y.$

  If $(X,p,Y)$ is a
    rational bundle 
over a non-uniruled K\"{a}hler
manifold $Y$  then $p:X\to Y$ is, by definition,  a maximal rational connected (MRC) fibration  of $X $ 
 (see \cite[Theorem 2.3, Remark 2.8]{C92}  and  
\cite [IV.5]{Kollar}  for the definition and discussion). 

 Bimeromorphic self-maps preserve the MRC-fibration.  This is a well-known fact, but we have not found a suitable  reference for  the proof of this fact in complex analytic case. We provide it here.  In case when   the Kodaira dimension $\kappa(Y)\ge 0$, the desired result follows from \cite[Theorem   1.1.5]{Mobuchi}.  For automorphisms the detailed exposition may be found in \cite[Section 2.4]{Akhiezer}.

\begin{Lemma}\label{mrc} Let $X, Y,Z $ be three complex compact connected manifolds,  $p:X\to Y$ and $q:X\to Z$ be  surjective holomorphic  maps. Assume that 
\begin{itemize}
\item $Z$ is      non-uniruled ; 
 \item
 there is an analytical  Zariski  open dense subset $U\subset Y, $    such that $P_u=p^{-1}(u)\sim\BP^1$  for every $u\in U .$\end{itemize}  Then there is a meromorphic map $\tau:Y\dasharrow Z$ such that $\tau\circ p=q,$  i.e., the following diagram commutes:\begin{equation}\label{l11} 
\begin{aligned}
&  &                        & X                         &  &        &  \notag \\
&   &  \swarrow_p     &                           &    \searrow^q & &\notag \\ 
&  Y  &    & \overset {\tau}{\dasharrow}      &      &  Z              &  \end{aligned},  \end{equation}
\end{Lemma}
\begin{proof}

Let $\Phi:X\to Y\times Z$ be defined by $\Phi(x)=(p(x),q(x)).$ 
 The image $T=\Phi(X)$ is an irreducible compact analytic  subspace of  $Y\times Z$  (see, e.g. \cite[ Theorem 2,   Chapter VII]{Narasimhan}).  
We denote by $pr_Y$ and $pr_Z$ the natural projections of $T$ on the first and second factor, respectively. Both projections are evidently surjective.
The set $$  T_1
=\{(y,z)\in T \ \  | \dim (\Phi^{-1}(y,z))>0\}$$  
is an   analytic subset  of $ T\subset Y\times  Z$
 (\cite{Re}, \cite[Theorem 3.6, p.137]{Fi}).
Its projections $T_Y=pr_Y(T_1)\subset Y$, and $T_Z=pr_Z(T_1)\subset Z$   to  the first and the second factor   are  analytic subsets of $Y$  and $Z,$  respectively,
(\cite{Re}, \cite [Theorem 2, Chapter  VII]{Narasimhan}).

  If $T_Y\ne Y$ then  $V:=(Y\setminus T_Y)\cap U$ is an analytical  Zariski open dense subset  of $Y.$   For each   $y\in V$  we have $p^{-1}(y)\sim\BP^1$ and $\dim (q(p^{-1}(y)))>0.$  Thus 
 the pair   $p: X\to Y, \ q:X\to Z$  would provide a covering family for $Z,$  which is impossible, since $Z$ is  not  uniruled.   Thus $T_Y=Y.$

   Take  $u\in U.$ Since  $T_Y=Y$  there is  $z\in Z$ such that $(u,z)\in  T$    and   $\dim   (\Phi^{-1}(u,z))\ge 1. $ 
       Moreover, $$\Phi^{-1}(u,z)=
   \{x \ \mid \  p(x)=u, q(x)=z\}
      \subset P_u\subset X.$$ Since $ P_u\sim\BP^1$ and $\dim   (\Phi^{-1}(u,z))\ge 1,$  we have $ P_u=\Phi^{-1}(u,z).$ Hence, $q\mid_{P_u}=z$ for every $u\in U$ and some $z\in Z$   and there is only one $z\in Z$ such that $(u,z)\in  T.$  
  Thus, \begin{enumerate}\item $T$ is an irreducible connected subset of $Y\times Z;$
  \item $\dim (T)=\dim (Y);$
  \item  for every $u\in U$ there is only one $z\in Z$ such that $(u,z)\in  T.$  
\end{enumerate} It  follows that $T$ is the  graph of a meromorphic map  that we denote as $\tau.$ 
 \end{proof}  
 
\begin{Remark}\label{Graf}  The fact that $q$ contracts every fiber of $p$  over an analytical  Zariski open non-empty subset of $Y$  is proven in    \cite [Proposition 6.2]{Graf}.   \end{Remark}

\begin{Lemma}\label{tau} 
 Let  $(X,p_X,Y)$ and $(W,p_W,Y)$  be    two  rational  bundles  over a non-uniruled (compact connected) manifold $Y.$
 Let 
 $ f:X\to W$ be a surjective meromorphic map. 

Then there exists a meromorphic  map $\tau(f):Y\to Y$  that may be included into the following commutative diagram.

\begin{equation}\label{diagram1001}
\begin{CD}
X @>{f}>>W\\
@V p_X VV @Vp_W VV \\
Y @>{\tau(f)}>> Y
\end{CD}.
\end{equation}
In addition, if $f$ is holomorphic, so is $\tau(f)$.
 \end{Lemma}

\begin{proof}   Let $a: \tilde X \to X$ be such a modification of $X$ that the following diagram is commutative 
\begin{equation}\label{l44}
\begin{aligned}
&  &                        & \tilde X                         &  &        &  \notag \\
&   &  \swarrow_{a}     &                           &    \searrow^b & &\notag \\ 
&  X  &    & \overset {f}{\dasharrow}      &      &  W             &  \end{aligned},  \end{equation}
where 
$b: \tilde X \to W$ is a holomorphic map (it always exists, \cite [Theorem 1.9]{Pe}).

 Consider  the holomorphic maps $\tilde p_X:=p_X\circ a:\tilde X\to Y$ and 
 $\tilde f:=p_W\circ b:\tilde X\to Y.$ We apply \lemref{mrc} to $\tilde X, Y=Z$ and  $\tilde p_X:\tilde X\to Y,$  $\tilde f:=\tilde X\to Y,$ and obtain the needed map $\tau(f)\in\Bim(Y)$ that may be included into the following commutative diagram
\begin{equation}\label{l5} 
 \begin{aligned}
&  &                        & \tilde X                         &  &        &  \notag \\
&   &  \swarrow_{a}     &                           &    \searrow^b & &\notag \\ 
&  X  &    & \overset {f}{\rightarrow}      &      &  W              & \notag \\ 
&\downarrow_{pr_X}&&&&\downarrow_{pr_W}\notag \\ 
&  Y  & &\overset {\tau(f)}{\rightarrow}      &      &  Y              &
 \end{aligned}.  \end{equation}
If $f$ is holomorphic, one may take $\tilde X=X$  and $U=Y$  (in the notation of \lemref{mrc}). Thus, 
$\tau(f)$ will be defined at every point of $Y.$
\end{proof}

\begin{Corollary}\label{tauh} For a  rational   bundle $(X,p,Y)$ over a non-uniruled (complex connected compact) manifold $Y$ there are  natural  group homomorphisms  $\tau:\Aut (X)\to \Aut(Y)$ and $\tilde {\tau}:\Bim (X)\to \Bim(Y)$ such that $$p\circ f=\tau(f)\circ p, \ p\circ f=\tilde \tau(f)\circ p$$ for every $f\in\Aut (X)$  or 
$f\in\Bim (X),$    respectively.\end{Corollary}

\begin{Remark}\label{tau-meromorphic}
 If $Y$ is K\"ahler  non-uniruled, then  
the restriction group homomorphism $$\tau \bigm |_{\Aut_0(X)}:\Aut_0(X)\to\tau(\Aut_0(X))$$  
is 
a holomorphic homomorphism of complex Lie groups and  $\tau(\Aut_0(X))$ is a closed complex Lie subgroup of $\Aut(Y)$  (A. Fujiki,  
 \cite[Lemma 2.4, 3, Theorem 5.5 and Lemma 4.6]{Fu78}).
\end{Remark}

  Further on we will use heavily the following classical theorems.
  
\begin{Theorem}\label{Remmert}[  Remmert-Stein Theorem](   see, e.g., \cite [Theorem of Remmert -Stein, ChapterVII]{Narasimhan}) 
  Let $X$ be a complex space and  $Y$ an  analytic subset of $X,$  $A$ an analytic subset of $X\setminus Y.$ Suppose that there is an integer $p>0$ such that $\dim  (Y)\le p-1,$ while  $\dim_a (A)\ge p$ for any $a\in A.$  ($\dim  (Y)\le -1$ means that $Y=\emptyset.$)  Then the closure $\ov A$ of $A$ in $X$ is an analytic set in $X.$
\end{Theorem}
  
\begin{Theorem}\label{Riemann}[Second  Riemann   removable singularity  theorem]  (\cite[ Chapter 2, Appendix]{Fi})  Assume that $X$ is a complex manifold and $A\subset X$    is an analytic subset such that $$\codim  _x (A)\ge 2 \text{ for every}   \ x\in X.$$  Then any holomorphic function $f:X\setminus A\to \BC$ has a unique holomorphic extension  $\tilde f:X\to\BC.$
\end{Theorem}
 
 \begin{Theorem}\label{Levi}[Levi   continuation theorem]    (\cite{Levi}, see also \cite[Chapter VII, Theorem 4]{Narasimhan} or \cite [Section 4.8]{Fi}) 
  Let $X$ be a normal complex space and $Y$ an analytic subset of $X$ such that for any $a\in X$ we have $\dim_a  (Y)\le \dim_a (X)-2.$ Then any meromorphic function on $X\setminus Y$ has an extension to a meromorphic function on $X.$
\end{Theorem}

 \begin{Remark}\label{affine} 
   It follows from the second Riemann Theorem that  a  holomorphic map  from  $f:X\setminus \Sigma  \to Z$ where $X$ is a complex manifold, $\Sigma$ an analytic  subset of codimension at least 2, and $Z\subset\BC^N$ an affine   complex set, 
  may be extended to a holomorphic map $\tilde f: X\to Z.$

   Indeed,  let   $z_1,\dots,z_N$  be coordinates in $\BC^N.$  The map $f$ consists of  $N$ holomorphic functions $z_i(x), i=1,\dots, N$  defined on 
$X\setminus \Sigma.$   By \thmref{Riemann} the functions $z_i$ may be extended to holomorphic functions $\tilde z_i$  defined on $X. $  
Since $Z$ is a closed subset of $\BC^N, $   we have $\tilde f(x)=(\tilde z_1(y),\dots,\tilde z_N(x))\in Z$  for every $x\in X.$

 This fact is a particular case of the  Extension Theorem of A. Andreotti and W. Stoll  (\cite{AS69}. Recall the a subset $ M\subset  X $ of a complex   space  $X$ is {\sl thin} if  in a neighborhood of every point  $m\in M$  it is contained in an analytic subset of codimension 1. 
\begin{Theorem}\label{AS}[Andreotti-Stoll Theorem] Let $\tau:A\to Y$ be a holomorphic map of the open subset A of a normal complex space $X$ into a Stein space $Y.$  let $M:=X\setminus A$ be a thin set. If $M$ has topological codimension  $\ge 3, $  then $\tau$ may be extended to a holomorphic map   of $X$ into $Y.$\end{Theorem}

\end{Remark}

We use this fact to prove the next 

 \begin{Lemma}\label{3.8}  Let $(X,p,Y)$ and $(Z,q,Y)$  be two $\BP^1-$bundles over a connected complex manifold $Y.$ Let $\Sigma\subset Y$ be an analytic subset of codimension  at least 2, $U=Y\setminus\Sigma,  \ V_X=p^{-1}(U),  V_Z=q^{-1}(U).$   Let 
 $f:X\to Y$ be a meromorphic map such that $q\circ f=p$  and   induced map
  $f:V_X\to  V_Z$ is  an isomorphism. 
Then $ f :X\to Z$  is a biholomorphic  isomorphism. \end{Lemma}

\begin{proof}  By construction, for every $u\in U$ the map $f$ induces an isomorphism $f\mid_{P_y}:P_y\to Q_y,$ where   $P_y=p^{-1}(y), Q_y=q^{-1}(y).$
Consider a point $s\in \Sigma $ and its open neighborhood  $U_s$ such that there are isomorphisms $\psi_X:p^{-1}(U_s)\to U_s\times\BP^1,$ 
$\psi_Z:q^{-1}(U_s)\to U_s\times\BP^1 $ compatible with projection maps $p$ and $q, $   respectively. Then for every $y\in U_s\cap U$   we have an element of $\PSL(2,\BC)$  representing $f\mid_{P_y}:P_y\to Q_y,$
which is an automorphism  of $\BP^1.$  Thus we have a holomorphic map $ U_s\cap U\to \PSL(2,\BC).$ Since the last one is an affine set,  the map  extends to a holomorphic map $U_s\to\PSL(2,\BC). $   Hence, we have an  extension of $f$ 
to an  isomorphism  $\tilde f_s: p^{-1}(U_s)\to  q^{-1}(U_s)$ that coincides with $f$ in $V_X\cap p^{-1}(U_s),$  hence everywhere.  
\end{proof}

\begin{Lemma}\label{3.10} Let $(X,p,Y)$ and $(Z,q,Y)$  be two $\BP^1-$bundles over a compact connected complex manifold $Y$ with $\dim (Y)=n.$ Let $\Sigma\subset Y$ be an analytic subset of codimension  at least 2, $U=Y\setminus\Sigma,  \ V_X=p^{-1}(U),  V_Z=q^{-1}(U).$   Let   $f:V_X\to  V_Z$ is  a  meromorphic  map such that $q\circ f=p.$   Then there exist a meromorphic map $\tilde f:X\to Y$ such that   $\tilde f\mid_U=f$ and $q\circ \tilde f=p.$ \end{Lemma}

  For  K\"ahler   manifold $Y$ this Lemma follows from the following general   Theorem of Y.-T. Siu (\cite{Siu}).

\begin{Theorem}\label{siu}  [Siu extension Theorem]   Let   $X$ be a complex manifold, $A$ be a subvariety of codimension $\ge  1$  in $X,$ and 
$G$ be an open subset of $X$ which intersect every branch of $A$ of codimension 1. If $M$ is a  compact K\"ahler manifold, then every meromorphic map $f$  from $(X-A)\cup G$  may be extended to   a meromorphic map from $X$ to $M.$\end{Theorem}

At this stage we do not require that $Y$ (and, a fortiori, $Z$) is K\"ahler, but we use the fact that $X, \ Z $  are $\BP^1-$bundles.

\begin{proof}(Proof of \lemref{3.10}).  Consider  a fiber product 
$$  W=X\times_{Y}Z=\{(x,z)\in X\times Z \ | \ p(x)=q(z)\}\subset X\times Z$$
and its subsets: $$\Gamma_f=\{(x,z)\in V_X\times V_Z \ |\ p(x)=q(z), \ z\in f(x)\}\subset W,$$
$$\tilde \Sigma=\{(x,z)\in X\times Z \ | \ p(x)=q(z)\in\Sigma\}\subset W.$$
 By construction  $\dim(\tilde \Sigma)\le n,$
 $\dim(\Gamma_f)=\dim (X)=n+1.$ Thus, according to the   Remmert-Stein  Theorem (\thmref{Remmert}) 
 the closure of $\ov\Gamma_f$ of $\Gamma_f$  in $W$ is an analytic subset in $W.$   Let $U_1\subset U$ be an open subset such that $f$ is defined at every point of $V_1:=p^{-1}(U_1).$ We have 
 \begin{itemize}\item  $\ov\Gamma_f$  is an irreducible  (since $\Gamma_f,$ being the graph of a meromorphic map is irreducible)   analytic subset of $X\times Z;$
\item $  \dim (\ov\Gamma_f)=\dim X;$
\item for every $v\in V_1$ there is unique $z\in Z$ such  that $(v,z)\in \ov\Gamma_f.$
\item the natural projection $\tau:\ov\Gamma_f\to X$ is proper, since both sets are compact.\end{itemize}
It follows that $\ov\Gamma_f$ is a graph of a meromomorphic map $\tilde f:X\to Z$  (see  \cite[page 75]{AS71}).\end{proof}

  We will use also the following

\begin{Lemma} \label {3.9}  Assume that  $Y$ is   a compact connected complex manifold,   $\Sigma\subset Y$ is an analytic subset of codimension  at least 2, $U=Y\setminus\Sigma.$  Let $(\LL,\pi, Y)$ be a holomorphic line bundle over $Y$  such that $\LL |_U$ is trivial. Then $\LL$ is trivial.\end{Lemma}

\begin{proof}  Indeed,  $V:=\pi^{-1}(U)\sim U\times \BC_z, $   thus $z=F(v)$ is a holomorphic function on $V.$   The set $\tilde \Sigma:=\pi^{-1}( \Sigma)$ has codimension 
at least two in $\LL.$  
 By  the Second  Riemann   removable singularity  theorem   (\thmref{Riemann}),  $F$ may be extended to a holomorphic function $\ov F$ on $\LL.$
   Thus we have a holomorphic map $\Phi:\LL\to Y\times \BC_z,  \ \   x\in \LL\to (p(x), \ov F(x)),$ that is an isomorphism outside $\tilde \Sigma.$    Let $S$  is 
the set of all points in $\LL$
where the differential $d\Phi$  of $\Phi$ does not have the maximal rank.  The sets  $S$ and $\tilde S=p(S)$ are analytic subsets of $\LL$
and  $ Y,$ respectively  (see, for instance,
\cite [Theorem 2, Chapter  VII]{Narasimhan},
 \cite[Therem  1.22]{PR}, \cite{Re}).  Moreover, $\codim(\tilde S)=1 $ (\cite{Ra}).   But $\tilde S\subset \Sigma,$  hence $\tilde S=\emptyset. $
It follows that $\Phi $ is an isomorphism.\end{proof}

\section {Non-trivial rational bundles}\label{conics} 

In this section we consider  non-trivial $\BP^1-$ bundles over a non-uniruled base.  It appears that the fact that $X\not\approx Y\times \BP^1$   imposes  the significant restrictions on the structure of the groups $\Aut(X)$ and $\Bim(X).$
We will start with projective case. 
\begin {Definition} \label{conicbundle}
       A regular surjective map $f: X\to Y$ of  smooth irreducible  projective complex  varieties is a {\it conic bundle} over $Y$ 
 if there is a Zariski-open  dense   subset $U\subset Y$ such  that   the fiber $f^{-1}(y)\sim\BP^1$ for all $y\in U$. \end{Definition}

 The { \it generic fiber} of $f$ is an irreducible  smooth projective curve $\mathcal{X}_f$ over the field $K:=\BC(Y)$  such that its field of rational functions $K(\mathcal{X}_f)$ coincides with $\BC(X).$
(The genus of $\mathcal{X}_f$ is $0$.) 
\begin{Theorem} \label{conicbundlethm} (\cite {BZ17})
Let $X$ be a conic bundle over a non-uniruled  smooth irreducible projective variety $Y$ with $\dim(Y)\ge 2$.  If $X$ is not birational to $Y\times \BP^1$ then $\Bir(X)$ is strongly Jordan.
\end{Theorem}

Let us sketch the  proof  of \thmref{conicbundlethm}.

 Let   $f: X\to Y$ be a conic bundle and assume that    $Y$ is  non-uniruled.  According to \corref{tauh}.
 every $\phi\in  \Bir(X)$  is fiberwise:  there is a homomorphism $\tilde\tau: \Bir(X)\to\Bir(Y)$ such that $\tilde\tau(\phi)\circ f=f\circ\phi:$ 

$$\begin{aligned}
&  X     & {} &\overset{\phi}\longrightarrow       & {} &     X  & \notag \\
&f\downarrow  &  {}    & {}                     & {} & \downarrow  f&\notag\\
 &  Y &{}      &\overset {\tilde\tau(\phi)} \longrightarrow   & {} & Y      &\notag\end{aligned}
.$$

 It follows that there is 
an  exact sequence of groups: 
\begin{equation}\label{ex100}
0 \to \Bir_{\mathbb  C(Y)}(\mathcal{X}_f) \to \Bir(X) \to \Bir(Y);\end{equation}

  Since $Y$ is non-uniruled the group $\Bir(Y)$ is strongly Jordan thanks to \thmref{PS14}  
(see also \cite[Cor. 3.8 and its proof]{BZ17}).

 Let us compute { $\Bir_{K}(\mathcal{X}_f)$}.  We have \begin{itemize}
\item[1.] $\Bir_{K}( \mathcal{X}_f)=\Aut( \mathcal{X}_f)$ since 
 $\dim(\mathcal{X}_f)=1.$

\item[2.]      Since $X\not\approx Y\times\BP^1$  
the genus $0$ curve $\mathcal{X}_f$ has {\bf no}
 $K$-points    and therefore 
there exists a ternary quadratic form
$$q(T)=a_1 T_1^2 +a_2 T_2^2 +a_3 T_3^2$$
over $K$ such that 

--- all $a_i$ are nonzero elements of $K;$ 

--- $q(T)=0$ if and only if  $ T=(0,0,0))$ (this means  that $q$ is  {\sl anisotropic}); 

--- $ \mathcal{X}_f$  is biregular  over $K$ to the plane projective quadric
$$\mathbf{X}_q:=\{(T_1:T_2:T_3)\mid q(T)=0\}\subset \BP^2_K.$$

\item[3.]       $K$ is a field of characteristic zero  that contains all roots of unity.

\end{itemize}

  Now we consider a quadric, i.e., a hypersurface in a    projective space  defined by one irreducible quadratic equation
over  $K.$ 
It is anisotropic  if it has no point defined over $K.$   In   \cite{BZ17}   proven was  the following 

\begin{Theorem}\label{quadrics} (\cite {BZ17}) 
Suppose that $K$ is a field of characteristic zero that contains all roots of unity,  $d \ge 3$ an odd integer,  $V$ a $d$-dimensional $K$-vector space and let $q: V \to K$ be a  quadratic form such that $q(v) \ne 0$ for all nonzero $v \in V$.
Let us consider  the projective quadric $X_{q}\subset \mathbb{P}(V)$  defined by the equation $q=0$, which is a smooth projective irreducible  $(d-2)$-dimensional variety over $K$.
Let $\Aut(X_{q})$ be the group of biregular automorphisms of $X_{q}$.
Let $G$ be a finite subgroup in $\Aut(X_{q})$. Then $G$ is commutative, all its non-identity elements have order $2$ and the order of $G$ divides $2^{d-1}$.\end{Theorem}

 Thus 
 if
  $G$ is a nontrivial finite subgroup of $ \Aut( \mathcal{X}_f)$ 
 then either $G\cong \BZ/2\BZ$ or  $G\cong (\BZ/2\BZ)^2$.

 Now applying \remarkref{exactsequences}(4)  we get from 
 \eqqref{ex100}  that $\Bir(X)$ is Jordan.

\begin{Remark}\label{pointless} Actually  in \thmref{conicbundlethm}  the
 variety $X$ is considered as a pointless  ($X(K)=\emptyset$)  rational  curve defined over a field $K,$    where field  $K$ contains all roots of unity.
 The ``pointless surfaces''  were studied by C. Shramov and V. Vologodsky in \cite{ShV}, \cite{ShVb}.\end{Remark}

For complex compact manifolds  the absence of a point in generic fiber has to be reformulated in  terms of sections.

  Let $(X,p,Y)$  be  a  rational bundle  over  a compact complex   connected non-uniruled   manifold $Y$ 
 (see  \defnref{rationalbundles1}),  i.e.,   \begin{itemize}
\item  $X, Y $ are compact connected manifolds;
\item $Y$ is non-uniruled;
\item  $p:X\to Y$  is a surjective holomorphic map; 
\item  $p^{-1}(U)$ is a   holomorphic     locally trivial fiber bundle over  a dense analytical Zariski open subset $U\subset Y$ with  fiber  $\BP^1$  and 
with the  corresponding projection map $p:p^{-1}(U)\to U.$ \end{itemize}  

According to \lemref{tau},   every map $f\in\Bim(X) $ maps the general fiber of $p$ to a fiber of $p.$  Let $$\Aut (X)_p=\{f\in\Aut(X) \ | \ \tau(f)=id\}, \  
\Bim (X)_p=\{f\in\Bim(X) \ | \ \tilde \tau(f)=id\},$$
be the kernels of $\tau$ and $ \tilde \tau,$ respectively.

 Then we have the following short exact sequences 
\begin{equation}\label{exseq1}
0\to\Aut(X)_p\to\Aut(X)\overset{\tau}{\to} \Aut(Y),\end {equation}
\begin{equation}\label{exseq2}
0\to\Bim(X)_p\to\Bim(X)\overset{\tilde \tau}{\to} \Bim(Y).
\end {equation}

\begin{Definition}\label{almost}  Let $(X,p,Y)$  be  an equidimensional rational  bundle  over  a compact complex   connected non-uniruled   manifold $Y.$  We will call an irreducible analytic subspace   $D$ of $X$ {\it almost section} if the intersection number $(D, F)$   of $D$ with a fiber   $F=p^{-1}(y), \ y\in Y$  is 1.\end{Definition}

\begin{Remark}\label{imagesection}
For $f\in \Bim(X)_p$ let $\tilde S_f$ be the  indeterminacy locus of $f$  that is an analytic subspace of $X$ of codimension at least 2 (\cite[page 369]{Re}). 
  Let   $S_f={p(\tilde S_f)},$ which is an analytic subset of $Y$ (\cite{Re}, \cite [Theorem 2, Chapter  VII]{Narasimhan}).  Since the dimension of a fiber of $p$ is one,   $Y\setminus S_f$ is an analytical   Zariski  open dence subset  $U$ of $Y.$  Hence the restriction $f  \mid_{P_y}$  of $f$ onto the  fiber   $P_y=p^{-1}(y)$ of  $p$
over a general point $y\in Y$ belongs to $\Aut(P_y).$ Thus $f$ induces an automorphism    of $V=p^{-1}(U)$ onto itself.

 Let $D$ be an almost section of $X$.  \begin{itemize}

\item[(1)]

 Let $a: \tilde X \to X$ be such a modification of $X$ that the following diagram is commutative 
\begin{equation}\label{l4}
\begin{aligned}
&  &                        & \tilde X                         &  &        &  \notag \\
&   &  \swarrow_{a}     &                           &    \searrow^b & &\notag \\ 
&  X  &    & \overset {f}{\dasharrow}      &      &  X             &  \end{aligned},  \end{equation}
where 
$b: \tilde X \to X$ is a holomorphic map (it always exists, \cite [Theorem 1.9]{Pe}).
 Then $f(D)=ba^{-1}(D)$ is an analytic subset   (\cite{Re},\cite[Theorem 3.6]{Fi})   that is a union of finite number of irreducible components $D_1,\dots, D_n.$

\item[(2)]
 We may assume   (maybe after  shrinking  $U$) that  $D$ meets every fiber $P_y, y\in U$ at precisely one point. Thus $f(D)$ meets $P_y, y\in U$ at precisely one point as well.

\item[(3)]   It follows from (2) that precisely one of  irreducible components  of $f(D),$ say, $D_1,$  meets  a  fiber $P_y, y\in U.$
The intersection $D_1\cap P_y, y\in U$  consists of precisely one point.\end{itemize}
 
Thus $D_1$ is an almost section.  It follows that the image of an almost section under $f\in\Bim(X)_p$ contains precisely one almost section.
 In particular, $f$ cannot contract an almost section.\end{Remark}

Similarly, if  $\Phi:X\to Z$   is  a bimeromorphic map of a $\BP^1-$bundle $(X,p,Y)$ to  a $\BP^1-$bundle $(Z,q,Y)$ such that $q\circ\Phi=p,$  then the image of an almost section contains an almost section.

  The following results  were proved by Yu.  Prokhorov and C. Shramov in more general setting, we formulate below its application for the case of  $\BP^1-$ bundles. 
\begin{Theorem}\label{PS100}  Let $(X,p,Y)$  be  a  $ \BP^1$-bundle  over  a compact complex   connected non-uniruled   manifold $Y.$ Let   $P_y=p^{-1}(y)$   be a   fiber of  p
over a general point $y\in Y.$ 

Then
\begin{itemize}\item[1.]   Every countable union of finite subgroups of $\Bim(X)_p$ may be embedded into $\Bim(P_y)$
(\cite [Lemma 4.1]{PS19-2}).
\item[2.]   If $X$ is K\"ahler, then   $\Bim(X)_p$ is Jordan ((\cite [Corollary  4.3]{{PS19-2}}
\item [3.] 
 If there exists an  almost section $D$ on $X$ then   $X\sim \BP(\EE)$
for some rank two holomorphic vector bundle $\EE$ on $Y.$   \cite[Lemma 3.5]{kostya}.
\item [4.]  Assume that no almost section exists on $X.$ 
Assume that $\Bim(Y)$ is  strongly Jordan.  Then $\Bim(X)$ is Jordan
\cite[ Corollary 5.8]{kostya}.
\item[5.]  If there exist $f\in\Bim(X)_p$ of finite order $d>2$ then there exist at least two   distinct almost sections on $X.$
If $f$ is biholomorphic, the almost sections may be chosen to be disjoint. 
\cite[Lemma 4.1]{kostya}
\end{itemize}\end{Theorem}

Let us add to this the following 
\begin{Lemma}\label{almost1}  In the Notation of \thmref{PS100}, assume that   there exists precisely one   almost section on $X.$  Then
 if $ \Bim(Y)$
is Jordan, so is $\Bim(X).$
\end{Lemma}

\begin{proof}  
Assume that  $D$ is  the only almost section.  Let $f\in\Bim(X)_p, \  f\ne id.$  The set $f(D)$ contains an irreducible component $D_1$ that is 
an almost section   (see  \remarkref{imagesection}). 
Therefore $D=D_1$ and $D$ is contained in the set $\mathrm{Fix}(f)$  of fixed points of $f. $  Let $V\subset Y$ be an analytical Zariski open   dense  subset such that the   restriction   $f_v$ of $f$ onto  the fiber $P_v$ is a  non-identical automorphism of $P_v$
for all $v\in V$.
  Since $f_v$ has at most two fixed points, we have:

---either $\Fx(f)\cap P_v=D\cap P_v$ contains one point, and $f_v$ has infinite order; 

---or $(\mathrm{\Fx}(f)\cap P_v)\setminus(D\cap P_v)$ contains a  point for the general $v\in V$   and $\mathrm{Fix}(f)$ contains an almost section distinct from $D,$  which is impossible. 

Thus every element $f\in\Bim(X)_p$ different from $id$ has infinite order.   Therefore  $G\cap\Bim(X)_p=\{id\}$  for every finite group $G\subset \Bim(X)$  and $\tilde \tau: G\to \Bim(Y) $ is a group  embedding. Hence, the Jordan index $J_{\Bim(X)}\le J_{\Bim(Y)}.$\end{proof}

The opposite case, when the $\BP^1-$bundle has many almost sections, is  when  $X\cong Y\times\BP^1.$ It will be considered in the next chapter.
 
\chapter{$\BP^1-$bundles over complex  tori}\label{directproducts}

In this section we deal with  $\BP^1-$bundles of a  special type, namely $(\ov \LL , p, T),$  where $\LL$ is a holomorphic line bundle over a complex torus  $T$ and     
   $      \ov \LL=\BP(\LL\oplus\mathbf 1_T). $
 Most examples of compact complex connected manifolds with a  non-Jordan group $\Bim(X)$  (at least for dimensions greater than 3)  are $\BP^1-$ bundles of this type.   Manifolds of this type were studied by one of the authors in papers \cite{Zar14}(projective case)
  and \cite{Zar19} (non-algebraic case).   The goal of this chapter is to present a unified approach for both situations. It is based on a construction motivated by symplectic geometry and  inspired by an  algebraic approach to theta functions   developed by \cite{Mum66}.  
The chapter starts with  symplectic   constructions, then the theta groups follow,   then we arrive to description of certain subgroups of $\Bim(\ov \LL).$

\section{Symplectic Group Theory}

This section contains elementary but useful facts about Jordan properties of central extensions of
 commutative groups by $\BC^{*}$.

Traditionally, some groups are written in the multiplicative form, and some in the additive one. We hope that no confusion will arise.

\begin{Definition} 
A {\sl symplectic pair} is a pair  $(A,e)$ that consists of  a commutative group $A$ and an alternating bilinear pairing
$$e: A \times A \to \BC^{*}.$$
Here {\sl alternating} means that
$$e(a,a)=1 \ \forall a \in A.$$
The bilinearity means that
$$e(a_1+a_2,b)=e(a_1,b) e(a_2,b), \ e(a,b_1+b_2)=e(a,b_1)e(a,b_2)
\ \forall a, a_1,a_2, b,b_1,b_2 \in A.$$
These properties imply that for all $a,b \in A$
$$1=e(a+b,a+b)=e(a,a)e(a,b)e(b,a)e(b,b)=e(a,b)e(b,a),$$
i.e.,
$$e(a,b)=e(b,a)^{-1} \ \forall a,b \in A.$$

\end{Definition}
As usual, $e$ gives rise to  the group homomorphism
\begin{equation}
\label{dualE}
\Psi_e: A \to \Hom(A,\BC^{*}), \ b \mapsto \{\Psi_e(b): A \to \BC^{*}, \ a \mapsto e(a,b)\}.
\end{equation}
A subgroup $B$ of $A$ is called {\sl isotropic} with respect to $e$ if
$$e(B,B)=\{1\}.$$

We define the  {\sl kernel} of $e$ as
$$\ker(e):=\{a \in A \mid e(a,A)=\{1\}\}=\ker(\Psi_e),$$
which is a subgroup of $A$ that is  isotropic with respect to $e$.

We say that $e$ is {\sl nondegenerate} if $\ker(e)=\{0\}$, i.e.,
$$\Psi_e: A \to \Hom(A,\BC^{*})$$ is an {\sl injective} homomorphism.
If $e$ is nondegenerate then we call $(A,e)$ a {\sl nondegenerate symplectic pair}.

\begin{Example}
\label{bHd}
Let $d$ be a positive integer, $\mathbf{S}_d=\left(\frac{1}{d}\BZ/\BZ\right)^2\cong (\BZ/d\BZ)^2$,
$$\mathbf{e}_d: \mathbf{S}_d\times \mathbf{S}_d \to \BC^{*}, \
(a_1+\BZ,b_1+\BZ), (a_2+\BZ,b_2+\BZ) \mapsto \exp(2\mathbf{\pi}\mathbf{i}d(a_1b_2-a_2 b_1)).$$
Then $(\mathbf{S}_d,\mathbf{e}_d)$  is a {\sl nondegenerate symplectic pair}.
\end{Example}

\begin{Remark}
\label{sumSym}
Let $(A_1e_1)$ and $(A_2,e_2)$  be nondegenerate symplectic pairs. Let us consider the bilinear alternating form
$$e_1 e_2: (A_1\oplus A_2) \times  (A_1\oplus A_2)\to \BC^{*}, $$
$$(a_1,a_2), (b_1,b_2) \mapsto e_1(a_1,b_1)\cdot e_2(a_2,b_2).$$
Then $(A_1\oplus A_2, e_1 e_2)$ is a {\sl nondegenerate symplectic pair}.
\end{Remark}

\begin{Remark}
\label{subPair}
If $(A,e)$  is a symplectic pair and $B$ is a subgroup of $A$ then $(B,e\mid_B)$ is also a symplectic pair.
Here $e\mid_B$ is the restriction of $e$ to $B \times B$.

\end{Remark}

\begin{Remark}
\label{barA}
\begin{itemize}
\item[(i)]
Each symplectic pair $(A,e)$ gives rise to a nondegenerate symplectic pair
$(\bar{A},\bar{e})$ where
\begin{equation}
\bar{A}=A/\ker(e), \ \bar{e}(a \ \ker(e), b \ \ker(e))=e(a,b) \ \forall a,b \in A.
\end{equation}
\item[(ii)]
Clearly, a subgroup $B$ of $A$ is isotropic with respect to $e$ if and only if its image $\bar{B}$ in $\bar{A}$
is isotropic with respect to $\bar{e}$. In particular, $B$ is isotropic if and only if $B+\ker(e)$ is isotropic.
\item[(iii)]
Let $B$ be a subgroup of $A$.  One may restate a property of $B$  to be isotropic with respect to $e$ as follows.
The  composition of $\Psi_e: A \to \Hom(A,\BC^{*})$ with the {\sl restriction map}
$\Hom(A,\BC^{*})\to \Hom(B,\BC^{*})$ is the group homomorphism
\begin{equation}
\label{AtoB}
A \overset{\Psi_e}{\to} \Hom(A, \BC^{*}) \to \Hom(B, \BC^{*}).
\end{equation}
 Clearly, the kernel $B^{\bot}$ of this homomorphism (which is the {\sl orthogonal complement} of $B$ in $A$
with respect to $e$) contains $B$ if and only if
$B$ is isotropic.
\item[(iv)]
Suppose that $B$ coincides with $B^{\bot}$.
This means that if $a \in A \setminus B$ then $e(B,a) \ne \{1\}$. In other words,
$B$ is a {\sl maximal} isotropic subgroup of $A$ with respect to $e$.

Conversely, suppose that $B$ is a {\sl maximal} isotropic subgroup of $A$ with respect to $e$.
Since $B$ is isotropic,
$$B\subset B^{\bot}\subset A,  \ e(B^{\bot},B)=\{1\}.$$
If $B^{\bot} \ne B$ then there is $a \in B^{\bot}\setminus B$ such that $e(a,B)=\{1\}$.
This implies that the subgroup $B_1$ of $A$ generated by $B$ and $a$ is isotropic, which contradicts the maximality of $B$.

It follows that $B=B^{\bot}$  if and only if $B$
is a {\sl maximal isotropic} subgroup of $A$.
\end{itemize}
\end{Remark}

\begin{Remark}
\label{finiteA}
Suppose that $A$ is finite. Then the finite groups $A$ and $\Hom(A,\BC^{*})$ are isomorphic (non-canonically);
in particular, they have the same order. It follows that in the case of finite $A$ the pairing $e$ is nondegenerate
if and only if $\Psi_e$ is a group {\sl isomorphism}.
\end{Remark}

\begin{Lemma}[Useful Lemma]
\label{Useful Lemma}
Let $(A,e)$ be a symplectic pair such that $A/\ker(e)$ is a finite group. If $B$ is a maximal isotropic subgroup of $A$
then the index $[A:B]$ equals $\sqrt{\#(A/\ker(e))}$. In particular, if $e$ is nondegenerate then 
$$[A:B]=\sqrt{\#(A)}=\#(B).$$
\end{Lemma}

\begin{proof}[Proof of Usefull Lemma]
In light of Remark \ref{barA}, $B$ contains $\ker(e)$ and therefore it suffices to prove the desired result for nondegenerate
$(\bar{A},\bar{e})$ (instead of $(A,e)$). In other words, without loss of generality, we may assume that
$\ker(e)=\{0\}$, i.e.,
$A=\bar{A}$  is {\sl finite}  and $e=\bar{e}$ is {\sl nondegenerate}.  

Since $\BC^{*}$ is a divisible group, every group homomorphism $B \to \BC^{*}$ extends to a group homomorphism
$A \to \BC^{*}$. This means that the restriction map
$\Hom(A, \BC^{*}) \to \Hom(B, \BC^{*})$ is surjective. Since $A$ is finite, the nondegeneracy of $e$ means (in light of Remark \ref{finiteA})
that
$\Hom(A, \BC^{*})=\Psi_e(A)$. On the other hand, the maximality of $B$ means that the kernel of the {\sl surjective} composition
$$A \overset{\Psi_e}{\cong} \Hom(A, \BC^{*}) \twoheadrightarrow \Hom(B, \BC^{*})$$
coincides with $B$ (see Remark \ref{barA})  and therefore there is an {\sl injective} group homomorphism
$$A/B \hookrightarrow \Hom(B, \BC^{*}),$$
which is also surjective and therefore is an isomorphism.
This implies that 
$$\#(A/B)=\#\big(\Hom(B, \BC^{*})\big)=\#(B),$$ which ends the proof
if we take into account that $\#(A/B)=\#(A)/\#(B)$.
\end{proof}

\begin{Remark}
\label{liftF}
Suppose that $\ker(e)$  is either {\sl finite} or {\sl divisible}. Then every finite subgroup $\bar{B}$ of $\bar{A}$ is the image of a 
finite subgroup $B\subset A$ under $A \twoheadrightarrow \bar{A}$. Indeed, if $\ker(e)$ is finite then one may take as $B$ the preimage
of $\bar{B}$ in $A$.  If $\ker(e)$ is divisible then it is a direct summand of $A$, i.e., $A$ splits into a direct sum
$A=\ker(e)\oplus A^{\prime}$ and the map $A \to \bar{A}$ induces an isomorphism $A^{\prime} \cong \bar{A}$. Now one may take
as $B$ the (isomorphic) preimage of $\bar{B}$ in $A^{\prime}$.
\end{Remark}

\begin{Definition}
A symplectic pair $(A,e)$ is called {\sl almost isotropic} if there exists a positive integer $D$
that enjoys the following property.

Each finite subgroup $\mathcal{B}$ of $A$ contains an {\sl isotropic} (with respect to $e$) subgroup $\mathcal{A}$
such that the index $[\mathcal{B}: \mathcal{A}] \le D$. Such a smallest $D$ is called the {\sl isotropy defect} of $(A,e)$ and denoted
by $D_{A,e}$.
\end{Definition}

\begin{Example}
If $e \equiv 1$ then every subgroup is isotropic and therefore $D_{A,e}=1$.

\end{Example}

\begin{Remark}
\label{isoD}
Suppose that $\ker(e)$ is either finite or divisible.
\begin{itemize}
\item[(i)]
It follows from Remarks \ref{liftF} and \ref{barA} that $(A,e)$ is almost isotropic if and only if
$(\bar{A},\bar{e})$ is almost isotropic. In addition, if this is the case then 
\begin{equation}
\label{DeA}
D_{A,e}=D_{\bar{A},\bar{e}}.
\end{equation}
Indeed, let $\mathcal{A}$ be a finite subgroup of $A$ and $B$ an isotropic subgroup of largest possible order in $\mathcal{A}$.
In particular,  $B$ is a maximal isotropic subgroup of $\mathcal{A}$. Since
$B_1=B+\left(\mathcal{A}\cap \ker(e)\right)$ is an isotropic subgroup of $\mathcal{A}$ that contains $B$, the maximality of $B$ implies that
$B_1=B$, i.e., $B \supset \mathcal{A}\cap \ker(e)$. This implies that the index $(\mathcal{A}:B)$ equals the index
$[\bar{\mathcal{A}}:\bar{B}]$ where the subgroups $\bar{\mathcal{A}}$ and $\bar{B}$ are the images in $\bar{A}$ of 
$\mathcal{A}$ and $B$ respectively. Taking into account that $\bar{B}$ is an isotropic (with respect to $\bar{e}$) subgroup of
finite group $\bar{\mathcal{A}}\subset \bar{A}$, we conclude that
$$D_{A,e} \ge D_{\bar{A},\bar{e}}.$$

Conversely, suppose that $\bar{B}$ is an isotropic (with respect to $\bar{e}$) subgroup of maximal order in a
finite group $\bar{\mathcal{A}}\subset \bar{A}$. As above, this implies that 
$\bar{B}$ is a maximal isotropic subgroup of $\bar{\mathcal{A}}$.
By Remark \ref{liftF},  $A$ contains a finite subgroup $\mathcal{A}$,
whose image in $\bar{A}$ coincides with  $\bar{\mathcal{A}}$. Let $B$ the preimage of $\bar{B}$ in $\mathcal{A}$.
Then $B$ is isotropic with respect to $e$ and the index $[\mathcal{A}:B]$ coincides with the index
$[\bar{\mathcal{A}}:\bar{B}]$. This implies that
$$D_{A,e} \le D_{\bar{A},\bar{e}},$$
which ends the proof.

\item[(ii)]
Assume additionally that $\bar{A}$ is {\sl finite}. Applying Lemma \ref{Useful Lemma} to  subgroups of $\bar{A}$ and using \eqref{DeA}, we conclude that
\begin{equation}
\label{DeAf}
D_{A,e}=D_{\bar{A},\bar{e}}=\sqrt{\#(\bar{A})}.
\end{equation}

\end{itemize}

\end{Remark}

\begin{Definition}
A {\sl theta group} attached to a symplectic pair $(A,e)$ is a group $G$ that sits in a short exact sequence
\begin{equation}
\label{ThetaExact}
1 \to \BC^{*} \overset{i}{\to} G \overset{j}{\to} A \to 0
\end{equation}
that enjoys the following properties.

The image of $\BC^{*}$ is a {\sl central} subgroup of $G$, and the alternating {\sl commutator} pairing
$$A \times A \to \BC^{*},   \ j(g_1), j(g_2)\mapsto i^{-1}\left(g_1 g_2 g_1^{-1}g_2^{-1}\right)\in \BC^{*} \ \forall g_1,g_2 \in G$$ 

attached to  exact sequence \eqref{ThetaExact} coincides with $e$.
\end{Definition}

\begin{Remark}
Every {\sl central} extension $G$ of a commutative group $A$ by $\BC^{*}$ gives rise to the  symplectic pair $(A,e)$ where
$e(a_1,a_2)\in \BC^{*}$ is the commutator of preimages of $a_1,a_2$ in $G$ (for all $a_1,a_2 \in A$). This makes  $G$ a theta group attached to $(A,e)$. 

\end{Remark}

\begin{Remark}
\label{thetaBasic}
\begin{itemize}
\item[(i)]
Clearly, an element $g$ of the theta group $G$ lies in the center of $G$  if and only if 
$$e(j(g), j(h))=1 \ \forall h \in G.$$
Since $j(G)=A$, the element $g$ is central if and only if $j(g)\in \ker(e)$. This implies that the center of $G$ coincides
with $j^{-1}\left(\ker(e)\right)$.

\item[(ii)]
Clearly,  a subgroup $H$ of $G$ is commutative if and only if its image $j(H)\subset A$ is an isotropic subgroup of $A$ with respect to $e$.
\end{itemize}
\end{Remark}

\begin{Remark}
\label{subTheta}
Let $G$ be a theta group that sits in the short exact sequence \eqref{ThetaExact}. If $B$ is a subgroup of $A$
then obviously the preimage $j^{-1}(B)$ is a theta group attached to the symplectic pair $(B,e\mid_B)$.

\end{Remark}

\begin{Lemma}
\label{liftingT}
Let $B$ be a finite subgroup of $A$. Then there exists a finite subgroup $\tilde{B}$ of the theta group  $G$ such that $j(\tilde{B})=B$.
\end{Lemma}

\begin{proof} 
In what follows, we identify $\BC^{*}$ with its image in $G$ and view it as a certain central subgroup of $G$.
Let $d$ be the {\sl exponent} of $B$. 

Let us consider  the finite multiplicative subgroups $\mu_d$ and $\mu_{d^2}$
of all $d$th roots of unity and $d^2$th roots of unity,  respectively, in $\BC^{*}$. We have
$$\mu_d \subset \mu_{d^2}\subset \BC^{*} \subset G;$$
in addition,
\begin{equation}
\label{eB}
e(B,B) \subset e(B,A)\subset \mu_d.
\end{equation}

For every $b \in B$ choose its {\sl lifting} $\tilde{b} \in G$ such that
\begin{equation}
\label{liftBd}
\tilde{b}^d=1, \ \tilde{b}^{-1}=\widetilde{b^{-1}}\ \forall b \in B;
\end{equation}
this is possible, since $\BC^{*}$ is a central divisible subgroup of $\BC^{*}$. Indeed,
let $\tilde{b}_1\in G$ be any lifting of $b$ to $G$, i.e., $j(\tilde{b}_1)=b$. Then
$$z_1:=\tilde{b}_1^d \in \ker(j)=\BC^{*}.$$
Let us choose any
$$z=\sqrt[d]{z_1}\in \BC^{*}$$ and put $\tilde{b}=z^{-1}\tilde{b}_1\in G$. We have
$$j(\tilde{b})=j(z^{-1})+j(\tilde{b}_1)=0+b=b; \ \tilde{b}^d=(z^{-1})^d\tilde{b}_1^d=z_1^{-1} z_1=1.$$

  Let us put
$$\tilde{B}:=\{\gamma \tilde{b}\mid \gamma \in \mu_{d^2}, b \in B\}\subset G.$$
Clearly,  $\tilde{B}$ is finite, $j(\tilde{B})=B$, and
$$ 1 \in \mu_{d^2}\subset \tilde{B}=\tilde{B}^{-1}:=\{u^{-1}\mid u \in \tilde{B}\}$$
(the latter equality follows from the invariance of the central subgroup
 $\mu_{d^2}$ and the subset $\{\tilde{b}\mid b\in B\}$ under the map $u \mapsto u^{-1}$).

So, in order to prove that $\tilde{B}$ is a subgroup of $G$, it suffices to check that $\tilde{B}$ is closed under multiplication in $G$.
Let $b_1,b_2\in B$ and $b_3=b_1+b_2\in B$. We need to compare $\tilde{b}_1\tilde{b}_2$ and $\tilde{b}_3$ in $G$. Clearly, there is $\gamma\in \BC^{*}$ such that
$$ \tilde{b}_3=\gamma \tilde{b}_1\tilde{b}_2.$$

Notice that
$$\tilde{b}_1^d=\tilde{b}_2^d=\tilde{b}_3^d=1 \in \BC^{*} \subset G.$$
On the other hand, in  light of \eqref{eB},
$$\gamma_0:=\tilde{b}_1\tilde{b}_2 \tilde{b}_1^{-1}\tilde{b}_2^{-1}=e(b_1,b_2)\in \mu_d\subset \BC^{*} \subset G.$$

It follows that the images of $\tilde{b}_1$ and $\tilde{b}_2$ in the quotient $G/\mu_d$ do commute and therefore the
image of  $\tilde{b}_1\tilde{b}_2$ in $G/\mu_d$ has order that divides $d$. This means that
$$\left(\tilde{b}_1\tilde{b}_2\right)^d\in \mu_d$$
and therefore
$$\left(\tilde{b}_1\tilde{b}_2\right)^{d^2}=1.$$
It follows that
$$1=\tilde{b}_3^{d^2}=\left(\gamma  \cdot \tilde{b}_1\tilde{b}_2\right)^{d^2}=\gamma^{d^2}\left(\tilde{b}_1\tilde{b}_2\right)^{d^2}
=\gamma^{d^2}\cdot 1=\gamma^{d^2}.$$

This implies that $\gamma^{d^2}=1$, i.e., $\gamma \in \mu_{d^2}$ and therefore
$$\tilde{b}_1\tilde{b}_2=\gamma^{-1}\tilde{b}_3 \in \tilde{B}.$$
This ends the proof.
\end{proof}

\begin{Theorem}
\label{JordanTheta}
Let $(A,e)$ be a symplectic pair. Suppose that $\bar{A}=A/\ker(e)$ is finite.
Assume also that either $\ker(e)$ is divisible or $A$ is finite.
Let $G$ be a theta group attached to $(A,e)$. 

Then $G$ is a Jordan group and its Jordan index equals $\sqrt{\#(\bar{A})}$.
\end{Theorem}

\begin{proof}
Assume that $G$ sits in a short exact sequence \eqref{ThetaExact}. We may view $\BC^{*}$ as  a central subgroup of $G$.
Let $\tilde{\mathcal{A}}$ be a finite subgroup of $G$ and $\tilde{B}$ a commutative subgroup of {\sl maximal order} in 
$\tilde{\mathcal{A}}$. Then $\tilde{B}$ contains the intersection $\tilde{\mathcal{A}}\cap \BC^{*}$ and therefore the index
$[\tilde{\mathcal{A}}:\tilde{B}]$ coincides with the index $[j(\tilde{\mathcal{A}}):j(\tilde{B})]$. The commutativeness of 
$\tilde{B}$ means that $j(\tilde{B})$ is an {\sl isotropic} subgroup in $j(\tilde{\mathcal{A}})$. This implies that
$$J_G \ge D_{A,e}.$$

Conversely, let $\mathcal{A}$ be a finite subgroup of $A$ and $B$ is an isotropic subgroup of maximal order in $\mathcal{A}$.
By Lemma \ref{liftingT}, there is a finite subgroup  $\tilde{\mathcal{A}}$ of $G$ such that
$$j(\tilde{\mathcal{A}})= \mathcal{A}.$$ 
Let $\tilde{B}$ be the preimage of $B$ in $\tilde{\mathcal{A}}$.  Then
$$j(\tilde{B})=B, \ [\mathcal{A}:B]=[\tilde{\mathcal{A}}:\tilde{B}].$$
By  Remark \ref{thetaBasic}(ii), $\tilde{B}$ is commutative, because its image $B$ is isotropic. The equality of indices implies that
$$J_G \le D_{A,e},$$
which, combined with the previous opposite inequality, implies that $J_G =D_{A,e}$. Now the explicit formula for $J_G$
follows from Remark \ref{isoD}.

\end{proof}

\section{Symplectic linear algebra}
\label{LinSymp}
In this section we construct theta groups that arise from  (non necessarily nondegenerate)  alternating bilinear form on integral lattices. 

\begin{Definition}

\begin{itemize}
\item[(i)]
An {\sl admissible triple} is a triple $(V,E,\Pi)$ that consists of a
 nonzero  real vector space $V$ of finite  positive even dimension $2g$, an alternating $\BR$-bilinear form
$$E: V \times V \to \BR$$  on $V$, and
a discrete lattice $\Pi$ of rank $2g$ in $V$ such that $E(\Pi,\Pi)\subset \BZ$.
Let us put
$$\Pi_E^{\bot}:=\{v\in V \mid E(v,l)\in \BZ \ \forall l \in \Pi.$$
By  definition, $\Pi_E^{\bot}$ is a closed real Lie subgroup of $V$ that contains $\Pi$ as a discrete subgroup.
\item[(ii)]
A symplectic pair attached to the admissible triple $(V,E,\Pi)$ is a pair $(K_{E,\Pi}, e_{E})$ where 
$K_{E,\Pi}:=\Pi_E^{\bot}/\Pi$ and
 the bilinear pairing $e_E$ is defined as follows.
$$e_E: \Pi_E^{\bot}/\Pi \times \Pi_E^{\bot}/\Pi \to \BC^{*}, \ (v_1+\Pi,v_2+\Pi) \mapsto \exp(2\mathbf{\pi} \mathbf{i} E(v_1,v_2)).$$
\end{itemize}
\end{Definition}

\begin{Definition}
Recall that a subgroup $C$ of a commutative group $D$ is called {\sl saturated} if it enjoys the following equivalent properties.

\begin{itemize}
\item
There are no elements of finite order in the quotient $D/C$ except $0$.
\item
If $x$ is an element of $D$ such that there is a positive integer $m$ with $mx \in C$ then $x \in C$.
\end{itemize}
\end{Definition}

Our goal is to find the isotropy index of $(K_{E,\Pi}, e_E)$. In order to do that, let us consider the {\sl kernel} of $E$, i.e., the subset
$$\ker(E)=\{v \in V\mid E(v,V)=\{0\}\}\subset V.$$
Clearly, $\ker(E)$ is a real even-dimensional (recall that $E$ is alternating) vector subspace of $V$ containing $\Pi_E^{\bot}$.  Let us put
$$\Pi_0:=\Pi \bigcap \ker(E) \subset \ker(E).$$
Clearly, $\Pi_0$ is a {\sl saturated} subgroup of $\Pi$.
The integrality property of $E$ implies that the natural homomorphism of real vector spaces
$$\Pi_0\otimes\BR \to \ker(E), \  l_0\otimes \lambda \mapsto \lambda \cdot l_0 \ \forall l_0\in \Pi_0, \lambda\in \BR$$
is an isomorphism. In particular, the following conditions are equivalent.

\begin{itemize}
\item[(a)]
$E$ is nondegenerate, i.e., $\ker(E)=\{0\}$.
\item[(b)]
$\Pi_0=\{0\}$.
\end{itemize}

Let us consider several cases.

\begin{itemize}
\item[{\bf Case I}]
If $E\equiv 0$ then 
$$\Pi_E^{\bot}=V, K_{E,\Pi}=\Pi_E^{\bot}/\Pi=V/\Pi, e_E \equiv 1,$$
$\ker(e_E)=K_{E,\Pi}$ is divisible and $K_{E,\Pi}/\ker(e)=\{0\}$ is finite.
By Remark \ref{isoD},  the isotropy  defect  $D_{
K_{E,\Pi}, e_E}=1$.
\item[{\bf Case II}]
Suppose that $E$ is a {\sl nondegenerate} form.  
Let $\{s_1, \dots, s_{2g}\}$ be any basis of the $\BZ$-module $\Pi$. Clearly, it is also a basis of the $\BR$-vector space space $V$. Let 
$$\tilde{E}=\Big(E(s_j,s_k)\Big)\in \mathrm{Mat}_{2g}(\BZ)$$ 
be the  $2g \times 2g$ {\sl skew-symmetric} matrix of $E$ with respect to this basis with {\sl integer} entries.  Let
$\det(\tilde{E})$ and $\mathrm{Pf}(\tilde{E})$ be the {\sl determinant} of $\tilde{E}$  and the {\sl pfaffian} of $\tilde{E}$ respectively. Then
$$\det(\tilde{E})\in \BZ,  \mathrm{Pf}(\tilde{E})\in \BZ; \ 0 \ne \det(\tilde{E})=\mathrm{Pf}(\tilde{E})^2.$$
In particular,  $\det(\tilde{E})$ is a {\sl positive integer}. Clearly, $\det(\tilde{E})$ does not depend on the choice of a basis of $\Pi$
and therefore $|\mathrm{Pf}(\tilde{E})|$ does not depend on this choice as well. That is why we denote
$\det(\tilde{E})$ by $\det(E,\Pi)$ and $|\mathrm{Pf}(\tilde{E})|$ by $|\mathrm{Pf}(E,\Pi)|$.

We claim that {\sl $\Pi_E^{\bot}/\Pi$ is  finite, the form
$$e_E: \Pi_E^{\bot}/\Pi \times \Pi_E^{\bot}/\Pi \to \BC^{*}$$
is {\sl nondegenerate} and its isotropy defect is  $|\mathrm{Pf}(E,\Pi)|$.}

Indeed, there is a basis $\{f_1,h_1, \dots, f_g,h_g\}$ of $\Pi$ such that
$$E(f_j, h_k)=-E(h_k,f_j)=0 \ \forall j \ne k \ (1 \le j,k \le g)$$
(\cite[Ch. XV, Ex. 17 on p. 598]{Lang}).
Let us put 
$$d_j=E(f_j,h_j)\in \BZ \ \forall j=1, \dots, g.$$

The nondegeneracy of $E$ means that all $d_j \ne 0$. Replacing if necessary, $h_j$ by $-h_j$, we may and will assume that
all $d_j>0$.  If $\tilde{E}$ is the matrix of $E$ with respect to this basis then the {\sl pfaffian}  $\mathrm{Pf}(\tilde{E})$ 
of $\tilde{E}$ is $\pm \prod_{j=1}^g d_j$ and therefore
$$|\mathrm{Pf}(E,\Pi)|=\prod_{j=1}^g d_j.$$
We claim that
\begin{equation}
\label{botEE}
\Pi_E^{\bot}=\oplus_{j=1}^g \frac{1}{d_j}\left(\BZ \cdot f_j\oplus \BZ\cdot h_j\right).
\end{equation}
Indeed, a vector 
$$v=\left(\sum_{j=1}^g \lambda_jf_j\right)+ \left(\sum_{j=1}^g \mu_j h_j\right) \text{ with all } \ \lambda_j,\mu_j\in\BR$$
lies in $\Pi_E^{\bot}$ if and only if  
$$\BZ\ni E(f_j,v)=d_j \mu_j, \ \BZ \ni(h_j,v)=-d_j \lambda_j \ \forall j,$$
 which is obviously equivalent to \eqref{botEE}.

It follows from \eqref{botEE} that
\begin{equation}
\label{KEpi}
\Pi_E^{\bot}/\Pi=\oplus_{j=1}^g \frac{1}{d_j}\left(\BZ \cdot f_j\oplus \BZ\cdot h_j\right)/\left(\BZ \cdot f_j\oplus \BZ\cdot h_j\right).
\end{equation}

Clearly, different summands of $\Pi_E^{\bot}/L$  are mutually orthogonal with respect to $e_E$ while the
restriction of $e_E$ to each 
$$\frac{1}{d_j}\left(\BZ \cdot f_j\oplus \BZ\cdot h_j\right)/\left(\BZ \cdot f_j\oplus \BZ\cdot h_j\right)
$$ is isomorphic
to $(\mathbf{S}_{d_j},\mathbf{e}_{d_j})$. In particular, this restriction is a nondegenerate symplectic pair.
This implies that the direct sum $(\Pi_E^{\bot}/\Pi,e_E)$ is also a nondegenerate symplectic pair.
On the other hand, clearly, 
$$\Pi_E^{\bot}/\Pi\cong \oplus_{j=1}^g \left(\frac{1}{d_j}\BZ/\BZ\right)^2.$$
This implies that
$$\#(\Pi_E^{\bot}/\Pi)\cong \prod_{j=1}^g d_j^2, \ \sqrt{\#(\Pi_E^{\bot}/\Pi)}=\prod_{j=1}^g d_j=|\mathrm{Pf}(E,\Pi)|.$$
This implies that $(K_{E,\Pi},e_E)$ is almost isotropic and its {\sl isotropy defect} is $|\mathrm{Pf}(E,\Pi)|$.
\item[{\bf Case IIbis}]
We keep the notation and assumptions of {\bf Case II}. Let us consider the form $nE$ where $n$ is a positive integer.
Then 
$$\Pi_{nE}^{\bot}=\frac{1}{n} \Pi_E^{\bot}= \oplus_{j=1}^g \frac{1}{nd_j}\left(\BZ \cdot f_j\oplus \BZ\cdot h_j\right),$$
$$\Pi_{nE}^{\bot}/\Pi\cong \oplus_{j=1}^g \left(\frac{1}{nd_j}\BZ/\BZ\right)^2,$$
$$\#(\Pi_{nE}^{\bot}/\Pi)=\prod_{j=1}^g (nd_j)^2, \ \sqrt{\#(\Pi_E^{\bot}/\Pi}=n^g\prod_{j=1}^g d_j=n^g\cdot| \mathrm{Pf}(E,\Pi)|.$$
Hence, the corresponding isotropy index
$$D_{K_{nE,\Pi}, e_{nE}}=n^g\cdot |\mathrm{Pf}(E,\Pi)|$$
for all positive integers $n$.

\item[{\bf Case III}]
Now let us consider the case of degenerate  nonzero $E$, i.e., the case when
$$\{0\} \ne \Pi_0 \ne \Pi.$$
Clearly, $\Pi_0$ is a free abelian group of a certain positive even rank $2g_0<2g$.
Since $\Pi_0$ is a saturated subgroup of $\Pi$, it is a {\sl direct summand} of $\Pi$, i.e., there is a (nonzero saturated) subgroup $\Pi_1$ in $\Pi$
that is a free abelian group of rank $2g-2g_0$ and such that
$$\Pi=\Pi_0 \oplus \Pi_1.$$
In other words, there is a basis $\{u_1, \dots, u_{2g_0}; v_1, \dots, v_{2g-2g_0}\}$ of the $\BZ$-module $\Pi$ such that
$\{u_1, \dots, u_{2g_0}\}$ is a basis of $\Pi_0$ and $\{v_1, \dots, v_{2g-2g_0}\}$ is a basis of $\Pi_1$.  Let us consider
the real vector subspaces
$$V_0:=\sum_{j=1}^{2g_0}\BR u_j\subset V, \ V_1:=\sum_{k=1}^{2g_1}\BR v_k\subset V.$$
Clearly,
$$V=V_0\oplus V_1; \ \Pi_0=V_0\cap \Pi, \ \Pi_1=V_1\cap \Pi;$$
in addition, $V_0=\ker(E)$, the subspaces $V_0$ and $V_1$ are mutually orthogonal with respect to $E$ and the restriction of $E$ to $V_1$
$$E_1: V_1 \times V_1 \to \BR, \ u,v \mapsto E(u,v)$$
is a {\sl nondegenerate} alternating bilinear form.
It is also clear that
$$E_1(\Pi_1,\Pi_1)=E(\Pi_1,\Pi_1)\subset E(\Pi,\Pi)\subset \BZ.$$
On the other hand, the  restriction of $E$ to $V_0$, which we denote by $E_0$,  is identically $0$.
This implies that (as the symplectic pair)
$$(K_{E,\Pi},e_E)=(K_{E_0,\Pi_0}, e_{E_0})\oplus (K_{E_1,\Pi_1}, e_{E_1}).$$
By {\bf Case I} applied to $(V_0,E_0,\Pi_0)$,  the group $K_{E_0,\Pi_0}=V_0/\Pi_0$ is {\sl divisible} as a quotient of a complex vector space,  and $ e_{E_0}\equiv 1$.
By {\bf Case II} applied to $(V_1,E_1,\Pi_1)$, the group
$K_{E_1,\Pi_1}$ is  {\sl finite}  of order $|\mathrm{Pf}(E,\Pi)|^2$
and the pairing
$$e_{E_1}: K_{E_1,\Pi_1} \times K_{E_1,\Pi_1} \to \BC^{*}$$
is {\sl nondegenerate}. This implies that
$\ker(e_E)=K_{E_0,\Pi_0}$ and therefore $\ker(e_E)$ is divisible and
$$K_{E,\Pi}/\ker(e_E)=K_{E_1,\Pi_1}$$
is a finite group. This implies that $(K_{E,\Pi},e_E)$ is almost isotropic and its isotropy defect, by \thmref{JordanTheta},
\begin{equation}
\label{dgD}
d_{K_{E,\Pi},e_E}=\sqrt{\#(K_{E,\Pi}/\ker(e_E))}=\sqrt{\#(K_{E_1,\Pi_1})}=|\mathrm{Pf}(E_1,\Pi_1)|.
\end{equation}
\item[{\bf Case IIIbis}]
We keep the notation and assumptions of {\bf Case III}. Let $$M: V \times V \to \BR$$ be an alternating bilineaer form that enjoys the following
properties. 

\begin{enumerate}
\item
$M(\Pi,\Pi)\subset \BZ$.
\item
$\ker(E)\subset \ker(M)$.
\end{enumerate}
If $n$ is an integer then we write $\mathbf{M}(n)$ for the alternating bilnear form $nE+M$ on $V$. Clearly,
$$\mathbf{M}(n)(\Pi,\Pi)\subset n E(\Pi,\Pi)+M(\Pi,\Pi)\subset n \BZ+\BZ=\BZ.$$

\begin{Lemma}\label{claim}  There exists a degree $(g-g_0)$ polynomial $\mathcal{P}(t)\in \BZ[t]$  with integer coefficients and leading
coefficient $|\mathrm{Pf}(E_1,\Pi_1)|$ that enjoys the following property.

{\sl For all but finitely many positive integers $n$ the symplectic pair
 $(K_{\mathbf{M}(n),\Pi}, e_{\mathbf{M}(n)})$ is almost isotropic and its
isotropy defect}
\begin{equation}
\label{pencilD}
 D_{K_{\mathbf{M}(n),\Pi}, e_{\mathbf{M}(n)}}=\mathcal{P}(n). 
\end{equation}\end{Lemma}

\begin{proof}
Indeed, let $M_1:V_1\times V_1 \to \BR$ be the restriction of $M$ to $V_1\times V_1$. 
 Let $\tilde{E}_1$ and $\tilde{M}_1$ be the matrices of $E_1$ and$M_1$ with respect to the basis 
$\{f_1, \dots, f_{2g-2g_0}\}$ of $\Pi_1$. The nondegeneracy of $E_1$ implies that 
$\det(\tilde{E}_1) \ne 0$ and therefore
$$\det(n \tilde{E}_1+\tilde{M}_1)=\det(\tilde{E}_1) \det (n \mathrm{I}_{2g-2g_0}+\tilde{E}_1^{-1} \tilde{M}_1)$$
does {\sl not} vanish for all but finitely many integers $n$. (Hereafter $\mathrm{I}_{2g-2g_0}$ is the identity square matrix
of size $2g-2g_0$.) Taking into account that  $n \tilde{E}_1+\tilde{M}_1$ is the matrix of the restriction of $nE+M=\mathbf{M}(n)$, we obtain that
 for all but finitely many integers $n$
\begin{equation}
\label{conditionN}
\ker(\mathbf{M}(n))=\ker(nE+M)=\ker(E)=V_0.
\end{equation}
In what follows, we assume that $n$ is any integer that enjoys the property \eqref{conditionN}
(this assumption excludes only finitely many  integers $n$). Now we may apply results of {\bf Case III} to $\mathbf{M}(n)=nE+M$ (instead of $E$) and
get that $(K_{\mathbf{M}(n),\Pi},e_{\mathbf{M}(n)})$ is almost isotropic and its isotropy defect is
$$|\mathrm{Pf}(nE_1+M_1,\Pi_1)|=\sqrt{\det(nE_1+M_1,\Pi_1)}=
\sqrt{\det(\tilde{E}_1) \det (n \mathrm{I}_{2g-2g_0}+\tilde{E}_1^{-1} \tilde{M}_1)}=$$
$$|\mathrm{Pf}(E_1,\Pi_1)|
\sqrt{\det (n \mathrm{I}_{2g-2g_0}+\tilde{E}_1^{-1} \tilde{M}_1)}.$$
Clearly,  there is a polynomial $\mathcal{Q}(t) \in \BZ[t]$ with integer coefficients such that  for {\sl all} our $n$
$$\mathcal{Q}(n)=\mathrm{Pf}(n\tilde{E_1}+\tilde{M}_1).$$
This implies that
$$\mathcal{Q}(n)^2=\det(n\tilde{E}_1+\tilde{M}_1)=\det(\tilde{E}_1) \det (n \mathrm{I}_{2g-2g_0}+\tilde{E}_1^{-1} \tilde{M}_1).$$
It is also clear  that there exists a {\sl monic} degree $(2g-2g_0)$ polynomial $\mathcal{R}(t) \in \BQ[t]$ with rational coefficients such that
for all our $n$
$$\mathcal{R}(n)=\det (n \mathrm{I}_{2g-2g_0}+\tilde{E}_1^{-1} \tilde{M}_1).$$
This implies that
$$\mathcal{Q}(n)^2=\det(\tilde{E}_1) \mathcal{R}(n)=|\mathrm{Pf}(E_1,\Pi_1)|^2 \mathcal{R}(n).$$
Since $\mathcal{R}(t)$ is monic of degree $(2g-2g_0)$, we have $$\deg(\mathcal{Q})=(g-g_0)$$ and 
the leading coefficient of $\mathcal{Q}(t)$ is $\pm |\mathrm{Pf}(E_1,\Pi_1)|$.

Let   $\mathcal{P}(t)$ be the polynomial with posittive leading coefficient that coincides either with $\mathcal{Q}(t)$
or with $-\mathcal{Q}(t)$. Then $\mathcal{P}(t)$  is a degree $(g-g_0)$ polynomial with integer coefficients and leading coefficient
 $|\mathrm{Pf}(E_1,\Pi_1)|$ such that
$$\mathcal{P}(n)=\pm \mathrm{Pf}(n\tilde{E_1}+\tilde{M}_1).$$
Since the leading coefficient of $\mathcal{P}(t)$ is positive,  $\mathcal{P}(n)$ is positive for all but finitely many positive integers $n$. This implies that 
$$\mathcal{P}(n)=|\mathrm{Pf}(n\tilde{E_1}+\tilde{M}_1)|=
|\mathrm{Pf}(nE_1+M_1,\Pi_1)|$$
for all such $n$. This ends the proof.
\end{proof}

\end{itemize}

\begin{Theorem}
\label{JordanUnbounded}
Let $g$ be a positive integer, $V$ a $2g$-dimensional real vector space,
$(V,E,\Pi)$ and $(V,M,\Pi)$ are admissible triples such that
$$E \not\equiv 0, \ \ker(E) \subset \ker(M).$$
If $n$ is an integer then  we write $\mathbf{M}(n)$ for the alternating bilinear form
$nE+M$ on $V$.

Let $\mathcal{G}$ be a group that enjoys the following properties.

There are infinitely many positive integers $n$ such that $\mathcal{G}$ contains
a subgroup $G_n$ that is a theta group attached to $(K_{\mathbf{M}(n),\Pi}, e_{\mathbf{M}(n)})$.

Then  $\mathcal{G}$ is not Jordan.
\end{Theorem}

\begin{proof}
It suffices to check that the Jordan index of $G_n$ tends to infinity while $n$ tends to infinity.
But this assertion follows from results of {\bf Cases {\bf II}, {\bf III}, {\bf IIIbis}} of this section combined with
Theorem \ref{JordanTheta}.

\end{proof}

\section{Line bundles over tori  and theta groups}\label{Appel}

In this section we use results from previous two sections in order to  compute the Jordan index of certain  automorphism groups of holomorphic line bundles on complex tori.

Let $V$ be a complex vector space of finite positive dimension $g$,
$\Pi$ a discrete lattice of rank $2g$ in $V$,
$$H: V \times V \to \BC$$
an Hermitian form on $V$ such that its imaginary part
$$E:V \times V \to \BR,  \ (v_1,v_2) \mapsto \mathrm{Im}(H(v_1,v_2))$$
satisfies 
$$E(\Pi,\Pi) \subset \BZ.$$
One may view  $V$ as the $2g$-dimensional real vector space. Then $E$ becomes an alternating $\BR$-bilinear form on $V$ such that
$$E(\mathbf{i}v_1,\mathbf{i}v_2)=E(v_1,v_2) \ \forall v_1,v_2 \in V.$$

In addition,
$$H(v_1,v_2)=E(\mathbf{i}v_1,v_2) +\mathbf{i} E(v_1,v_2)  \ \forall v_1,v_2 \in V$$
(see \cite[Lemma 2.1.7]{CAV}).
This implies that
 $H$ and $E$ have the same kernels, i.e.,
$$\ker(H):=\{w\in V\mid H(w,V)=0\}=\{w\in V\mid E(w,V)=0\}=:\ker(E).$$

\begin{Definition}[see \cite{BL}, \cite{Kempf}]
\label{AHdefinition}
A pair $(H, \alpha)$ is called an {\sl Appel-Humbert data} (A.-H. data) on $(V,\Pi)$ if $H,E,\Pi$ are as above and
$\alpha$ is a map (``semicharacter")
$$\alpha: \Pi \to \mathbf{U}(1)=\{z \in \BC, |z|=1\}\subset \BC^{*}$$
such that
\begin{equation}
\label{alphaL}
\alpha(l_1+l_2)=(-1)^{E(l_1,l_2)}\alpha(l_1)\alpha(l_2) \ \forall l_1,l_2\in \Pi.
\end{equation}
In particular, if $l_1=l_2=0$ then $\alpha(0)=\alpha(0)^2$, i.e.,
$$\alpha(0)=1.$$
\end{Definition}
Notice that a classical theorem of Appel-Humbert (\cite[Theorem 1.5]{Kempf}, \cite[Theorem 21.1]{BL}) classifies holomorphic line bundles on the complex torus $V/\Pi$
in terms of A.-H. data.  

The construction of Section \ref{LinSymp} gives us the symplectic pair
$(K_{E,\Pi}, e_E)$.  The aim of this section is to constuct a certain theta group $\mathfrak{G}(H,\alpha)$ attached to this pair that corresponds to any A.-H. data $(H,\alpha)$.  
We define $\tilde{\mathfrak{G}}(H,V)$ as a certain group of biholomorphic automorphisms of $\mathcal{L}(H,\alpha)$. Here
$\LL(H,\alpha)$ is the {\sl total body} of the holomorphic line bundle $\mathcal{L}(H,\alpha)$ over $V/\Pi$ that corresponds to A.-H. data $(H,\alpha)$.

First, we start with a certain theta group $\tilde{\mathfrak{G}}(H,V)$ attached to the symplectic pair
$(V, \tilde{e}_E)$ where
$$\tilde{e}_E: V \times V \to \BC^{*}, \ (v_1,v_2) \mapsto \exp\left(2\mathbf{\pi}\mathbf{i}E(v_2,v_1)\right).$$
We define $\tilde{\mathfrak{G}}(H,V)$ as a certain group of holomorphic automorphisms of 
$$V_{\mathbb{L}}:=V \times \mathbb{L}$$ where $\mathbb{L}$
is a one-dimensional $\BC$-vector space. Namely,  $\tilde{\mathfrak{G}}(H,V)$ consists of automorphisms $\mathcal{B}_{H,u,\lambda}$
indexed by $u\in V,\lambda \in \BC^{*}$ that are defined as follows.
$$\mathcal{B}_{H,u,\lambda}:  (v,c) \mapsto \big(v+u, \lambda \exp(\mathbf{\pi}H(v,u)c\big) \ \forall v\in V, c \in \mathbb{L}.$$
One may easily check (see \cite[Sect. 2.1]{Zar19}) that the composition
\begin{equation}
\label{Ggroup}
\mathcal{B}_{H,u_1,\lambda_1} \circ\mathcal{B}_{H,u_2,\lambda_2}=\mathcal{B}_{H,u_1+u_2,\lambda_1 \lambda_2 \mu} \
\text{where} \
\mu=\exp(\mathbf{\pi} H(u_2,u_1))
\end{equation}
and the inverse
\begin{equation}
\label{inverseG}
\mathcal{B}_{H,u,\lambda}^{-1}=\mathcal{B}_{H,-u,\nu/\lambda} \ \text{where} \ \nu=\exp(-\mathbf{\pi}H(u,u)).
\end{equation}

This implies that  $\tilde{\mathfrak{G}}(H,V)$ is indeed a subgroup of the group of biholomorphic automorphisms of $V_{\mathbb{L}}$.
(Our  $\mathfrak{G}(H,\alpha)$  will be defined as a subquotient of $\tilde{\mathfrak{G}}(H,V)$.)
Notice that
 for all $\lambda \in \BC^{*}$ the automorphism 
$\mathcal{B}_{H,0,\lambda}$ sends every $(u,c)$ to $(u,\lambda c)$. This implies that the map
$$\mathrm{mult}:\BC^{*} \to \tilde{\mathfrak{G}}(H,V), \ \lambda \mapsto \mathcal{B}_{H,0,\lambda}$$
is an injective group homomorphism, whose image lies in the center of $\tilde{\mathfrak{G}}(H,V)$. This allows us to include $\tilde{\mathfrak{G}}(H,V)$ in a short exact sequence of groups
$$1 \to \BC^{*} \overset{\mathrm{mult}}{\to} \tilde{\mathfrak{G}}(H,V) \overset{\tilde{j}}{\to}  V \to 0$$
where $\tilde{j}$ sends $\mathcal{B}_{H,u,\lambda}$ to $u$. It follows from \eqref{Ggroup}  and \eqref{inverseG} 
(see also \cite[Sect. 2.1]{Zar19})
 that
\begin{equation}
\label{commutator}
\mathcal{B}_{H,u_1,\lambda_1} \circ\mathcal{B}_{H,u_2,\lambda_2}\circ \mathcal{B}_{H,u_1,\lambda_1}^{-1} 
\circ\mathcal{B}_{H,u_2,\lambda_2}^{-1}=\mathrm{mult}(\exp(2\mathbf{\pi}\mathbf{i}E(u_2,u_1))
=\mathrm{mult}(\tilde{e}_E(u_1,u_2)).
\end{equation}
This implies that $\tilde{\mathfrak{G}}(H,V)$ is a theta group attached to the symplectic pair $(V, \tilde{e}_E)$.

Let us consider the following subgroups of $\tilde{\mathfrak{G}}(H,V)$.

\begin{equation}
\label{subPii}
\tilde{\mathfrak{G}}(H,\Pi)=\tilde{j}^{-1}(\Pi)=\{\mathcal{B}_{H,u,\lambda}\mid \lambda\in \BC^{*}, u \in \Pi\};
\end{equation}
\begin{equation}
\label{subPi}
\tilde{\mathfrak{G}}(H, \Pi_E^{\bot})=\tilde{j}^{-1}( \Pi_E^{\bot})=\{\mathcal{B}_{H,u,\lambda}\mid \lambda\in \BC^{*}, u \in  \Pi_E^{\bot}\}.
\end{equation}

By Remark \ref{subTheta}, $\tilde{\mathfrak{G}}(H,\Pi)$ 
and  $\tilde{\mathfrak{G}}(H, \Pi_E^{\bot})$ are theta groups attached
to the symplectic pairs $(\Pi,\tilde{e}\mid_{\Pi})$ and $( \Pi_E^{\bot},\tilde{e}\mid_{ \Pi_E^{\bot}})$ respectively.
Since $\Pi \subset  \Pi_E^{\bot}$,  the group $\tilde{\mathfrak{G}}(H,\Pi)$ is a subgroup of $\tilde{\mathfrak{G}}(H, \Pi_E^{\bot})$.
It follows from \eqref{commutator} that $\tilde{\mathfrak{G}}(H,\Pi)$ is actually a {\sl central} subgroup of $\tilde{\mathfrak{G}}(H, \Pi_E^{\bot})$,
because
$$E(\Pi, \Pi_E^{\bot})=\{0\}.$$

We will define $\mathfrak{G}(H,\alpha)$ as a quotient of $\tilde{\mathfrak{G}}(H, \Pi_E^{\bot})$ by a certain central subgroup that depends on the  ``semicharacter" $\alpha$.  In order to define this subgroup, let us consider the {\sl discrete free} action of the group $\Pi$ on $V_{\mathbb{L}}$ by holomorphic
automorphisms defined as follows. An element $l$ of $\Pi$ acts as 
\begin{equation}
\label{AlVL}
\mathcal{A}_{H,\alpha,l}: V_{\mathbb{L}} \to V_{\mathbb{L}}, \ (v,c)
 \mapsto (v+l, c\alpha(l)\exp\big(\mathbf{\pi} H(v,l)+\mathbf{\pi}H(l,l)/2)\big) \ \forall v \in V, c \in \mathbb{L},
\end{equation}
i.e.,
\begin{equation}
\label{ABcalc}
\mathcal{A}_{H,\alpha,l}=\mult(\alpha(l))\mathcal{B}_{H,l,1}\in \tilde{\mathfrak{G}}(H,\Pi).
\end{equation}

Direct calculations that are based on \eqref{alphaL} show that 
$$\mathcal{A}_{H,\alpha,l_1}\mathcal{A}_{H,\alpha,l_2}=\mathcal{A}_{H,\alpha,l_1+l_2} \ \forall l_1,l_2 \in \Pi,$$
i.e., 
$$\mathbf{A}^{\Pi}: \Pi \to \tilde{\mathfrak{G}}(H,\Pi), \ l \mapsto \mathcal{A}_{H,\alpha,l}$$
is an {\sl injective} group homomorphism, whose image we denote by
$$\tilde{\Pi}=\tilde{\Pi}(H,\alpha):=\mathbf{A}^{\Pi}(\Pi)\subset  \tilde{\mathfrak{G}}(H,\Pi)\subset \tilde{\mathfrak{G}}(H, \Pi_E^{\bot}).$$
Notice that $\tilde{\Pi}$ meets $\mathrm{mult}(\BC^{*})$ precisely at the identity element of $\tilde{\mathfrak{G}}(H, \Pi_E^{\bot})$.
Notice that the  quotient
$V_{\mathbb L}/\tilde{\Pi}(H,\alpha)$ is precisely the total body $\LL(H,\alpha)$  of the holomorphic vector bundle 
$\mathcal{L}(H,\alpha)$ 
over $V/\Pi$
attached to the A.-H. data $(H,\alpha)$ where the structure map
$$p:\LL(H,\alpha)=V_{\mathbb{L}}/\tilde{\Pi}(H,\alpha) \to V/\Pi$$
is induced by the projection map
$$V_{\mathbb{L}} =V \times \mathbb{L} \to V$$
\cite[Ch. 2,  Sect. 2.2, p. 30]{CAV}.
Let us put 
\begin{equation}
\label{ThetaHalpha}
\mathfrak{G}(H,\alpha):=\tilde{\mathfrak{G}}(H, \Pi_E^{\bot})/\tilde{\Pi}(H,\alpha).
\end{equation}

The faithful action of $\tilde{\mathfrak{G}}(H, \Pi_E^{\bot})$ on $V_{\mathbb{L}}$ induces the {\sl faithful} action of
 $\mathfrak{G}(H,\alpha)$ on $\LL(H,\alpha)$. Under this action, each coset
 $$\mathcal{B}_{H,u,\lambda}\tilde{\Pi}\in \tilde{\mathfrak{G}}(H, \Pi_E^{\bot})/\tilde{\Pi}(H,\alpha)=\mathfrak{G}(H,\alpha)$$
maps {\sl $\BC$-linearly} and isomorphically the fiber of $p$ over $v+\Pi \in V/\Pi$ to the fiber over
 $(v+u)\Pi \in V/\Pi$ for any pair
$$ u+\Pi \in \Pi_E^{\bot}/\Pi\subset V/\Pi, \  \text{ and } \ v+\Pi \in V/\Pi,   \  \text{ and }  \ \lambda \in \BC^{*}.$$
In particular, $\mathrm{mult}(\lambda)\tilde{\Pi}$ acts as the automorphism $[\lambda]$ that leaves invariant each fiber of  $p:\LL(H,\alpha)\to V/\Pi$ 
and acts on this fiber (which is a one-dimensional $\BC$-vector space) as multiplication by $\lambda$
(for all $\lambda\in\BC^{*}$). Clearly, each $[\lambda]$ lies in the {\sl center} of $\mathfrak{G}(H,\alpha)$.

\begin{Lemma}
\label{thetaTheta}
The group  $\mathfrak{G}(H,\alpha)$ is a theta group attached to the symplectic pair $(K_{E,\Pi}, e_E)$.  
\end{Lemma}

\begin{proof}

Clearly,
$$\mathrm{[mult]}: \BC^{*} \to \mathfrak{G}(H,\alpha), \ \lambda \mapsto [\lambda]$$
is an {\sl injective} group homomorphism, whose image $\mathrm{[mult]}(\BC^{*})$ is a {\sl central} subgroup of $\mathfrak{G}(H,\alpha)$.
On the other hand, $\tilde{j}$ induces the surjective group homomorphism
$$j:\mathfrak{G}(H,\alpha) =\tilde{\mathfrak{G}}(H, \Pi_E^{\bot})/\tilde{\Pi}\twoheadrightarrow
\Pi_E^{\bot}/\Pi=K_{E,\Pi},$$
$$ \mathcal{B}_{H,u,\lambda}\tilde{\Pi} \mapsto u+\Pi \in \Pi_E^{\bot}/\Pi.$$
Clearly, the kernel of $j$ consists of all $\mathcal{B}_{H,0,\lambda}\tilde{\Pi}=\mathrm{[mult]}(\lambda)$, i.e., coincides
with $\mathrm{[mult]}(\BC^{*})$. Hence, $\mathfrak{G}(H,\alpha)$ sits in the short exact sequence
$$1\to \BC^{*} \overset{\mathrm{[mult]}}{\to} \mathfrak{G}(H,\alpha)  \overset{j}{\to} \Pi_E^{\bot}/\Pi \to 0.$$

It follows from \eqref{commutator} that $\mathfrak{G}(H,\alpha)$ is a theta group attached to  the symplectic pair $(K_{E,\Pi}, e_E)$.
\end{proof}

\begin{Remark}
\label{tensorAH}
It is well known \cite[Lemma 2.2.1]{CAV} that if $(H_1,\alpha_1)$ and $(H_2,\alpha_2)$ are A.H. data on  $(V,\Pi)$
then $(H_1+H_2,\alpha_1\alpha_2)$ is also an A.H. data on $(V,\Pi)$ and holomorphic vector bundles
$\mathcal{L}(H_1+H_2,\alpha_1\alpha_2)$ and $\mathcal{L}(H_1,\alpha_1)\otimes \mathcal{L}(H_2,\alpha_2)$ are canonically isomorphic.

\end{Remark}

\section{$\BP^1$-bundles bimeromorphic to the   direct product}\label{ttori}
In this section we prove  the non-Jordanness of the groups of bimeromorphic selfmaps of certain $\BP^1$-bundles over complex tori of positive algebraic dimension.

Let $V$ be a complex vector space of finite positive dimension $g$, $\Pi$ a discrete lattice of rank $2g$ in $V$ and
$T=V/\Pi$ the corresponding complex torus.  Recall that
$\mathbf{1}_T$ stands for the trivial holomorphic line bundle $T \times \BC$ over $T$.
If $x$ is point of $T$  then we write$\mathcal{L}_x$ for the fiber of a
holomorphic vector  bundle $\mathcal{L}$   over $T$, which is a one-dimensional complex vector space. 
We write $\bar{\mathcal{L}}$ for the projectivization $\BP(\mathcal{E})$ of the two-dimensional holomorphic
vector bundle $\mathcal{E}=\mathcal{L}\oplus \mathbf{1}_T$.  The fiber $\mathcal{E}_x$ of $\mathcal{E}$ over $x$
is the set of pairs $(s_x,c)$ where $s_x\in \mathcal{L}_x, c \in \BC$ and the fiber $\bar{\mathcal{L}}_x$ of $\bar{\mathcal{L}}$
over $x$ is the set of equivalence classes of $(s_x:c)$  where either $s_x \ne 0$ or $c\ne 0$ and the equivalence class
of  $(s_x:c)$ is the set of all
 $$(\mu s_x:\mu c), \ \mu \in \BC^{*}.$$

\begin{Lemma}
\label{ThetaEmbBP1}
Suppose that $\mathcal{L}=\mathcal{L}(H,\alpha)$ where $(H,\alpha)$ is an A.-H. data.
Then there is a natural group embedding
$$\mathfrak{G}(H,\alpha)\hookrightarrow \Aut(\overline{\mathcal{L}(H,\alpha)}).$$

\end{Lemma}

\begin{proof}

First, let us define the group embedding
\begin{equation}
\label{embedE}
\mathfrak{G}(H,\alpha)\hookrightarrow \Aut(\mathcal{L}(H,\alpha)\oplus\mathbf{1}_T)
\end{equation}
by the formula
\begin{equation}
\label{embedEx}
g:  \big(s_x,  (x,c) \big)\mapsto \big(g(s_x), (x+j(g),c)\big) \\
 \ \forall g \in \mathfrak{G}(H,\alpha),  x \in V/\Pi= T, c \in \BC,\ s_x \in \mathcal{L}_x\subset \mathcal{L}.
\end{equation}
In particular, $g$ induces an isomorphism of two-dimensional complex vector spaces between the fibers of
 $\mathcal{L}(H,\alpha)\oplus\mathbf{1}_T$ over $x$ and over $x+j(g)$. Since 
$\mathfrak{G}(H,\alpha)\to \Aut(\mathcal{L}(H,\alpha))$ is a group {\sl embedding}, we conclude that
if $j(g)=0$ then $g_x$
is  multiplication by a scalar if and only if $g$ is the identity element of $\mathfrak{G}(H,\alpha)$. This implies that
\eqref{embedE} and \eqref{embedEx} induce a group embedding
\begin{equation}
\label{embedP}
\mathfrak{G}(H,\alpha)\hookrightarrow \Aut(\BP(\mathcal{L}(H,\alpha)\oplus\mathbf{1}_T))=
\Aut(\overline{\mathcal{L}(H,\alpha)})
\end{equation}
such that each $g \in \mathfrak{G}(H,\alpha)$ sends every $(s_x:c)\in \mathcal{L}(H,\alpha)_x$ to $(g(s_x):c) \in \mathcal{L}(H,\alpha)_{x+j(g)}$.  This ends the proof.
\end{proof}

Let $\mathcal{L}$ be a holomorphic line bundle over the complex torus $T=V/\Pi$. Then 
$\mathcal{L} \cong \mathcal{L}(H,\alpha)$ for a certain (actually, precisely one) A.-H. data $H,\alpha)$ on $(V,\Pi)$  (\cite[Theorem 1.5]{Kempf}).
Let us denote by $\mathfrak{G}(\mathcal{L})$ the group $\mathfrak{G}(H,\alpha)$. By Lemma \ref{ThetaEmbBP1},
there exists a {\sl group embedding}
\begin{equation}
\label{embedPL}
\mathfrak{G}(\mathcal{L}) \hookrightarrow \Aut(\bar{\mathcal{L}}).
\end{equation}

\begin{Lemma}
\label{CP1Bim}
Let $\mathcal{L}$ and $\mathcal{N}$ be holomorphic line bundles over $T=V/\Pi$.
Assume that $\mathcal{L}$ admits a nonzero holomorphic section. Then  the compact complex 
manifolds $\bar{\mathcal{N}}$ and $\overline{\mathcal{L}^n\otimes \mathcal{N}}$ are bimeromorphic
for all positive integers $n$. In particular, for all such $n$ there is a {\sl group embedding}
\begin{equation}
\label{embedBim}
\mathfrak{G}(\mathcal{L}^n\otimes \mathcal{N}) \hookrightarrow \Bim(\bar{\mathcal{N}}).
\end{equation}
\end{Lemma}

\begin{proof}
Let $t$ be a nonzero section of $\mathcal{L}$. Then  $t^n$ is a nonzero section of $\mathcal{L}^n$.
So, it suffices to prove the Lemma for $n=1$, i.e., to prove that 
$\bar{\mathcal{L}}$ and $\overline{\mathcal{L}\otimes \mathcal{N}}$ are {\sl bimeromorphic}.

The holomorphic $\BC$-linear map of rank 2 vector bundles 
$$\mathcal{N}\oplus \mathbf{1}_T \to (\mathcal{L}\otimes \mathcal{N})\oplus \mathbf{1}_T, \
\big(s_x; (x, c)\big) \mapsto \big(s_x\otimes t(x); (x, c)\big) \ \forall x \in T, s_x \in \mathcal{N}_x,  c \in \BC$$
induces a bimeromorphic isomorphism of their projectivizations $\bar{\mathcal{N}}$ and $\overline{\mathcal{L}\otimes \mathcal{N}}$. Hence, the groups $\Bim(\bar{\mathcal{N}})$ and $\Bim(\overline{\mathcal{L}\otimes\mathcal{N}})$
are isomorphic. Now the second assertion  of our Lemma follows from Lemma \ref{ThetaEmbBP1}.

\end{proof}

\begin{Corollary}
\label{LandN}
We keep the notation and assumptions of Lemma \ref{CP1Bim}.
In particular, $\mathcal{L}$ is isomorphic to $\mathcal{L}(H, \alpha)$
and admits a nonzero holomorphic section.

Suppose that $\mathcal{N}$ is isomorphic to $\mathcal{L}(H_0, \beta)$
where the kernel $\ker(H_0)$ of the Hermitian form $H_0$
contains the kernel $\ker(H)$  of the Hermitian form $H$.

Then the group $ \Bim(\bar{\mathcal{N}})$ is not Jordan. 

\end{Corollary}

\begin{proof}
Let us consider the alternating $\BR$-bilinear forms
$E:=\mathrm{Im}(H)$ and $M:=\mathrm{Im}(H_0)$ on $V$.
We have
 $$\ker(E)=\ker(H)\subset \ker(H_0)=\ker(M)$$
and therefore $\ker(E) \subset \ker(M)$. Notice also that 
the alternating form $\mathbf{M}(n)=nE+M$ is the imaginary part of the Hermitian form $nH+H_0$
for all positive integers $n$; in addition, obviously, the holomorphic line bundle 
$$\mathcal{L}^n \otimes \mathcal{N} \cong
\mathcal{L}(H,\alpha)^n \otimes \LL(H_0,\beta)=\mathcal{L}(nH+H_0,\alpha \beta^n)=
\mathcal{L}(\mathbf{M}(n),\alpha \beta^n).$$

In light of Lemma \ref{CP1Bim}, there is a group embedding
$$\mathfrak{G}(nH+H_0,\alpha \beta^n) \hookrightarrow \Bim(\bar{\mathcal{N}}).$$

On the other hand, applying Lemma \ref{thetaTheta} to $(nH+H_0,\alpha \beta^n)$ (instead of $(H,\alpha)$),
we conclude that $\mathfrak{G}(nH+H_0,\alpha \beta^n)$ is a {\sl theta group} attached to
the {\sl symplectic pair} $\left(K_{\mathbf{M}(n),\Pi}, e_{\mathbf{M}(n)}\right)$.
Now the desired result follows from Theorem \ref{JordanUnbounded}.

\end{proof}

\begin{Definition}
\label{AlgModel}
Let $T=V/\Gamma$ be a complex torus. We write $T_a$ for its algebraic model, which is also
a complex torus (even an abelian variety) provided with a surjective holomorphic homomorphism of complex tori
$$\pi_a: T \twoheadrightarrow T_a$$
with connected kernel (actually, all the fibers of $\pi_a$ are connected) \cite[Ch. 2, Sect. 6]{BL}. We write $\dim_a(T)$
for $\dim(T_a)$ and call it the {\sl algebraic dimension} of $T$. 

Clearly,
$$\dim(T_a) \le \dim(T);$$
the equality holds if and only if $T=T_a$, i.e., $T$ is an {\sl abelian variety}.
\end{Definition}

\begin{Theorem} [Theorem 1.7 of \cite{Zar19}]
\label{torusBP1}
Suppose that a complex torus $T=V/\Pi$ has positive algebraic dimension.
Then $\Bim(T\times \BP^1)$ is not Jordan.
\end{Theorem}

\begin{proof}
Take $\mathcal{N}=\mathbf{1}_T$. Then $\bar{\mathcal{N}}=T\times \BP^1$. On the other hand,
$\mathcal{N}=\mathbf{1}_T \cong \LL(\mathbf{0},\mathbf{1})$ where $\mathbf{0}$ is the {\sl zero}
Hermitian form on $V$ and 
$$\mathbf{1}_{\Pi}: \Pi \to \{1\}\subset \mathbf{U}(1)\subset \BC^{*}$$
is the constant semicharacter (actually, a character) of $\Pi$ that  identically equals $1$. Clearly,
$$\ker(\mathbf{0})=V. $$

Since $\dim_a(T)>0$, the algebraic model $T_a$ is a {\sl positive-dimensional} abelian variety.
Then $T_a$ admits an {\sl ample} holomorphic line bundle $\mathcal{L}_a$ with a nonzero section. Since $\psi: T \to T_a$ is surjective,
the inverse image $\mathcal{L}=\psi^{*}\mathcal{L}_a$ is a holomorphic line bundle on $T$ that also admits a nonzero section.
We have $\mathcal{L}\cong \LL(H,\alpha)$ for some A.-H. data $(H,\alpha)$. Obviously,
$$\ker(H)\subset V=\ker(\mathbf{0}).$$
Therefore we may apply Corollary \ref{LandN} and obtain that the group $\Bim(\bar{\mathcal{N}})$ is {\sl not} Jordan.
It remains to recall that $\bar{\mathcal{N}}=T \times \BP^1$.

\end{proof}

 The following assertion is a generalization of Theorem \ref{torusBP1}.

\begin{Theorem}[A special case of Theorem 1.8 in \cite{Zar19}]
\label{liftAH}
Let $\psi: T \to A$ be a surjective holomorphic group homomorphism  from a  complex torus $T=V/\Pi$ to a positive-dimensional complex abelian variety $A$.
Let $\mathcal{M}$ be a holomorphic line bundle over $A$ and $\mathcal{F}$ be a holomorphic line bundle over $T$ that is isomorphic to the inverse image $\psi^{*}\mathcal{M}$.

Then the group $\Bim(\bar{\mathcal{F}})$ is not Jordan.
\end{Theorem}

\begin{proof}
A positive-dimensional complex abelian variety $A$ is a complex torus
$A=W/\Gamma$ (where $W$ is a complex vector space of finite positive dimension $m$ and $\Gamma$ a discrete lattice of rank $2m$ in $W$) that admits a {\sl polarization}, i.e., a {\sl positive} (and therefore {\sl nondegenerate}) Hermitian form
$$\mathbf{H}_A: W \times W \to \BC,$$
whose imaginary part
$$\mathbf{E}_A: W \times W \to \BR,  \ (w_1,w_2)\mapsto \mathrm{Im}(\mathbf{H}_A(w_1,w_2))$$
satisfies the condition
$$\mathbf{E}_A(\Gamma,\Gamma)\subset \BZ.$$
Replacing if necessary, $\mathbf{H}_A$ by $2\mathbf{H}_A$, we may and will assume that
$$\mathbf{E}_A(\Gamma,\Gamma)\subset 2\cdot \BZ.$$

Then obviously $(\mathbf{H}_A, \mathbf{1}_{\Gamma})$ is an A.H. data on $(W,\Gamma)$. 
The positiveness of  $\mathbf{H}_A$ implies that  the corresponding holomorphic line bundle $\LL(\mathbf{H}_A,\mathbf{1})$ over $A$ has a {\sl nonzero} holomorphic section (the corresponding theta function)  (see \cite[Theorem  2.1]{Kempf}).

It follows from \cite[Lemma 2.3.4 on p. 33]{BL} that every surjective holomorphic homomorphism $\psi: T \to A$ is induced by
a certain surjective $\BC$-linear map $\bar{\psi}: V \to W$ in the sense that
$$\bar{\psi}(\Pi)\subset \Gamma; \ \
\psi(v+\Pi)=\bar{\psi}(v)+\Gamma \in W/\Gamma=A \ 
\forall v+\Pi\in V/\Pi=T.$$
The surjectiveness of $\psi$ implies that the induced holomorphic line bundle $\mathcal{L}=\psi^{*}\LL(\mathbf{H}_A,\mathbf{1}_{\Gamma})$ over $T$ also has a {\sl nonzero} holomorphic section. 

Let $(H_A,\beta)$ be an A.-H. data on $(W,\Gamma)$ and $\mathcal{L}(H_A,\beta)$ the corresponding holomorphic line bundle over $A=W/\Gamma$.
Then the inverse image $\psi^*\mathcal{L}(H_A,\beta)$ is isomorphic to $\mathcal{L}(H_A\circ \bar{\psi},\beta\circ \bar{\psi})$ where the A.-H. data
$(H_A\circ \bar{\psi},\beta\circ  \bar{\psi})$ for $(V,\Gamma)$ is as follows (see \cite[Lemma 2.3.4]{Kempf}).
\begin{equation}
\label{baseAHchange}
H_A\circ \bar{\psi}: V \times V \to \BC,  \ (v_1,v_2) \mapsto H_A(\bar{\psi}v_1, \bar{\psi}v_2); \ \
\beta\circ \bar{\psi}: \Pi \to \mathbf{U}(1), \  l \mapsto \beta(\bar{\psi}(l)).
\end{equation}

In light of the nondegeneracy of $\mathbf{H}_A$, this implies that
\begin{equation}
\label{kernelAHchange}
\ker(\mathbf{H}_A\circ \psi)=\ker(\bar{\psi})\subset \ker(H_A\circ \bar{\psi})\subset V.
\end{equation}

Now let $(H_A,\beta)$ be the A.-H. data on $(W,\Gamma)$ such that $\mathcal{M}$ is isomorphic to
$\LL(H_A,\beta)$. In light of \eqref{baseAHchange}, $\mathcal{F}$ is isomorphic to
$\LL(H_A\circ \bar{\psi},\beta\circ \bar{\psi})$.  In particular,
$\mathcal{L}=\psi^{*}\LL(\mathbf{H}_A,\mathbf{1}_{\Gamma})$ is isomorphic
to $\LL(\mathbf{H}_A\circ \bar{\psi},\mathbf{1}_{\Pi})$.
(Here 
$$\mathbf{1}_{\Pi}=\mathbf{1}_{\Gamma}\circ\bar{\psi}:\Pi \to \{1\}\subset \mathbf{U}(1)$$
 is the trivial character of $\Pi$.) Since $\mathcal{L}$ admits a nonzero holomorphic section,
the inclusion \eqref{kernelAHchange} allows us to
 apply Corollary \ref{LandN}  to $\mathcal{N}=\mathcal{F}$ and $H_0=H_A\circ \bar{\psi}$,
and conclude  that $\Bim(\bar{\mathcal{F}})$ is {\sl not} Jordan.

\begin{Remark}
Let $V,\Pi,T$ and $\mathcal{F}$ be as in Theorem \ref{liftAH}. Suppose that $\mathcal{F} \cong \mathcal{L}(H,\alpha)$.
Let $\alpha^{\prime}: \Pi \to \mathbf{U}(1)$ be a map such that $(H,\alpha^{\prime})$ is also an A.H. data on $(V,\Pi)$.
Let $\mathcal{F}^{\prime}$ be a holomorphic line bundle on $T$ that is isomorphic to $\mathcal{L}(H,\alpha^{\prime})$.
Then the same arguments as in the proof of Theorem \ref{liftAH} prove that $\Bim(\overline{\mathcal{F}^{\prime}})$ is also non-Jordan
 (see Theorem 1.8 of \cite{Zar19}).

\end{Remark}

\end{proof}

\chapter{Non-trivial $\BP^1-$bundles over a non-uniruled base}\label{base}

In this chapter we consider the group $\Aut(X)$ for a  non-trivial $\BP^1-$ bundle over a  non-uniruled   compact complex connected K\"ahler  manifold 
$Y.$ Recall that there is  homomorphism $\tau:\Aut(X)\to\Aut(Y)$  and its kernel is denoted by $\Aut(X)_p.$ 
 First we classify  automorphisms  $f\in\Aut(X)_p,$  i.e. those automorphisms that 
  do not move fibers of $p.$ We get that if $\Aut(X)_p\ne\{id\}$  then either $X$ or its double cover  is a projectivization $\BP(\EE)$ of rank two vector bundle over $Y$ or its double cover, respectively.  Thus, if $Y$  is K\"ahler, so is $X$ (\cite  [Proposition 3.5]{Voisin}).
  Thus the group $\Aut(X)$ is Jordan by a  Theorem of Jin Hong   Kim (\cite{Kim}). 
  It appears that 
 if $X$ is scarce,  (i.e. it does not have many sections, see \defnref{configuration}  below),     then $\Aut_0(X) $ is commutative and $\Aut(X) $  is very Jordan.  This is, for example, the case when $Y$ is torus of algebraic dimension zero.

\section{ Automorphisms  of $\BP^1-$bundles  that preserve fibers}\label{automorphisms}

This section contains the classification of those automorphisms of a  $ \BP^1$-bundle $X$  that preserve the fibers of $p: X\to Y.$  There are three different types, each one is described in a separate subsection. 

  Let $(X,p,Y)$  be  a  $ \BP^1$-bundle  over  a compact complex   connected 
     manifold $Y,$  i.e.,   \begin{itemize}
\item  $X, Y $ are compact connected complex  manifolds  of positive dimension;

 \item  $p:X\to Y$  is a surjective holomorphic map; 
\item  $X$ is a   holomorphically locally trivial fiber bundle over $Y$ with  fiber  $\BP^1$  and 
with the  corresponding projection map $p:X\to Y.$ \end{itemize}
 Let 
 $P_y$ stand for the fiber $p^{-1}(y).$
  Let $U\subset Y$  be an open non-empty subset of $Y.$ We call a covering $U=\cup U_i, i\in I,$   by open subsets $U_i$ of  $Y$ to be  {\it fine} if  for every $i\in I$      there exists an isomorphism  $ \phi_i: V_i=p^{-1}(U_i)\to U_i\times \BP^1_{(x_i:y_i)}$ such that:

--  $(u, z_i), \ u\in U_i, z_i=\frac{x_i}{y_i}\in \ov\BC,$ are   local coordinates in  $ V_i:=p^{-1}(U_i)\subset X;$ 

-- $\mathbf{pr}\circ\phi_i=p, $
 where  $\mathbf{pr}:U_i\times \BP^1\to\BP^1$ is the natural projection
      (see {\bf Notation  and Assumptions}(14)).

 \begin{Definition}\label{ssection} 
 An $k-$section $S$ of $p$ is a  codimension 1  irreducible analytic subset $D\subset X$  such that the intersection $X\cap P_y$  is finite for every $y\in Y$ and  consists of $k$ distinct  points for a general $y\in Y.$  We call {\sl bisection}  a $2-$section that meets  {\bf every} fiber at two distinct points. 
 Obviously, {\sl usual} holomorphic section  $S$  of $p$ is a $1$-section. 
A section $S$ is defined by the set $\mathbf a=\{a_i(y)\}$ of functions $a_i:U_i\to \BP^1$ such that $p (y,a_i(y))=id, \ y\in U_i.$
 We will denote this by $S=\mathbf a.$
\end{Definition}

\begin{Lemma}\label{p00}  Let $A_1,A_2,A_3$ be 3 distinct  almost 
sections of $p$  (see \defnref{almost}).   
Assume that there is an analytic subspace $\Sigma\subset Y$ of codimension  at least  2
such that $A_k, k=1,2,3,$ are pairwise disjoint in $V=p^{-1}(U),$ where $U=Y\setminus \Sigma.$

Then there exists an isomorphism  $\Phi :X\to Y\times\BP^1$ such that $\mathbf{pr}\circ \Phi=p$  where  $\mathbf{pr}:Y\times \BP^1\to\BP^1$ is the natural projection
 (see {\bf Notation  and Assumptions}(14)).
\end{Lemma}

\begin{proof}    Indeed, let $\{U_i\}$ be a fine covering of $Y$  and let 
$$a_{ki}(u)x_i-b_{ki}(u)y_i=0, u\in U_i$$ be the eqaution of $A_k\cap  U, k=1,2,3,$ over  $U_i.$ 
 We define a  meromorphic function $F(x)$ in every $V_i$ by 
\begin{equation}\label{p001}F(x)=\frac{(a_{1i}(u)x_i-b_{1i}(u)y_i)(a_{2i}(u)b_{3i}(u)-a_{3i}(u)b_{2i}(u))}
{(a_{2i}(u)x_i-b_{2i}(u)y_i)(a_{1i}(u)b_{3i}(u)-a_{3i}(u)b_{1i}(u))}, \ u=p(x).
\end{equation}
  Then $F(x)$ is  globally everywhere   defined and meromorphic  in $V.$  Its restrictions to $A_1\cap V, A_2\cap V,A_3\cap V$ are equal to $0,\infty,1,$ respectively. 
  
     The fiber of $p$ has dimension 1, thus   $X\setminus V=p^{-1}(\Sigma)$ has codimension 2  in $X.$ Thus  the  function $F$ may be extended to a   meromorphic function on the whole $X$  by the  Levi’s continuation theorem
	(\thmref{Levi}). Thus, we have  the bimeromorphic map $\Phi:X\to Y\times\BP^1, \Phi(x)=(p(x),F(x))$  that induces   an isomorphism of $V$ onto $U\times \BP^1$  that is compatible with $p.$
 According to \lemref{3.8}, $\Phi$ is an isomorphism.
 \end{proof}

\begin{Remark}\label{threedisjoint}   In particular, if there are three disjoint sections in $X$ then $X\sim Y\times\BP^1.$ 
\end{Remark}

\begin{Remark}\label{codim2} 
 Note that a section is an almost section.  If $A$ is an almost section but not a section then the set 
 $$\Sigma(A)=\{y\in Y \ | \ p^{-1}(y)\subset A\}\subset Y$$
  has codimension at least two  because

--$\tilde \Sigma:=p^{-1}(\Sigma(A)) $ is a proper analytic subset of $A$ with $\dim(A)=\dim(Y)=n;$ thus $\dim(\tilde \Sigma)\le  n-1;$

-- every  fiber of restriction of $p$ to $\tilde \Sigma$ has dimension 1.
\end{Remark}
 \begin{Definition}\label{configuration} 
We say that three sections $S_1,S_2,S_3$ in $X$ are  good configuration if $S_1\cap S_2=S_1\cap S_3=\emptyset$ and $S_2\cap  S_3\ne \emptyset. $
We say that three almost sections $A_1,A_2,A_3$ in $X$ are a special  configuration if $A_1\cap A_2=A_1\cap A_3=A_2\cap  A_3. $
We say that $X$ is {\it scarce }   if $X$ admits no special configurations.

\end{Definition}

\begin{Lemma}\label{p0}  Let $S_1,S_2,S_3,S_4$ be 4 distinct
sections of $p$ such that 
$S_1\cap  S_2=\emptyset,$ $S_3\cap  S_4=\emptyset.$
Then $X\sim  Y\times\BP^1.$\end{Lemma}
\begin{proof} If $S_3\cap (S_1\cup  S_2)=\emptyset, $ then
$X\sim Y\times\BP^1$  (\remarkref{threedisjoint}).     Assume that $X\not \sim Y\times\BP^1.$ 
Let
$\emptyset\ne S_3\cap S_2=D\subset S_2.$   Let $\{U_i\},i\in I$ be a fine   covering of $Y.$
In every $V_i=p^{-1}(U_i)$ we choose coordinates   $(y,z_i)$   in such a way  that $S_2\cap V_i=\{z_i=0\},S_1\cap V_i=\{z_i=\infty\}.$
Then  $z_j=\lambda _{ij}z_i$ in $V_i\cap V_j,$  where $\lambda _{ij}$ are non-vanishing in $U_i\cap U_j$   holomorphic functions. 

Let  $S_3\cap  V_i=\{(y,z_i=p_i(y)), y\in U_i\},$  where $p_j=\lambda _{ij}p_i,$ and $S_4\cap  V_i=\{(y,z_i=q_i(y)), y\in U_i\},$  where $q_j=\lambda _{ij}q_i .$ Then $r(y):=\frac {p_i(y)}{q_i(y)}$ is a  globally defined  
meromorphic function on $Y$  that omits value $1$  (since $S_3\cap S_4=\emptyset $). Thus, $r:=r(y)=constant.$  But then $q_i$   vanishes at $D$ and $S_3\cap  S_4\supset D. $ Contradiction. 
\end{proof}

\begin{Remark}\label{Disjointsections}  We proved also the following fact: If  $X$ contains two disjoint sections  $S_1$ and $S_2$, then \begin{itemize}
\item there is a holomorphic line bundle $\LL:=\LL(S_1,S_2)$ such that $X\sim\BP(\LL\oplus\mathbf{1}_Y);$
\item there is a fine covering $\cup U_i, i\in I$ of $Y$    and coordinates  $(u,z_i), u\in U_i, z_i\in\ov\BC$  in $V_i, $  such that $$S_1\cap  V_i=\{z_i=\infty\}, S_2\cap  V_i=\{z_i=0\}. $$ 
 \item $z_j=a_{ij}z_i, $   and cocycle 
 $\mathbf{a}=\{a_{ij}\}$   defines $\LL.$\end{itemize}\end{Remark}

 \begin{Lemma}\label{pp00}   If there exist  3 distinct  almost 
sections  $A_1,A_2,A_3$  of $p$ then  there exist a bimeromorphic map   $\Phi :X\to Y\times\BP^1$ such that $\mathbf{pr}\circ \Phi=p.$  

\end{Lemma}

\begin{proof} 

  We maintain the  notation of the proof   of \lemref{p00}

 Let $$\Sigma(A_i)=\{y\in Y \ | p^{-1}(y)\subset A_i\}, \   i=1,2,3, \ \text{ and}  \ \Sigma=\bigcup\limits_1^3\Sigma(A_i).$$
Let $\tilde \Sigma=p^{-1}(\Sigma).$

  The  function  $F(x)$  defined by \eqqref{p001}   is  defined and meromporhic   at every point outside the set 
  $$D= (A_1\cap A_3)\cup (A_2\cap A_3)\cup(A_1\cap A_2)\cup\tilde \Sigma.$$ 
 Since codimension  of $D$ is at least 2,
  the function $F$ may be extended to a meromorphic function on $X$ by the Levi Theorem. 
 Consider a map $\Phi:  X\to Y\times \BP^1. \ x\mapsto (p(x),F(x)).$ It is  meromorphic and 
 induces an 
  isomorphism on every fiber 
 $P_u,u\not\in  p(D)$ to $\BP^1.$    Thus $\Phi$ is bimeromorphic.
 \end{proof}

\begin{Lemma}\label{p000}  If $X$ admits a good configuration  $S_1,S_2,S_3,$ then $X$ admits a special configuration. 
\end{Lemma}
\begin{proof}  By assumption  $S_1\cap  S_2=S_1\cap  S_3=\emptyset, S_3\cap S_2\ne\emptyset$. 
Recall that  $S_2$ is a zero section of the line bundle $\LL(S_1,S_2)$  (see   \remarkref{Disjointsections}). Let $\{U_i\}, i\in I$ be a fine covering of $Y $   and $(u,z_i), u\in U_i, z_i\in\ov\BC$ be  coordinates in $V_i, $  such that $S_1\cap  V_i=\{z_i=\infty\}, S_2\cap  V_i=\{z_i=0\}. $ 
Let  the non-zero section of $\LL,$ namely,   $S_3$ have  the equation  $z_i=h_i(u)$ in $V_i.$ 
For  any $c\in\BC^*$ the equations $z_i=ch_i$ will also define a section   $S_4\ne S_3$  of $\LL.$ 
By construction, $S_2\cap  S_3=S_2\cap  S_4=S_3\cap  S_4=\bigcup\limits_{i\in I}\{h_i=0\}.$  Thus, $S_2, S_3, S_4$ is a special configuration.\end{proof}

  We now consider the subgroup $\Aut(X)_p$ of those automorphisms $f$  of $X$ that do not move fibers of $p$, i.e., such that $p\circ  f=f.$
Similarly to \lemref{p00}, every $f\in\Aut(X)_p$ defines locally a   holomorphic map $\psi_f:Y\to \PSL(2,\BC)$ and the function $$\TD(y), y\to \TD(\psi_f(y))=\frac{\tr^2(\psi_f(y))}{\det(\psi_f(y))}$$  (see {\bf Notation})    is everywhere defined and holomorphic, hence constant on $Y$  (\cite[Remark 4.9]{BZ20}).    We denote this constant by $\TD(f).$

Assume that $X\not\sim  Y\times\BP^1.$ Let  $f\in\Aut(X)_p, f\ne id.$   Recall that $\Fx(f)$ is  the set of all  fixed points  of $f$.  Let $\{U_i\},i\in I$  be a fine covering of $Y.$
 We summarize  in \lemref{typea}  and \lemref{doublecover} below the properties of   {\sl non-identity} automorphisms  $f\in\Aut(X)_p$ with
 $\TD(f)\ne 4$  (\cite{BZ20}).

\begin{Lemma}\label{typea}    Assume that  $(X,p,Y) $  is a $\BP^1-$bundle and  $X\not\sim  Y\times\BP^1.$ Let  $f\in\Aut(X)_p, f\ne id,$ and  $\TD(f)\ne 4.$ 
 Then  one of two following  cases holds. 
\begin{itemize}
\item[{\bf  Case  A.}]   The  set $\Fx(f)$ consists  of exactly  two disjoint sections $S_1,S_2 $  of $p.$   We say that $f$ is 
 of type {\bf A} with data  $(S_1,S_2), $  an ordered pair. In notation of \remarkref{Disjointsections},
 let   $\{U_i\}, i\in I,   \ \LL(S_1,S_2),$   and $\mathbf a=\{a_{ij}\}$  be the corresponding  fine covering,  holomorphic line bundle  and cocycle, respectively.  

 Then \begin{itemize}\item  Defined is the number $\lambda_f\in\BC^*$ such that in every $V_i$ 
\begin{equation}\label{lambda}f(u,z_i)=(u,\lambda_f z_i);\end{equation}
\item If  $G_0\subset Aut(X)_p$   be the subgroup of all $f\in\Aut(X)_p$  such that  $f(S_1)=S_1,$  $f(S_2)=S_2,$  then  $G_0\cong \BC^*;$
\item The restriction $f\to f \mid _{P_y}$ defines a group embedding of $G_0$ into $\Aut(P_y).$\end{itemize}

\item[{\bf    Case C.}] The  set $\Fx(f)$   is  a smooth unramified double cover  $S$ of $Y.$  We will call such  $f$ an automorphism  of type {\bf C }  with data $S.$ Here  $S$ is a  bisection of $p.$ \end{itemize}\end{Lemma}
\begin{proof} $\TD(f)\ne 4$ implies that  $f$ has exactly two distinct fixed points at {\bf every   }  fiber $ P_y=p^{-1}(y), y\in Y.$  
 Thus $\Fx(f) $ is either a union  of   two disjoint  sections or is a 2-section  of $p.$  In {\bf  Case  A}    \eqqref{lambda} follows from the fact that 
$$f(u,z_i)=\lambda_iz_i,  \   f(u,z_j)=\lambda_jz_j=\lambda_ja_{ij}z_i=a_{ij}\lambda_iz_i.$$
 The constant  $\lambda_f=\lambda_i\ne 0$ does not depend on the choice of the fiber, hence   $f$ is  determined uniquely by its  restriction to every given fiber. 
On the other hand for every $\lambda\in\BC^*$  there exists an  automorphism $f_{\lambda}\in\Aut(X)_p$ defined in every $V_i$ by  $$(u,z_i)\to (u, \lambda z_i).$$
Therefore  $G_0\cong \BC^*.$  
\end{proof}

\begin{Lemma}\label{doublecover} (see \cite{BZ20})   Let $S$ be a bisection of the $\BP^1-$bundle $(X,p,Y).$

Consider $$\tilde X:=\tilde X_S:=S\times_Y  X=\{(s,x)\in S\times  X \subset X\times  X  \ \ \mid  p(s)=p(x)\}.$$ We denote the  restriction of $p$  to  $S$
by the same letter  $p,$  while   $p_X$ and $\tilde p$ stand for 
 the restrictions to $\tilde X$ of the  natural projections $ S\times  X\to X$ and $ S\times  X\to S $ respectively.  We write
$\invl:S\to S$ for  the involution (the  only non-trivial deck transformation for $ p \bigm |_S$).
Then   $(\tilde X,\tilde p,S)$   is a $\BP^1-$bundle  with the following properties:

 \begin{itemize}

\item [a)]  
The following 
 diagram commutes
\begin{equation}\label{diagram41}
\begin{CD}
 \tilde X  @>{p_X}>>X\\
@V\tilde p VV@V pVV\\
S  @>{p\mid_{S}}>>Y
\end{CD}.\end{equation}
\item [b)] $p_X:\tilde X\to X $ is an unramified double cover of $X;$
\item [c)] Every fiber  $\tilde p^{-1}(s), s\in S$ is isomorphic to $$P_{p(s)}=p^{-1}(p(s))\sim \BP^1;$$
\item [d)] The  $\BP^1$-bundle $\tilde X$ over $S$ has two disjoint sections, namely: $$S_+:=S_+(f):=\{(s, s)\in \tilde X, \  s\in S\subset X\}$$ and $$ S_-:=S_-(f):=\{(s, \invl(s))\in \tilde X,  \ s\in S\subset X\}.$$ They are mapped onto $S$ isomorphically by $p_X.$

\item [e)]  Every $h\in \Aut(X)_p$ induces an automorphism $\tilde h\in \Aut(\tilde X)_{\tilde  p}$ defined by $$\tilde h(s,x)=(s,h(x)).$$ 
 
\item [f)]  The  involution  $s\to \invl(s)$ may be extended from $S$ to a  holomorphic involution  of $\tilde X$ by  $$\invl(s,x)= (\invl(s),x);$$

\item[g)] Every section $N=\{y,\sigma(y)\}  $   of $p$ in $X$  induces the section
 $\tilde N :=\{(s,  \sigma(p(s))\}$  of  $\tilde p $  in $\tilde X.$   
We have   $p_X(\tilde N)=N$ is a section of $p,$ 
  thus  $\tilde N$ cannot coincide  $ S_+$ or $ S_-.$\end{itemize}
\end{Lemma}

\subsection{Automorphisms with   $TD=4$}\label{sectypeb}

If  $f\in\Aut(X)_p, f\ne id$ and $\TD(f)=4,$ then   there is precisely one fixed point of $f$ in the fiber $P_y=p^{-1}(y)$ 
over  the general point $y\in Y.$ That means that $\Fx(f)$ contains precisely one  almost section $D$  of $p.$   In this case we say that $f$ is of type {\bf B} 
with data $D.$

\begin{Lemma}\label{typeb2}Let $(X,p ,Y)$ be a $\BP^1-$bundle, where $X,Y$ are compact connected   complex manifolds, $\dim(Y)=n, $ $f\in\Aut(X)_p,f\ne id, $ and $\TD(f)=4.$ Let $D$  be the only   almost section  contained in  $\Fx(f).$  Let $\Sigma =\{y\in Y \ | \ {P_y }\subset D\}$ and $U=Y\setminus \Sigma, \  V=p^{-1}(U)\subset X.$  Let $\tilde S$ be the union of all irreducible   distinct from $D$  components of $\Fx(f)$  and $S=p(\tilde S).$
 
 Then\begin{enumerate}
 
\item
 there is a  fine covering   $  U_i,i\in J$ 
of $U$ and coordinates $(u, z_i)$  in $V_i=p^{-1}(U_i)$ such that $D\cap
V_i=\{z_i=\infty\}.$
 \item  $f(u,z_i)=(u, z_i+\tau_i(u)),$  where $\tau_i$ are holomorphic functions on $U_i;$
\item   if $i,j\in  J$ then $z_j=\mu_{ij}z_i+\nu_{ij}$ where $\mu_{ij}$  and $\nu_{ij}$ are holomorphic functions in $U_i\cap  U_j$ and $\mu_{ij}$   does not vanish.
Moreover,  $\mu_{ij}$  depend on $D$ and the choice of coordinates in $V_i$ but not on $f.$

\item  if $i,j\in  J$ then $\tau_j=\mu_{ij}\tau_i$ in $U_i\cap  U_j.$
 \item $S$  has pure  codimension 1  in $Y.$ \end{enumerate}\end{Lemma}

\begin{proof}   Recall that  the set $\Sigma$ has codimension at least two  in $Y.$ (\remarkref{codim2}).  

 (1) follows from the fact that   $D$ is a section of $p$ over $U.$

(2) follows from the fact that    $D\subset\Fx(f)$, thus the restriction of $f$  onto a fiber $P_y, y\in U_i$ is an automorphism of $\BP^1$ which is either identity or has   the only fixed point  $z_i=\infty.$

(3) follows from the  fact that    $z_j$  is obtained from $z_i$  by an automorphism of $\BP^1$ with $z=\infty$ fixed.

 Since $X$ admits an almost section,  $X\sim \BP(\EE)$
for some rank two holomorphic vector bundle $\EE$ on $Y$   with projection $\pi:\EE\to Y$  (\cite[Lemma 3.5]{kostya},\thmref {PS100}). 
   That means that we have a fine covering $\{U_i\}$ and  a  cocycle  $\{A_{ij}\in\GL(2,\OO(U_i\cap U_j))\}$ of two by two   transition matrices  of $\EE$
such that \begin{itemize}\item   $\pi^{-1}(U_i)\sim U_i\times \BC^2_{x_i,y_i};$
\item     if $U_i\cap U_j\ne\emptyset$   then 
$$A_{ij}\begin{bmatrix}x_i  \\   y_i\end{bmatrix}=
\begin{bmatrix}x_j\\y_j\end{bmatrix}.$$\end{itemize}
 
  Since $D\cap V$ is a section of $p$ over $U$ we may  choose  a basis in  $ \BC^2_{x_i,y_i}$ in such a way that the preimage 
of $D\cap  U_i$  in $U_i\times \BC^2_{x_i,y_i}$  is $U_i\times \{(x_i,0)\}, \ x_i\in\BC. $ For these coordinates 
\begin{itemize}\item
$$A_{ij}\begin{bmatrix}1 \\   0   \end{bmatrix}=
\begin{bmatrix}\lambda_{i,j}\\0\end{bmatrix}$$
\item   \begin{equation}\label{l100}A_{ij}=\begin{bmatrix}\lambda_{i,j}&b_{ij}\\0&\tilde\lambda_{i,j}\end{bmatrix},\end{equation}  where  
$b_{ij},\lambda_{i,j}, \tilde\lambda_{i,j}$  and 

\begin{equation}\label{lambda1}d_{ij}=\lambda_{i,j}\tilde\lambda_{i,j}=\det (A_{ij})\end{equation}
are holomorphic  functions  in $U_i\cap U_j.$\end{itemize}
Let now $z_j=\frac{x_j}{y_j}, z_i=\frac{x_i}{y_i}.$  Then  \begin{equation}\label{lambda2}z_j=\frac{\lambda_{i,j}x_i+b_{ij}y_i}{y_i\tilde\lambda_{i,j}}=\mu_{ij}z_i+\nu_{ij}.\end{equation}

Thus $\mu_{ij}=\frac{\lambda_{i,j}^2}{d_{ij}}=\frac{\lambda_{i,j}}{\tilde\lambda_{i,j}}$ depends on  the choice of $D,$ and is defined by the eigenvalue of the   basis vector in the invariant subspace 
representing $D.$ It does not depend on the choice  of   $f$ with the given  data $D.$

  Note that both $\{\lambda_{i,j}\}$ and $\{\tilde\lambda_{i,j}\}$ form cocycles for the covering of $U.$

 (4) follows 
 from the fact that $f$ is globally defined, and $D$ is fixed,  thus 
$$f(u,z_j)=(u, z_j+\tau_j(u))=(u, \mu_{ij}z_i+\nu_{ij}+\tau_j(u))=(u, \mu_{ij}(z_i+\tau_i(u))+\nu_{ij}).$$

(5) follows 
 from the fact that $\tau_i$ are holomorphic  and $S\cap  U_i=\{\tau_i=0\}.$  Indeed, let  $\tilde S_1\subset \tilde S$   be an irreducible component of $\tilde S.$  It cannot be an almost section, thus $ S_1=p(\tilde S_1)$ is a proper analytic subset of $Y.$ Moreover, since $\tilde \Sigma\subset D,$ we have: $\tilde S_1\not\subset\
\tilde \Sigma, \ S_1 \not\subset \Sigma.$  Thus, $S_1\cap U$ is a dense open subset of $S_1.$ Since $S\cap  U_i=\{\tau_i=0\}$  has  pure codimension 1
(if $S\cap  U_i\ne \emptyset$),   the same is valid for every its component that intersect $U_i.$ Thus, $\dim(S_1)=n-1.$
 
 \end{proof}

  \begin{Proposition}\label{c11}   We maintain   the notation of \lemref{typeb2}.
 Let $S_1,\dots,S_k$ be all irreducible components of $S.$   Then \begin{itemize}\item
    For every $l, 1\le  l\le k, $ defined is a non-negative  number   $n_l,$  that is the order  of zero  of $\tau_i$ along the component $S_l$ if
    $S_l\cap  U_i\ne \emptyset.$
 It depends on $l$ but not on $i.$ 
 The holomorphic   line bundle  $\LL(f)$  corresponding to 
the  effective divisor $\Delta_f:=\sum\limits_{1}^{k}n_lS_l$ 
restricts to $U$ to the holomorphic line bundle defined by the cocycle
 $\mu_{ij}.$
 
\item Let $G_D$ be the subgroup of $\Aut(X)_p$  of all those $g\in \Aut(X)_p,$ that have $\TD(g)=4$ 
and $D\subset   \Fx(g).$  Then $G_D$ is isomorphic to the additive group of $H^0(Y,\LL(f)).$   Thus $G_D\cong (\BC^+)^m, \ m>0.$ 
  \end{itemize}\end{Proposition}

\begin{proof}  
Let  $S_l$  be an irreducible component of $S$. For every $U_i$   such that  $S_l\cap  U_i\ne \emptyset$ defined is the order   $n_{li}$ of zero of $\tau_i$ along $S_l.$  In $U_i\cap  U_j$ we have $\tau_j=\tau_i\mu_{ij}.$  Since $\mu_{ij}$ does not vanish, $\tau_j$ has the same order  of zero along $S_l\cap  U_j.$   Since $S_l$ is irreducible
 and $U\cap S_l$ is open and dense in $S_l,$ the order $n_l$ is well   defined
(see, for example \cite[Remarks 2.3.6]{Huy}). 
By construction, the divisor of  $\tau_i$ in $U_i$  is $\Delta_f\cap U_i$, thus the transition, functions for $\LL(f)$ in $U_i\cap U_j$ are $\tau_j/\tau_i=\mu_{ij}.$

 Let $h\in \Aut(X)_p,$ and   $ \TD(h)=4,$ 
and $D\subset   \Fx(g).$  Applying item (3) of \lemref{typeb2}, we get  $h(u,z_i)=(u, z_i+h_i(u))$  where 
 $h_j=\mu_{ij}h_i.$   Thus the function defined in every $U_i$ by $G_h(u)=\frac{h_i}{\tau_i}$ is meromorphic in $U.$  By the Levi Theorem, $G_h(u)$ is meromorphic on $Y.$   By construction, its divisor $(G_h)\ge -\Delta_f,$ thus $G\in  H^0(Y,\LL(f)).$

 On the other hand, let $G$  be  a  meromorphic function on $Y$  with   divisor  $(G)\ge  -\Delta_f$   (i.e., $G\in    H^0(Y,\LL(f))$). 
   For every $i$ the   function $ h_i=G\tau_i$  is holomorphic in  $U_i,$ hence we  can define a holomorphic  automorphism of  every $V_i=p^{-1}(U_i)$  by 
\begin{equation}\label{l1}h(u,z_i)=(u, z_i+h_i(u)). \end{equation}
Since  $h_j:=\mu_{ij}h_i,$ the map $h$ 
  is  an automorphism  of   $V.$   Moreover,  all the     points of $D\cap V=\cup\{z_i=\infty\}$ are fixed  by $h$. 
  By \lemref{3.10} it may be extended to a bimeromorphic map of $X.$ 
 
   By \lemref{3.8},  $ h\in\Aut(X)_p.$  Moreover, $\Fx(\tilde h)$ contains the closure of $D\cap  V,$
   that is $D.$ 
In the general   fiber   $P_y$  of $p$ it has precisely one fixed point $D\cap  P_y,$ thus $\TD(h)=4.$

Thus, we get a one-to-one map $$ \phi:   G_D\to  H^0(Y,\LL(f)), \ \ h\in  G_D\mapsto G_h\in  H^0(Y,\LL(f)).$$

From item (3) of \lemref{typeb2} we get that the composition of $g,h\in\Aut(X)_p$ is defined by the cocycle  $g_i+h_i$  of corresponding cocycles, which implies that 
$$\phi(h\circ  g)=\phi(h)+\phi(g).$$
\end{proof}

 The next Lemma answers the question  when an almost section $D\subset\Fx (f)$ is the section. We used this fact in \cite{BZ20} while dealing with automorphisms of type {\bf B}.

\begin{Lemma}\label{3.12}   We maintain   the notation of \lemref{typeb2}  and \propref{c11}.  If $\Delta_f=0$  then  $D$  is a section.\end{Lemma}

\begin{proof}

 First, let us note that $\Delta_f=0$ implies that   corresponding line bundle $\LL_f$ is trivial and that $f\ne id $ in a fiber $F_y=p^{-1}(y)$ if $y\not\in \Sigma.$
 
  Since $X$ admits an almost section,  $X\sim \BP(\EE)$
for some rank two holomorphic vector bundle $\EE$ on $Y$   (\cite[Lemma 3.5]{kostya},    \thmref {PS100}).  That means that we have  a fine covering $\{U_i\}_{i\in I}$ of $Y$ and a  cocycle   $A_{ij}$ of two by two matrices  ( with holomorphic in $U_i\cap U_j$ entrees)
such that \begin{enumerate}\item   $p^{-1}(U_i)=V_i\sim  U_i\times\BP^1_{x_i:y_i}, z_i=\frac{x_i}{y_i}$ and if $U_i\cap U_j\ne\emptyset$   then 

$$A_{ij}\begin{bmatrix}x_i  \\   y_i\end{bmatrix}=
\begin{bmatrix}x_j\\y_j\end{bmatrix}$$
\item  In every $U_i$ defined is a $2\times 2$ matrix $F_i$  (representing $f$ )  with holomorphic   functions (in $u\in U_i$)  as entries and  with $\TD(F_i)=4, \det(F_i)=d_i\ne 0, $ and  such that $f(u,(x_i:y_i))=(u,(x'_i:y_i')),$ where 
$$
\begin{bmatrix}x_i'\\y_i'\end{bmatrix}=F_i\begin{bmatrix}
x_i\\y_i\end{bmatrix}$$
\item $$F_j(u)A_{ij}(u)=A_{ij}(u)F_i(u)\frac{d_j}{d_i}.$$
\end{enumerate}

Since $4d_i=\tr(F_i)^2$ is a square we may    divide $F_i$  by $\tr(F_i)/2=\sqrt{d_i}$ and assume that $d_i=1$  (we use that $(x_i:y_i)$ are homogeneous coordinates in $\BP^1_{x_i:y_i}$).

Assume that $D$ is not a section, i.e.,  $\Sigma=\{y\in Y \ | \ p^{-1}(y)\subset D\}\ne \emptyset.$

Let  a fine  covering of $Y$  consist of open sets $U_0,\dots,   U_N$  and  let $U_0,\dots,U_k$ intersect  $\Sigma$   while $U=Y\setminus \Sigma= \bigcup\limits_{k+1}^{N}U_i.$

Then for   each  $i>k$ we may assume that \begin{itemize}\item
 $$F_i=\begin{bmatrix}1&\tau_i\\0&1\end{bmatrix}=I+\tau _iV$$ 
with $I $ being   the identity matrix, $\tau_i$ holomorphic functions in $U_i,$ and $V= \begin{bmatrix}0&1\\ 0&0\end{bmatrix}$ 
(by \lemref{typeb2}(2)). 
\item   Recall  that $\LL_f \mid_{U}$ is defined  on $U$ by   cocycle  $\{\mu_{ij}\}, $ where $\mu_{ij}=\tau_j/\tau_i$ is holomorphic non-vanishing function on $U_i\cap  U_j,$ if $U_i\subset U$ and $U_j\subset U $ (by \lemref{typeb2}). 
Since $\LL_f$ is trivial,  
 we may assume that    cocycle  $\{\mu_{ij}\}$ is trivial ,   i.e.,      $\mu_{ij}=1$ and $\tau_i =1$ do not depend on $i$ for $U_i\subset U=Y\setminus \Sigma.$
Moreover from Equations  \eqref{lambda1} and \eqref{lambda2}  we get that $A_{ij}$ are triangular matrices,  and  for the    eigenvalues $\lambda_{ij},\tilde \lambda_{ij} $   of matrices $A_{ij}$ we  have 
    $\lambda_{ij}=\tilde \lambda_{ij} .$    hence,  $\det (A_{ij})=\lambda_{ij}^2.$

 Thus  if both  $i,j>k,$ we may assume that 
$$A_{ij}=\begin{bmatrix}\lambda_{ij}&  \nu_{ij}\\0&\lambda_{ij}\end{bmatrix}$$
where $\lambda_{ij},\nu_{ij}$ are holomorphic  functions in $U_i\cap  U_j.$ 

\end{itemize}

  Take a point $\mathbf {s}\in\Sigma$ and let $U_0$ be a neighborhood of $\mathbf {s}.$ 
Let   $\tilde r(\mathbf {s}) $ be the number of those neighborhoods $U_i$ with $i>k$ in our fine covering  that have 
$U_i\cap  U_0\ne\emptyset.$  Let $r=\tilde r(\mathbf {s}).$
Let 
$$U_t,\dots,U_{t+r}, t>k $$ 
those neighborhoods for which  $U_i\cap  U_0\ne\emptyset, t\le  i\le  t+r.$
 For $t\le i,j\le t+r$  we have :
 \begin{itemize}\item
   $$F_0=A_{i0}(u)F_iA_{i0}(u)^{-1}=I+ W_i=I+  A_{j0}(u)VA_{j0}(u)^{-1 }=I+ W_j,$$ 
where $ W_i= A_{i0}(u)VA_{i0}(u)^{-1 },  t\le  i\le  t+r.$  It follows that the matrix function  $W_i$ defined apriori in $U_0\cap U_i$
may be extended as a matrix function with holomorphic entries to $U_0\setminus \Sigma$ (hence to all $U_0$), and 

\begin{equation}\label{l101}
W_i=W_j.\end{equation}

\item
\begin{equation}\label{l102}A_{i0}(u)A_{j0}^{-1}(u)=A_{ij}(u) \end{equation}
whenever $U_i\cap U_j\cap U_0\ne \emptyset$.

\item   Let $$A_{i0}(u)=\begin{bmatrix}\alpha_1(u)&\beta_1(u)\\  \gamma_1(u)&\delta_1(u)\end{bmatrix},
 A_{j0}(u)=\begin{bmatrix}\alpha_2(u)&\beta_2(u)\\  \gamma_2(u)&\delta_2(u)\end{bmatrix}$$  
Then
\begin{equation}\label{m101} W_i(u)=\begin{bmatrix}-\alpha_1(u) \gamma_1(u)&\alpha_1^2(u)\\  -\gamma_1^2(u)&\alpha_1(u) \gamma_1(u)\end{bmatrix}=
 W_j(u)=\begin{bmatrix}-\alpha_2(u) \gamma_2(u)&\alpha_2^2(u)\\  -\gamma_2^2(u)&\alpha_2(u) \gamma_2(u)\end{bmatrix},\end{equation}

\begin{equation}\label{m102}A_{i0}(u)A_{j0}^{-1}(u)=\frac{1}{d_{j0}}\begin{bmatrix}\alpha_1 \delta_2-\beta_1\gamma_2
& - \alpha_1\beta_2+ \beta_1\alpha_2\\ \gamma_1 \delta_2-\delta_1\gamma_2&-\gamma_1\beta_2+\delta_1\alpha_2\end{bmatrix}=\begin{bmatrix}\lambda_{ij}&  \nu_{ij}\\0&\lambda_{ij}\end{bmatrix}.\end{equation}
\end{itemize}

  Let $\tilde U_{ij}=U_i\cap U_j\cap U_0\ne\emptyset.$ 
    From \eqqref{m101} we get that  in $\tilde U_{ij}$  we have $\al_1^2=\al_2^2$ and $\alpha_1(u) \gamma_1(u)=\alpha_2(u) \gamma_2(u).$ 
Note that these  equations are valid   in all $U_0,$   since   $W_i, W_j$ are defined there. 
    
In  $\tilde U_{ij}$ the following three  cases are possible:     $\al_1=\al_2, \gamma_1=\gamma_2,$ or $\al_1=-\al_2, \gamma_1=-\gamma_2,$  or $\al_1=\al_2=0.$

{\bf Case 1.}  $\al_1=\al_2, \gamma_1=\gamma_2$  in $\tilde U_{ij}.$
 Plugging in  this into     \eqqref{m102} we get the following:
$$\frac{1}{d_{j0}}\begin{bmatrix}\alpha_1 \delta_2-\beta_1\gamma_2
& - \alpha_1\beta_2+ \beta_1\alpha_2\\ \gamma_1 \delta_2-\delta_1\gamma_2&-\gamma_1\beta_2+\delta_1\alpha_2\end{bmatrix}=\frac{1}{d_{j0}}\begin{bmatrix}\alpha_1 \delta_2-\beta_1\gamma_1
& - \alpha_1\beta_2+ \beta_1\alpha_1\\ \gamma_1 \delta_2-\delta_1\gamma_1&-\gamma_2\beta_2+\delta_1\alpha_2\end{bmatrix}
=$$
$$\frac{1}{d_{j0}}\begin{bmatrix}d_{i0}+\alpha_1 (\delta_2-\delta_1)
& \alpha_1(\beta_1-\beta_2)\\ \gamma_1 (\delta_2-\delta_1)&d_{j0}-\alpha_2 (\delta_2-\delta_1)\end{bmatrix}=\begin{bmatrix}\lambda_{ij}&  \nu_{ij}
\\0&\lambda_{ij}\end{bmatrix}.$$

 Thus there are once more  two cases.

{\bf Case 1.1}  $\gamma_1\equiv 0 $   in   $\tilde U_{ij},$ hence $\gamma_1^2=0$ in $U_0.$
Then in all $U_0$ 
 $$ F_0=\begin{bmatrix}1&\alpha_1^2(u)\\  0&1\end{bmatrix}$$
and $\alpha_1^2(u)$ does not vanish in $U_0$  since $\codim(\Sigma)\le 2$ and $\Delta_f=0,$ i.e  $F_0(u)\ne \mathbf {I}$ if $u\not\in\Sigma.$ Thus 
  $D\cap V_0=\{y_0=0\}$ and $\Sigma\cap U_0=\emptyset.$    This contradicts to $\mathbf{s}\in\Sigma.$

{\bf Case 1.2}   $\gamma_1\not\equiv 0,   \ \delta_2\equiv\delta_1$   in  $\tilde U_{ij}.$   Then $1=\lambda_{ij}=\frac{d_{i0}}{d_{j0}}.$  Moreover
$\beta_1=\frac{\al_1\delta_1-d_{i0}}{\gamma_1}=\beta_2=\frac{\al_2\delta_2-d_{j0}}{\gamma_2}$ 
and $\nu_{ij}=0$ in $\tilde U_{ij}\cap\{\gamma_1\ne 0\}.$  Since this set is  open in $U_i\cap U_j$ we have $\nu_{ij}\equiv 0$  and 
$$A_{ij}\equiv \begin{bmatrix}1&0\\  0&1\end{bmatrix}.$$

It follows that that   there is a compatible with $p$  isomorphism $V_i\cup V_j \sim  (U_i \cup U_j)\times \BP^1_{z},$   where $z=\frac{x_i}{y_i}= \frac{x_j}{y_j}.$
 Thus we can replace $U_i,U_j$ by $U_i\cup U_j$ and obtain a new fine covering of $Y$ consisting of $N-1$ open subsets and such that $\tilde r(\mathbf {s}) =r-1.$  Since $U_0$ is connected we can repeat this process (recall that $\gamma_1=\gamma_2\not\equiv 0$ in $U_i\cup U_j$ so we will  stay in {\bf Case 1.2}) till we get a covering with $\tilde r(\mathbf {s})=1.$

 Thus, since $U_0\setminus \Sigma$ was  contained   in  $U_t\cup\dots  U_{t+r}$ we get $p^{-1}(U_0\setminus \Sigma)\sim(U_0\setminus \Sigma)\times \BP^1_{z}.$   By \lemref{3.8}  and \lemref{3.10} it extends to an isomorphism and $D$ is the  preimage of $\{z=\infty\}.$

{\bf Case  2.}  $\al_1=-\al_2, \gamma_1=-\gamma_2.$

 Plugging in  this into     \eqqref{m102} we get the following:
$$\frac{1}{d_{j0}}\begin{bmatrix}\alpha_1 \delta_2-\beta_1\gamma_2
& - \alpha_1\beta_2+ \beta_1\alpha_2\\ \gamma_1 \delta_2-\delta_1\gamma_2&-\gamma_1\beta_2+\delta_1\alpha_2\end{bmatrix}=\frac{1}{d_{j0}}\begin{bmatrix}\alpha_1 \delta_2+\beta_1\gamma_1
& - \alpha_1\beta_2- \beta_1\alpha_1\\ \gamma_1 \delta_2+\delta_1\gamma_1&\gamma_2\beta_2+\delta_1\alpha_2\end{bmatrix}
=$$
$$\frac{1}{d_{j0}}\begin{bmatrix}-d_{i0}+\alpha_1 (\delta_2+\delta_1)
&- \alpha_1(\beta_1+\beta_2)\\ \gamma_1 (\delta_2+\delta_1)&-d_{j0}-\alpha_1 (\delta_2+\delta_1)\end{bmatrix}=\begin{bmatrix}\lambda_{ij}&  \nu_{ij}
\\0&\lambda_{ij}\end{bmatrix}.$$

Similarly to {\bf  Case 1} we have 

{\bf Case 2.1}  $\gamma_1\equiv 0.$  Then 
 $$ F_0=\begin{bmatrix}1&\alpha_1^2(u)\\  0&1\end{bmatrix}$$
 and  $D$ is a section of $p$  over $U_0.$ 

{\bf Case 2.2}    $\gamma_1\not \equiv 0, \ \   \delta_2\equiv -\delta_1$ in $\tilde U_{ij}.$ Then
$-1=\lambda_{ij}=\frac{-d_{i0}}{d_{j0}}.$

Then $\beta_1=\frac{\al_1\delta_1-d_{i0}}{\gamma_1}=-\beta_2=-\frac{\al_2\delta_2-d_{j0}}{\gamma_2}$ and $\nu_{ij}=0.$
Similarly  to {\bf Case 1.2}  we get that  $p^{-1}(U_0\setminus \Sigma)\sim(U_0\setminus \Sigma)\times \BP^1_{z}$  and  $D$ is a section of $p$
over $U_0.$

{\bf Case  3.} $\al_1=\al_2=0.$    According to \eqqref{m101}
$$ F_0=  I+W_i=\begin{bmatrix}1&0\\ -\gamma_1^2(u) &1\end{bmatrix}$$
and $\gamma_1^2(u)$ does not vanish in $U_0$  since $\Delta_f=0.$  Thus   $D\cap V_0=\{z=0\}$   that contradicts to  $\mathbf{s}\in\Sigma.$

\end{proof} 
 
 \begin{Remark}   We may assume that a fine covering of $Y$ contains a finite covering of  $U$ since  $U_0\setminus \Sigma$ may be covered by two neighborhoods $U_0\cap\{\al_i\ne 0\}$ and  $U_0\cap\{\gamma_i\ne 0\}$  (see \eqqref{m101}).\end{Remark}

\begin{Lemma}\label{onlyb}   
Let $f\in \Aut(X)_p, f\ne id$  be an automorphism of type {\bf B} with data $D.$
  Assume that there exists an almost section $A$ of $p$ distinct from $D.$  Then  $X$ contains a special configuration.
\end{Lemma}
\begin{proof}  
    Since $A\ne D ,$ and $ A\not\subset \Fx(f)$, we have $A_1:=f(A)\ne D$ and $A_1\ne A.$ Similarly,  $A_2:=f(A_1)\ne D$ and $A_2\ne A_1.$   Let us show that $A_2\ne A.$ 
    
    If $A_2=A,$ then in the fiber  $P_y=p^{-1}(y)$  over the general point $y\in Y$ there is  point  $a=A\cap P_y$ such that $f(a)\ne a$ but $f(f(a))=a.$
But along the general fiber $P_y$ the map $f$ act as translation  $z\to  z+\tau$ where $\tau\ne  0.$ This map has no periodic points except $z\ne \infty.$ This contradiction shows  that $A_2\ne A.$  

  Let us show that  $A,A_1,A_2$ is a special configuration.  For a fiber $P_y$ we have the following options.\begin{itemize}\item $f\mid_{P_y}=id. $ Then $P_y\cap  A=P_y\cap  A_1=P_y\cap  A_2;$
 \item $f\mid_{P_y}$ is translation $z\to  z+\tau$ and  $P_y\cap  A\ne  P_y\cap  D.$   Then $P_y\cap  A, \ P_y\cap  A_1,  \ P_y\cap  A_2$  are pairwise disjoint sets.   \item $f\mid_{P_y}$ is translation $z\to  z+\tau$ and  $a:=P_y\cap  A= P_y\cap  D.$  Then $P_y\cap  A_1=a, P_y\cap  A_2=a.$\end{itemize}
  
 It follows that $A\cap   A_1=A\cap  A_2=A_1\cap  A_2$ and $A,A_1,A_2$ is a special configuration.\end{proof}

   \begin{Corollary}\label{onlyb22}  In the notation of \lemref{onlyb}, if    $X$ is  scarce  and $\Aut(X)_p$ contains an automorphism $f$ of type {\bf B} with data $D$  then  it contains no automorphisms of type {\bf B} with another data and no 
automorphisms of type {\bf A}.\end{Corollary}
\begin{proof}   Indeed,   the existence of such automorphisms would imply the existence of an almost section (in particular, section in case  of type {\bf A}) distinct from 
one contained in $\Fx(f).$\end{proof}

\subsection{Automorphisms of type   {\bf A}}\label{2sec}

\begin{Lemma}\label{p1}  Assume that $X\not \sim  Y\times\BP^1.$
Let $S_1,S_2$ be two sections of $p$ such that 
$S_1\cap  S_2=\emptyset.$ Let $f\in\Aut(X)_p.$  Then one of the following holds.\begin{enumerate}\item  $f(S_1)\subset S_1\cup  S_2;$\item $f(S_2)\subset S_1\cup  S_2;$
\item  $f(S_1\cup  S_2)=S_1\cup  S_2.$\end{enumerate}\end{Lemma}

\begin{proof}   Note that a fiberwise automorphism moves a section to a section.  Let $S_3=f(S_1), S_4=f(S_2).$   
Since $S_1\cap S_2=\emptyset, $  we have $S_3\cap  S_4=\emptyset.$ According to \lemref{p0} it may happen only if  the pairs ($S_3, S_4$) and ($S_1, S_2$)   share a section. This may happen only if one of the sections of the pair 
($S_3, S_4$)  coincides with either $S_1$ or $S_2.$
\end{proof}
Recall that 
 the group $G_0$ of all  those  $f\in\Aut(X)_p$ that have data $(S_1, S_2)$
 is isomorphic to $\BC^*$ (see \lemref{typea}).

 Assume that the holomorphic  line bundle $\LL(S_1,S_2)$ is defined by cocycle $\{\lambda _{ij}\}$ and $\LL(S_1,S_2)^{\otimes 2}$ has a section 
 $T\subset X$  defined by $\mathbf a:=\{a_i(y)\},$ with $  a_j(y)=\lambda _{ij}^2a_i(y)\}.$
 
 Define  $$\phi_T:X\to X,  \phi_T(y, z_i)=(y,\frac  {a_i(y)}{z_i}).$$
The fixed point set  $\Fx(\phi_T)=\{\phi_T(y,z_i)=(y,z_i)\}$  is defined  by  $T\cap V_i=\{z_i^2=a_i\}.$
If $\phi_T\in\Aut(X)_p,$   then $a_i$ do not vanish.  In this case 
$\mathbf a:=\{a_i\}$ provide a section  of  $\LL_p^{\otimes 2}$  that does not meet the zero section, thus  $\LL_p^{\otimes 2}$ is a trivial bundle and  we may define $z_i$ in such a way that $a_i=a=const\ne 0.$   We will then  write  $T=T_a$ and $\phi_a:=\phi_T.$

\begin{Proposition}\label{twosections}  Let $(X,p,Y)$ be a $\BP^1-$bundle, where $X,Y$ are compact connected complex manifolds,  and $X\not \sim Y\times\BP^1.$ 
Let  $S_1,S_2$ be two sections of $p$ such that 
$S_1\cap  S_2=\emptyset. $  Let $\LL:=\LL(S_1,S_2)$ be the corresponding holomorphic line bundle   over $Y.$  Let \begin{itemize}\item $G_1\subset \Aut(X)_p$   be the subgroup of all $f\in\Aut(X)_p$  such that  $f(S_1)=S_1;$
\item  $G_2\subset \Aut(X)_p$   be the subgroup of all $f\in\Aut(X)_p$  such that  $f(S_2)=S_2;$
\item $G \subset \Aut(X)_p$   be the subgroup of all $f\in\Aut(X)_p$  such that  $f(S_1\cup S_2)=S_1\cup  S_2;$
\item $F_1$ be the additive group of $H^0(Y,\mathcal O(\LL)).$
\item $F_2$ be the additive group of $H^0(Y,\mathcal O(\LL^{-1})).$\end{itemize}
  Then \begin{enumerate} \item  $X$ does not admit a good configuration (see   \defnref{configuration}) if and only if $F_1=F_2=\{0\};$
\item $G_1\cong \BC^*\rtimes F_1;$
\item $G_2\cong\BC^*\rtimes F_2;$
\item 
 either $G=G_0=G_1\cap  G_2\cong \BC^*$ or  $\LL_p^{\otimes 2}$ is a trivial bundle and  $G=G_0\sqcup \phi_a\cdot  G_0$    for some $a\in\BC^*.$ \end{enumerate}\end{Proposition}

\begin{proof}    Let ${\mathbf {\lambda}}=\{\lambda_{ij}\}$ be the cocycle corresponding to $\LL.$     Take  $f\in  G_1$.     Since $S_1=\{z_i=\infty\} $ is $f$-invariant, we have 
\begin{equation}\label{fa}
f(y,z)=(y,a_iz_i+b_i)\end{equation} in $V_i, $  where both $a_i$ and $b_i$ are holomorphic functions in $U_i.$  Since $f$ is globally defined,we have 
$$\lambda _{ij}(a_iz_i+b_i)=a_j\lambda _{ij}z_i+b_j.$$ 
It follows that $a_i=a_j:=a$ is constant (as  globally defined holomorphic function)  and   $b_j=\lambda _{ij}b_i,$ hence $\mathbf b:=\{b_i\} $    is  a section of $\LL.$    On the other hand, every section $\mathbf b:=\{b_i\} $ of  $\LL$   defines $f\in  G_1$ by formula \eqref{fa}.  
 Thus, $G_1$ is isomorphic to the group of matrices 
$$\begin{bmatrix}a&\mathbf b\\0&1\end{bmatrix},$$
 where   $a\in  \BC^*$ and $\mathbf b\in F_1.$ 
We also showed that if $f\in  G_1$ is defined by  $\mathbf b:=\{b_i\}\ne 0$   then $f(S_2)\ne S_2, $ and $f(S_2)\cap S_1=\emptyset.$ If  $f(S_2)\cap S_2=\emptyset, $ then $S_1, f(S_2), S_2$ would be three 
pairwise disjoint section, which contradicts to $X\not \sim Y\times\BP^1.$

Thus $S_1, f(S_2), S_2$  is a good configuration.  

 In opposite direction:  consider a good configuration $S_1,S_2,S_3$ such that $S_3\cap S_1=\emptyset, S_3\cap S_2\ne\emptyset.$ Since $S_3$ is a section of $p$ and does not meet  $S_1$ it is defined by a section
$\mathbf b:=\{b_i\}$ as $z_i=b_i(y), y\in U_i.$ Thus, $F_1\ne\{0\}.$

 The case of $G_2$  and sections that meet $S_1$ but do not meet $S_2$ may be treated in the same way, interchanging   $S_2$ with $S_1$ and $F_1$ with $F_2.$
 This proves (1-3). 

 Let us prove (4).
   If for each $f\in G$ all the points in $(S_1\cup S_2)$ are fixed then, by \lemref{typea}, $G=G_0\cong\BC^*. $    If it is not the case, take  $\phi\in G\setminus  G_0.$ Then $\phi(S_1)=S_2$ and $\phi(S_2)=S_1.$  Thus, $\phi(y, z_i)=\frac  {a_i(y)}{z_i}$ in every $V_i$ and 
  \begin{equation}\label{e2}\lambda _{ij}\frac  {a_i(y)}{z_i}=\frac  {a_j(y)}{\lambda _{ij}z_i}
\end {equation}
where $a_i(y)$ are non-vanishing holomorphic functions in $U_i.$  
Thus $\{a_i(y)\}$ define a section of $\LL^{\otimes 2}.$  
Since $a_i(y)$ never vanish, we get that  $\LL^{\otimes 2}$   is trivial.   Therefore, we may choose $z_i$ in such a way that ${ a_i}=a\in\BC^*.$   Then  $\phi=\phi_a.$

For any other $f\in G\setminus  G_0$ the composition $f\circ\phi\in G_0, $ hence $G=G_0\sqcup \phi_a \cdot  G_0.$ \end{proof}

\begin{Corollary}\label{p34}  Let $(X,p,Y)$ be a $\BP^1-$bundle, where $X,Y$ are compact connected manifolds  and  $X\not\sim Y\times \BP^1.$ Assume that $p$ admits no good configurations  but admits two disjoint sections  $S_1,S_2. $  Then  one of the following holds.
\begin{enumerate}
\item  $\Aut(X)_p\cong \BC^*;$  
\item the holomorphic    line bundle $\LL(S_1,S_2)^{\otimes 2}$ is trivial 
and 
$\Aut(X)_p=G_0\sqcup \phi_a \cdot  G_0,$ for some   $a\in \BC^*.$   Here $G_0\cong \BC^* $  and $a\in \BC^*$\end{enumerate}
The restriction map $\Aut(X)_p\to \Aut(P_y), \ f\to f \mid_{P_y}$ is a group embedding. 
\end{Corollary}
\begin{proof}  It follows from \propref{twosections} that $F_1=F_2=\{0\},$ thus $\Aut(X)_p=G.$\end{proof}

\subsection{ Automorphisms of type    {\bf C}}\label{nos} 

Let $(X,p,Y)$ be a
$\BP^1-$bundle     where $X,Y$ are complex compact connected manifolds. 
Assume that $X\not\sim  Y\times\BP^1$ and  $ f\in\Aut(X)_p, f\ne id$   has type {\bf C}.
The analytic subset $F\subset X$  of all   fixed points     of 
$f$ contains no sections, but contains a bisection $S$ that is   a smooth unramified double cover of $Y$  (see   \lemref{typea}). Further on we use the notation of \lemref{typea}    and \lemref{doublecover}.

 \begin{Lemma}\label{p7}  Assume that $\tilde X:=\tilde X_S\not\sim S\times \BP^1.$ Let $N\subset \tilde X$ be a section of $\tilde p$ distinct from   $S_+$  and $S_-. $  Then 
 $N_X:=p_X(N) $  is a section of $p$  and $ (S_+,S_-, N)$  is not a good configuration.\end{Lemma} \begin{proof}
 Let us show that $p_X:N\to N_X $  is  an unramified  double cover.  Indeed, assume that it is not the case. Since $\tilde X$ is the unramified  double cover
 of $X,$   the preimage $p_X^{-1}(x)$ contains precisely two points for every $x\in N_X.$ Thus if $p_X^{-1}(N_X)\ne N,$  the preimage 
$p_X^{-1}(N_X)$  consists of two irreducible components, $N$ and $ N_1$ . Moreover, since $p_X$ is unramified, $N\cap  N_1=\emptyset. $
 It follows that there are two distinct pairs of non-intersecting sections   of $\tilde p, $ namely,  $S_+, S_-$ and $N, N_1.$   According the  \lemref{p0}, 
$\tilde X\sim S\times \BP^1,$ which gives us  a contradiction. It follows that $N$ is a double cover
of  $N_X.$   Let $s\in S, y=p(s)=p(\invl(s)).$   Then 
$$ p_X^{-1}(N_X\cap  P_y)=N\cap p_X^{-1}(P_y)=
N\cap (\tilde p^{-1}(s)\cup\tilde p^{-1}(\invl(s)))$$
contains two points  (since $N$ meets every fiber of $\tilde p$   at a single point.)

Since $N$ is double cover of $N_X$ it follows that $(N_X\cap  P_y)$ contains precisely one point. 
Therefore, $N_X$  is a section of $p.$ 

 Assume that $N$ meets $S_+$ at a point $a=(s,s)\in\tilde X, s\in S .$  Then it meets  $S_-$ at the point  $\invl(a)=(\invl(s),s)$  since $
p_X(a)=p_X(\invl(a)).$  Thus, $N$ meets both $S_+$ and $S_-$ and  the  configuration is not good. 
\end{proof}

\begin{corollary}\label{p8}   Assume that  $(X,p,Y)$ is  a $\BP^1-$bundle  that admits a non-identity automorphism
$f\in \Aut(X)_p$ of type {\bf C} with data $S.$  Assume that the corresponding double cover   $\tilde X_S\not \sim S\times \BP^1.$Then  
\begin{enumerate}\item one of the following holds:\begin{itemize}
\item  $\Aut(\tilde X)_{\tilde p}\cong \BC^*;$ 
\item$\Aut(\tilde X)_{\tilde p}=\tilde G_0\sqcup \phi_a \cdot  \tilde G_0,$  where $\tilde G_0\cong \BC^*$ and $\phi\in\Aut(\tilde X)_{\tilde p}$ interchanges $S_+ $ with  $S_-.$
\end{itemize}
\item 
The restriction map $\Aut(X)_p\to \Aut(P_y), \ f\to f \mid_{P_y}$ is a group embedding  for every $y\in Y.$
 \item  the map $h\mapsto\tilde h$ is a   group  embedding of $\Aut(X)_p$ to $\Aut(\tilde X)_{\tilde p}.$ 
\end{enumerate}
 \end{corollary}
 
 \begin{proof}  Since, by \lemref{p7}, there are no good configurations in $\tilde X_f ,$ item (1)  follows from \corref{p34}  applied to $\tilde X.$

   Take  $u\in S, \ t\in Y, t=p(u).$ If  $f\bigm |_{P_t}=id,$  then, by construction, 

- \  $\tilde f\bigm |_{P_u}=id,$ , hence 

- \  $\tilde f=id,$   (by \corref{p34}  applied to $\tilde X$), hence

- \ $\tilde f\bigm |_{P_{s}}=id$  for every $s\in S,$  hence 

 - \  $f\bigm |_{P_{y}}=id$   for $y=p(s)\in Y.$
 
 Hence $f$ is uniquely determined by its restriction
to the   fiber $ P_t= p^{-1}(t).$  This proves (2).

 On the other hand, in (2) was shown that $\tilde h=id$  implies   $f\bigm |_{P_{y}}=id$  for every $y\in Y,$ i.e. $h=id.$ Therefore $h\mapsto\tilde h$  is an embedding.  This proves (3).
  \end{proof}

\begin{Lemma}\label {Yf}  Assume that \ $ f\in \Aut(X)_p,\  f\ne id,$ and $f$ is of type {\bf C} with  Data (bisection) $S.$

(1)  If   the corresponding double cover (see case {\bf C}) $\tilde X:=\tilde X_S$  is  not isomorphic to $S\times \BP^1$ then the group  $\Aut(X)_p$ 
has exponent 2 and consists of $2$ or 
 4 elements.

 (2)
   If  $\tilde X$ is isomorphic to $S\times \BP^1$  then there  are two  disjoint  sections  $S_1,S_2\subset X$ of $p.$  
   Moreover,   if $X\not\sim  Y\times\BP^1 $ then  $\Aut(X)_p$  is  a disjoint union of its abelian complex Lie subgroup $\Gamma\cong \BC^*$ of index 2 and its coset 
   $\Gamma^{\prime}.$ The subgroup $\Gamma$ consists of  those $f\in\Aut(X)_p $ that fix $S_1$ and $S_2.$  The coset  $\Gamma^{\prime}$ consists  of  those $f\in\Aut(X )_p$ that interchange $S_1$ and $S_2.$  Moreover,  the restriction homomorphism $
\Aut(X)_p\to\Aut(P_y), \    f\to f\mid_{P_y}$ is a group  embedding for every $y\in Y.$
\end{Lemma}

\begin{proof}   We modify  the proof of  \cite[Lemma 4.7]{BZ20}.

Choose a point  $a\in S. $ Let $b=p(a)\in Y.$ It means that $a$ sits in the  two elements  set $S\cap P_b.$
  The lift $\tilde f$ of $f$ onto $\tilde X $  has type {\bf A} with Data $(S_+,S_-)\subset \tilde X,$ since points of $S$ are fixed by $f.$
   It is determined uniquely by its restriction to $P_a$  (see \propref{twosections}).  For the corresponding holomorphic line bundle $\tilde \LL:=\tilde \LL(S_ -,S_+)$  
the  section   $S_+$ is the zero section.  Let \begin{itemize}
 \item
 $\{\tilde U_j\}$ be a fine covering of $S;$
  \item  $(u, z_j)$ be  local coordinates in $\tilde V_j=\tilde p^{-1}(\tilde U_j), $   such that
$z_j \mid_ {S_+}=0,  z_j \mid_ {S_-}=\infty;$ 
   
   \item  $a\in \tilde U_i, \ \invl(a)\in \tilde U_k$ and $\tilde U_k\cap \tilde U_i=\emptyset;$ 
     \item $b=p(a)=p(\invl(a))\in Y.$
   \end{itemize}
It was shown in \cite[Lemma 4.7]{BZ20}  that 
\begin{itemize}
\item[A.]
 If we define 
the isomorphism    $\alpha:\ov\BC_{z_i}\to \ov\BC_{z_k}$    in such a way that the following diagram is commutative
\begin{equation}\label{diagram4}
\begin{CD}
 P_b  @>{(a,id)}>>a\times P_b@>{z_i}>>\ov \BC_{z_i}\\
@V id VV@V{ }VV@V \alpha VV\\
P_b  @>{ (\invl(a),id)}>>\invl(a)\times P_b@>{z_k}>>\ov \BC_{z_k}
\end{CD},\end{equation}
  then $$
z_k=\alpha(z_i)=
\frac{\nu }{z_i}$$
for some $\nu=\nu(a)\ne 0.$ 
\item[B.] Consider an automorphism $h\in\Aut(X)_p.$  Let $\tilde h$ be its pullback to 
$\Aut(\tilde X)_{\tilde p}$  defined by $\tilde  h(s,x)=(s,h(x)).$ 
Let $n_1(z_i)=\tilde h \bigm |_{\tilde  P_a}, $  which  means that $h(a,z_i)=(a,n_1(z_i)).$ 
Let $n_2(z_k)=\tilde h \bigm |_{\tilde P_ {\invl(a)}}, $  which  means that $h( \invl(a),z_k)=(a,n_2(z_k)).$ 
 Then

 \begin{equation}\label{116} \frac {\nu}{n_1(z_i)}=         \al(n_1(z_i))=n_2(\al(z_i))=n_2(\frac{\nu}{z_i}).\end{equation}
 \end{itemize}
Proof of {\it (1)}. Assume that $\tilde X\not\sim S\times\BP^1$.  
 
  According to \corref{p8}, if 
$\tilde h \in \Aut(\tilde X)_{\tilde p}$ then  either 
  $\tilde h(s,z_j)=\lambda  z_j, $ 
  or $h(s, z_j)=\frac{\lambda }{ z_j}$  in every $\tilde U_j$ of our fine covering,
where  $\lambda\in\BC^*$  does not depend on $s$  or $j.$

Fix $a\in S.$  According to item B  one of following two conditions holds.

  (a)   $n_1(z_i)=\lambda z_i, \ n_2(z_k)=\lambda z_k,$ \  $z_k=\frac{\nu(a)}{z_i}$  and from  \eqref{116} 
$$\frac {\nu(a)}{\lambda z_i}=\lambda\frac{\nu(a)}{z_i}.$$  

(b)  $n_1(z_i)=\frac{\lambda}{ z_i}, \ n_2(z_k)=\frac{\lambda}{ z_k},$ $z_k=\frac{\nu}{z_i}$ and   from  \eqref{116} 
$$\frac {\nu z_i}{\lambda}=\frac {\lambda z_i}{\nu}.$$
In the former case $\lambda=\pm 1, $    in the latter  case $\lambda=\pm \nu.$ Hence, at most 4 maps are possible. Clearly,  the 
 squares of all these maps are the identity map.  
 
  Note, that all   the calculations are done for the  fiber of $\tilde p$  over the point $a.$
 We use the  fact that the map $\tilde h$  is defined by its restriction to a fiber. 
   Apriori, $\nu$  could depend on a fiber. But since $\lambda $  does not, we got as a byproduct that  the same  is valid  for $\nu.$
 
Proof of  {\it  (2)}.  Assume that $\tilde X\sim S\times\BP^1$.  Let $\zeta:S\times\BP^1\to \BP^1$
 be the projection on the second factor, let $\zeta_1=\zeta\mid_{S_+}, \zeta_2=\zeta\mid_{S_-}.$
Since $S_+\cap  S_-=\emptyset,$  the function $z=\frac{\zeta-\zeta_1}{\zeta-\zeta_2}$ is well defined  on $\tilde X$.

Since  $z=0$ on $ S_+=\{(s,s)\}$ and $z=\infty$ on $S_-=\{(s,\invl(s))\}$   we may assume that  $z_j=z $ for all $j.$   Recall that  for every $s$ 
\begin{equation}\label{nu}
  \invl(s,z)=(\invl(s),\alpha(z))=(\invl(s),\frac{\nu(s)}{z}).\end{equation}
   This implies that $\nu(s)$ is a holomorphic function on $S$, hence 
$\nu=const$.  
 From  \eqref{nu} we get that two disjoint  sections $N_1=\{(s, z=\sqrt{\nu})\}$ and  $N_2=\{ (s,z=-\sqrt{\nu})\}$ (for some choice of $\sqrt{\nu}$) 
  are invariant  under the involution, which means that their images are two disjoint  sections  $S_1, S_2,$  respectively, in $X.$    

 Thus, $X$ has two disjoint sections. Let us show that there is no good configuration  in $X$. Assume that $S_3$ is a third section  (of $p$) in $X$. On  $\tilde S_3=p_X^{-1}(S_3)\subset \tilde X$ the function $z$ is either constant or get all values in  $\ov{\BC}.$ If it is constant, then $X$  has three disjoint sections($S_1, S_2, S_3$), thus $X=Y\times \BP^1.$  If $z$ takes on all the values  on   $\tilde S_3 , $ then $S_3$ meets both $S_1$ and $S_2, $   thus $S_1, S_2 ,S_3$
 is not a good configuration. 
 
 Now {\it(2)} follows from \corref{p34}.

\end{proof}

We have proved (see \lemref{typeb2})  that if 
$X\not\sim Y\times\BP^1$ and 
 there is $f\in\Aut(X)_p, f\ne id,$  of type {\bf B}  then $\Aut(X)_p$ contains a subgroup isomorphic to $(\BC^+)^n$  for some positive integer $n.$

  \begin{Corollary}\label{p5} Assume that $X\not\sim Y\times\BP^1$ and $\Aut(X)_p, $ contains an automorphism $f\ne id$ of type {\bf B}.  Then $\Aut(X)_p $  
contains no  automorphisms  of type {\bf C}.\end{Corollary}
\begin{proof}   Assume that $\Aut(X)_p,$   contains an   automorphism of type {\bf C}. Then
by \lemref{Yf}  $\Aut(X)_p$   is either   finite or  consists of two cosets isomorphic to $\BC^*$;  in both cases  $\Aut(X)_p$   does not
contain a Lie subgroup $\Gamma\cong  (\BC^+)^n$ with 
$n>0.$ \end{proof}

\begin{Proposition}\label{kahler} Let  $(X,p,Y) $  be a  $\BP^1-$bundle, where $X,Y$ are complex compact connected manifolds, and    $Y$ is   K\"{a}hler and not uniruled.   Then $\Aut(X)$ is Jordan. \end{Proposition}
\begin{proof} Indeed, we proved  that three cases are possible. 
\begin{enumerate}\item    \ $\Aut(X)_p=\{id\}.$ Then $\Aut(X)$ embeds into $\Aut(Y)$ that is Jordan according to \cite{Kim}.
\item    \ $\Aut(X)_p$ contains an automorphisms of type {\bf A} or  {\bf B}. Then $X=\BP(\EE)$ for  some rank 2 vector bundle $\EE$ on $Y.$
Thus, $X$ is K\"{a}hler  (\cite  [Proposition 3.5]{Voisin}).
\item    \ $\Aut(X)_p$ contains an automorphisms of type {\bf C}. Then the  double cover $\tilde X$  of $X$  fits into Case 2. Thus, $X$ is K\"{a}hler.
\end{enumerate}
 In Cases 2 and 3  $\Aut(X)$ is Jordan, once more, according to \cite{Kim}.
\end{proof}

  \section {Structure of $\Aut_0(X)$  and $\Aut(X)$ }\label{Aut0}

In this section we prove the main result of this chapter.  Namely, that the group $\Aut(X)$ is very Jordan provided that the $\BP^1-$bundle $(X,p, Y)$  is scarce.

\begin{theorem}\label{main}  Let  $(X,p,Y) $  be a  $\BP^1-$bundle, where $X,Y$ are complex compact connected manifolds,  $X$ is not biholomorphic  to the direct product    $Y\times \BP^1$ and    $Y$ is   K\"{a}hler and not uniruled.   Assume that $(X,p,Y) $  is scarce.

 Then:
 \begin{itemize}
 
\item [a)]  The connected  identity component  $\Aut_0(X)$ of  the complex Lie  group $\Aut(X)$   is {\bf commutative}; 
\item [b)] The group $\Aut(X)$ is very Jordan. More precisely, there is a short exact sequence  
 \begin{equation}\label{goal111}     1\to \Aut_0(X)\to \Aut(X)\to F\to 1, \end{equation}
where  $F$ is a bounded   group.
\item[c)] The commutative group $\Aut_0(X)$ sits in a short exact sequence of complex Lie groups
\begin{equation}\label{goal3}
1\to\Gamma\to \Aut_0(X)\to H\to 1,\end{equation}
where $H$ is a complex  torus and  one of the following  conditions holds:  
 \begin{itemize}\item $\Gamma=\{id\}$, the trivial group;
 \item $\Gamma\cong(\BC^{+})^n.$ 
\item $\Gamma\cong\BC^{*}.$
\end{itemize}

\end{itemize}\end{theorem}
\begin{proof}  
We know that the set of almost sections  is either infinite or contains at most 2 of them   (by  \lemref{pp00} and \remarkref{imagesection}). 

Consider cases. 

{\bf Case 1.} There are no almost sections of $p.$  Then, by \lemref{Yf}, $\Aut(X)_p$ is finite.

{\bf Case 2.}  $p$ has only two almost sections,  $A_1,A_2,$  that  meet. 

Assume that $f\in\Aut(X)_p, \  f\ne id.$   Since $f$ moves almost sections to almost sections,  $A_1\cup  A_2$ is invariant under $f.$    According to \propref{twosections},  the following cases are possible: \begin{itemize} 
\item   Points of $A_1$ are fixed points of $f.$ Then the same is true for $A_2.$   Since $A_1$ and $A_2$ meet,    $f$ is   neither   of type {\bf A} or  of  type {\bf C}. Since they are distinct, $f$ cannot be of  type {\bf B} (see \lemref{onlyb}).   Thus  $f= id$  and    $\Aut(X)_p=\{id\}.$

\item  Not all  points of $A_1$ are fixed points of $f.$  That means 
$f(A_1)=A_2, f(A_2)=A_1.$  Assume that $g\ne f\in \Aut(X)_p,   g\ne id.$  Since $g\ne id,$  it too does not fix points of $A_1$(due to the previous case).
Then for $h:=g\circ  f$ we have $h(A_1)=A_1, h(A_2)=A_2.$   Hence, as in previous item,  $h=id.$ It follows that  $f^2=id,$ $g=f =f^{-1}.$ \end{itemize}
$\Aut(X)_p$  is finite.

{\bf Case 3.}   $p$ has precisely one  almost section.  Then there are no automorphisms of type {\bf  A},
since there are no two disjoint sections.  If  $\Aut(X)_p$  contains no automorphisms  of type {\bf  B} then, by \lemref{Yf},  $\Aut(X)_p$ is finite. 
If  $\Aut(X)_p$  contains an automorphism  of type {\bf  B}, then, thanks to 
 \corref{p5},  $\Aut(X)_p$ contains no automorphisms of type {\bf  C}.  Since all automorphisms  of type {\bf  B}  have to share this section in their sets of fixed points,  $\Aut(X)_p\cong(\BC^+)^n$   
by \propref{c11}  (unless $\Aut(X)_p=\{id\}$).

{\bf Case 4.}  $p$ admits precisely  two almost sections $S_1,S _2$ and they   do not meet. Than they are sections.  But $X$ admits no good configuration. Thus, by 
\propref{twosections}  group   $\Aut(X)_p$  contains a subgroup isomorphic to $\BC^*$   of index at most 2.

{\bf Case 5.} $X$   is scarce  and all almost sections pairwise meet  (in particular, all sections pairwise meet).  Then $\Aut(X)_p$ contains no automorphism of type   {\bf A}. If $\Aut(X)_p$ contains an automorphism of type   {\bf B} then, by \lemref{onlyb} the set of sections cannot be scarce (assuming that there more than 1 of them), contradicition.  Hence, by \lemref{Yf}, $\Aut(X)_p$  is finite.

{\bf Case 6.} $X$   is  scarce  and admits  two  disjoint sections $S_1,S _2.$  By \lemref{p000}, $X$ admits no good configurations, and by \lemref{onlyb}   no automorphisms  of type {\bf  B} . By  \corref {p34}   $\Aut(X)_p $ contains a subgroup isomorphic to $\BC^*$   of index at most 2 .

The proof now 
repeats 
the proof of \cite[Theorem 5.4]{BZ20} with only one modification: $\BC^+$ should be changed to $(\BC^+)^n $ and,accordingly \lemref{commutative}   
 should be applied. The   group  $\Aut(X)_p$ may be included into the short exact sequence
\begin{equation}\label{short1}    1\to (\Aut(X)_p\cap\Aut_0(X))\to \Aut_0(X) \overset{\tau}{\to}   H_0\to 1,\end{equation} 
  where $H_0=\tau(\Aut_0(X))\subset \Tor(Y)$ is a torus  (see \remarkref{tau-meromorphic}). According to Cases {\bf 1-6}, one of the following  holds:
\begin{itemize} 
 \item  $\Aut(X)_p\cap\Aut_0(X)$ is finite (thus $\Aut_0(X)$ is a complex torus);
\item   $\Aut(X)_p\cap\Aut_0(X)\cong (\BC^+)^n$ ;
\item  $\Aut(X)_p\cap\Aut_0(X)\cong (\BC^*)$ ;
 \end{itemize}

Thus, due to  \lemref{commutative}, the group $\Aut_0(X) $ is commutative.  Now the theorem follows from the fact that $\Aut(X)/ \Aut_0(X) $ is bounded (see \propref{abounded}). 
   \end{proof}

\section{ Rational bundles over poor manifolds}\label{tori}

In this section we consider rational bundles over poor manifolds.  We prove that if $Y$ is poor then $p$ is scarce and  the results of the previous section   may be applied.

 \begin{Definition}\label{poor} We say that a compact connected complex  manifold  $Y$ of positive dimension is {\sl  poor} if it enjoys the following properties.
\begin{itemize}

\item
$Y$ does not contain analytic subspaces of codimension 1   ({\it a fortiori},  the algebraic dimension $a(Y)$ of $Y$ is $0$).
\item
$Y$  does not contain  rational curves, i.e., it is meromorphically hyperbolic in the sense of Fujiki \cite{Fu80}.
\end{itemize}\end{Definition}

 A  complex torus  $T$ with $\dim(T)\ge 2 $ and  $a(T)=0$ is a { poor}  K\"{a}hler manifold.  Indeed, a complex torus  $T$  is a   K\"{a}hler   manifold that does not contain rational curves.
 If $a(T)=0, $ it contains no analytic subsets of codimension 1 \cite[Corollary 6.4,   Chapter 2]{BL}.  An explicit example of such a torus   of dimension 2 is given in \cite[ Example 7.4] {BL}.   Explicit examples of poor tori of any dimension are presented   \cite{BZ22}.
  Another 
  example of a  {poor} manifold is   provided by  a non-algebraic $K3$ surface  $S$ with the N\'{e}ron-Severi group $\NS(S)={0}$   (see \cite[Proposition  3.6, Chapter VIII]{BHPV}).

Further on  $Y$ is  assumed to be a compact connected complex       manifold.

\begin{Proposition}\label{p1bundle}(\cite[Proposition 3.6]{BZ20}).  Let $(X,p,Y)$ be an equidimensional rational   bundle. Assume that $Y$ contains no
analytic subsets  of codimension   1.    Then $(X,p,Y)$  is a $\BP^1-$bundle.\end{Proposition}

\begin{proof}  Let   $\dim  (Y)=n, $  and $$S=\{x\in X   \ | \  \rk (dp)(x)<n\}$$  be the set of all points in $X$
where the differential $dp$  of $p$ does not have the maximal rank. Then $S$ and $\tilde S=p(S)$ are analytic subsets of $X$
and  $ Y,$ respectively  (see, for instance,
\cite [Theorem 2, Chapter  VII]{Narasimhan},
 \cite[Theorem  1.22]{PR}, \cite{Re}).  Moreover, $\codim(\tilde S)=1 $ (\cite{Ra}).  Since 
$Y$ contains no
analytic subsets of codimension  \ 1, we obtain: $\tilde S=\emptyset.$  Thus the holomorphic map $p$ has no singular fibers. \end{proof}

\begin{Lemma}\label{section}  Let $(X,p,Y)$ be a $\BP^1-$bundle, and $ \dim (Y)=n. $
For an almost section $A$ we denote $\Sigma(A)=\{y\in Y \ | \ p^{-1}(y)\subset  A\}.$
 If  $Y$ contains no analytic subsets of codimension 1, then
 \begin{enumerate}
\item a $n-$section has no ramification  points (i.e the intersection $X\cap P_y$  consists in $n$ distinct  point for every $y\in Y$);

\item  if  $A_1,A_2$ are two almost sections        then $p(A_1\cap  A_2)\subset \Sigma(A_1)\cap\Sigma(A_2).$
\item
any two distinct   sections  of $p$ in $X$ are disjoint;
\item if there is an alomost section $A\subset X$ that is not a section then $X$ contain neither sections  nor $n-$sections;
\end{enumerate} \end{Lemma}
\begin{proof}

(1) Let $R$ be an $n-$section   of $p,$ let $A$ be the set of all points  $x\in R$  where the restriction $p \bigm  |_R:R\to Y$ of $p$ onto $R$ is not locally biholomorphic.  Then 
the image   $p(A)$  is either empty or  has pure codimension 1  in $Y$     (\cite[Section 1,  9]{DG}, \cite[Theorem1.6]{Pe}, \cite{Re}). 
Since $Y$  carries no  analytic subsets of codimension 1,  $p(A) =\emptyset.$ Hence, $A=\emptyset.$

(2)     Let $B$ be an irreducible component of $A_1\cap A_2.$ Since $\dim (B)= n-1$, and $\dim (p(B))\le n-2, $ we have $p^{-1}(p(b))\subset B$ for every point $b\in B.$
 Thus, $p(b)\in\Sigma(A_1)\cap\Sigma(A_2).$

(3)  In particular, if $A_1, A_2$ are distinct sections, then $\Sigma(A_1)=\Sigma(A_2)=\emptyset$ and $A:=A_1\cap  A_2=\emptyset.$

(4)  Since $A$ is not a section, there is a point $y\in Y$ such that $P_y=p^{-1}(y)\subset A.$ Thus for any $n-$section   $S$ we have   $S\cap A\ne\emptyset.$
  This contradicts  item (2),  since $\Sigma(S)=\emptyset.$     Hence, such an $S$ does not exist. 
\end{proof}

\begin{Corollary}\label{poorsection}  Let $(X,p,Y)$ be a $\BP^1-$bundle, $ \dim (Y)=n. $
 If  $Y$ contains no analytic subsets of codimension 1, then  one of the following holds.\begin{enumerate}
\item  $X\sim Y\times \BP^1;$
\item $X$ admits two  disjoint sections,    $ \Aut(X)_p $ contains a subgroup $G\cong \BC^* $  of index at most 2;
\item $X$ admits two  meeting  almost sections,  $ \Aut(X)_p$ is finite. 
\item $X$ admits precisely one   almost section   $D, $   then  $ \Aut(X)_p \cong \BC^+ $ (and $D$, by \lemref{3.12}, is a section) or $ \Aut(X)_p =\{id\}; $
\item $X$ admits no almost sections,   $ \Aut(X)_p  $  is finite;
\end{enumerate}\end{Corollary}
\begin{proof}

First, note that since $Y$ does not admit meromorphic functions, for  a
line bundle $\LL$ on $Y$  either $H^0(\LL)=\{0\}$ or $\LL$ is trivial and  $H^0(\LL)\cong \BC.$

Item (1):  Assume that $X$ admits $m\ge 3$ almost sections. By \lemref{section} they are disjoint over an open set $U\subset Y$ that has complement of codimension 2.    Thus  $X\sim Y\times \BP^1$ by \lemref{p00}. 

Item (2) follows from \corref{p34}.

Item (3) is proven in {\bf Case 3} of the proof of \thmref{main}.

Item (4)   follows   from   \propref{c11}: if  $ \Aut(X)_p \ne \{id\} $  then  $ \Aut(X)_p $ is isomorphic to the  additive group of $\BC^m. $    That means that for corresponding line bundle
$0<m=H^0(\LL).$ Hence, $m=1.$

Item   (5)  follows from \lemref{Yf}.
\end{proof}

\begin{Lemma}\label{bimaut}    Let $(X,p,Y)$ be a $\BP^1-$bundle, $ \dim (Y)=n. $
 If  $Y$ is poor then $\Bim(X)=\Aut(X).$\end{Lemma}
\begin{proof}
  Since $Y$ contains no rational curves, it is not uniruled. 
  According to \corref{tauh},  every map $f\in\Bim(X)$ is $p-$fiberwise, i.e. there exists a group homomorphism $\tilde\tau:\Bim(X)\to \Bim(Y)$ 
(see \lemref{tau})  such that for all $f\in\Bim(X)$
$$p\circ f=\tilde\tau(f)\circ p.$$ Since $Y$ contains no rational curves, every meromophic map into $Y$ is holomorphic  (\cite{Fu80}, see  \remarkref{Fujiki}).
   Thus $\tilde\tau(f)\in\Aut(Y).$
   
    For $f\in \Bim(X)$ let $\tilde S_f$ be the  indeterminacy locus of $f$  that is an analytic subspace of $X$ of codimension at least 2 (\cite[page 369]{Re}). 
  Let   $S_f={p(\tilde S_f)},$ which is an analytic subset of $Y$ (\cite{Re}, \cite [Theorem 2, Chapter  VII]{Narasimhan},).
Since $Y$ 
contains  no analytic subsets of codimension 1, $\codim  S_f\ge 2.$  Moreover, $f$ is defined at   all points   of  $X\setminus p^{-1}(S_f).$ 
 By \lemref{3.8} both  $f\in \Bim(X)$  and  $f^{-1}\in \Bim(X)$ may be holomorphically extended to  $X,$   hence we get 
$ \Bim(X)=\Aut(X).$\end{proof}

  We summarize the result in the  following

\begin{Theorem}\label{IntroMain1}  Let  $(X,p,Y) $  be an equidimensional rational bundle over a  poor   K\"ahler  manifold 
 $Y.$    Then:
 \begin{itemize}\item   $(X,p,Y) $ is a $\BP^1-$bundle
 (see  \propref{p1bundle}); \item $\Bim(X)=\Aut(X)$  (see  \lemref{bimaut});

\end{itemize}
Assume additionally that   $Y$ is K\"ahler   and $X$ is not isomorphic to the direct product    $Y\times \BP^1.$ Then:
\begin{itemize}
\item $X$ admits at  most two almost sections (\corref{poorsection}).
\item   The connected  identity component  $\Aut_0(X)$ of  complex Lie  group $\Aut(X)$   is {\bf commutative}  (\thmref{main}); 
\item Group $\Aut(X)$ is very Jordan  (\thmref{main}); 

\item The commutative group $\Aut_0(X)$ sits in a short exact sequence of complex Lie groups\begin{equation}\label{goal2}
1\to\Gamma\to \Aut_0(X)\to H\to 1,\end{equation}
where $H$ is a complex  torus and  one of the following  conditions holds  (\corref{poorsection})):
 \begin{itemize}\item $\Gamma=\{id\}$, the trivial group;
 \item $\Gamma\cong\BC^{+},$  the additive group of complex numbers;
\item $\Gamma\cong\BC^{*},$   the multiplicative  group of complex numbers.  \end{itemize}\end{itemize}
\end{Theorem}

\


\begin{thebibliography}{MM}
\bibitem {Akhiezer}  D. N. Akhiezer, {\it Lie group actions in complex analysis}. Aspects of  Mathematics, E27. Friedr.
Vieweg \& Sohn, Braunschweig, 1995.

\bibitem{AS69}  A. Andreotti, W.   Stoll, {\it    Extension of holomorphic maps}, Annals of  Math., {\bf 72},  (1969), pp.  312--349.

\bibitem{AS71}  A. Andreotti, W.   Stoll, {\it    Analytic and algebraic dependence  of meromorphic functions}, Lecture Notes in Math., {\bf 234},  Springer-Verlag, 1971.





 
 \bibitem{BZ15} T. Bandman, Yu.G. Zarhin, {\sl Jordan groups and algebraic surfaces}. 	Transformation Groups {\bf 20} (2015), no. 2, 327--334. 

\bibitem{BZ17} T. Bandman, Yu.G. Zarhin,   {\it Jordan groups, conic bundles and abelian varieties}.  Algebr. Geom. 4 (2017), no. 2, 229–246.

\bibitem{BZ18} T. Bandman, Yu.G. Zarhin, {\sl  Jordan properties of automorphism  groups of certain open algebraic varieties}. Transform. Groups  {\bf  24} (2019), no. 3, 721–739. 
 
 \bibitem{BZ20}  T. Bandman, Yu.G. Zarhin, {\sl  Bimeromorphic automorphism groups of certain  $\BP^1-$bundles}.  European Journal of Mathematics {\bf 7} (2021), 641--670.

\bibitem{BZ22} T. Bandman, Yu.G. Zarhin, {\sl  Simple tori of algebraic dimension 0},  
Trudy Mat. Inst. Steklova {\bf 320} (2023)	to appear;	arXiv:2106.10308 [math.AG].

 \bibitem{BHPV} W.P. Barth, K. Hulek, C.A.M. Peters,
A.  Van  de Ven,  {\it   Compact  Complex Surfaces}, A Series of Modern Surveys in Mathematics, {\bf 4}, Springer, 2004.

\bibitem{Bi} C. Birkar, {\sl   Singularities of linear systems and boundedness of Fano
varieties},
 Annals of Mathematics, {\bf  193}, No. 2 pp. 347-405.


\bibitem{BL} C. Birkenhake,  H.  Lange,  Complex Tori.   Birkhauser,  Boston Basel Stutgart, 1999.

\bibitem{CAV} C. Birkenhake,  H.  Lange,  Complex Abelian Varieties, 2nd edition.
 Springer-Verlag Berlin Heidelberg New York, i 2004


\bibitem{BM}   S. Bochner,  D. Montgomery,  {\it  Groups on analytic manifolds}. Ann. of Math. (2) {\bf 48} (1947). 659–669.

\bibitem {BF} R. Brauer, W. Feit, {\it An analogue of Jordan's theorem in characteristic}
 $p$. Ann. of Math. (2) {\bf 84} (1966), 119--131.



 \bibitem{C92} F. Campana,  
{\it Connexit\'{e} rationnelle des vari\'{e}t\'{e}s de Fano}.
Annales scientifiques de l’É.N.S. 4e s\'{e}rie,  {\bf  25},  5, (1992), p. 539--545.



\bibitem{CP} F. Campana, Th. Peternell,  {\it Cycle spaces}. In: Encyclopaedia of Mathematical Sciences, vol. {\bf 74},  Several Complex variables, VII,  Sheaf-Theoretical methods in Complex Analysis, Springer Verlag, 1984, pp. 319--349.





\bibitem {CS}   Y. Chen, C. Shramov,  {\it  Automorphisms of surfaces over fields of positive characteristic},
arXiv:2106.15906.  

\bibitem{Collins}   M.J.   Collins, {\sl  On Jordan's theorem for complex linear groups}. J. Group Theory   {\bf 10 } (2007), no. 4, 411–423. 


\bibitem  {CPS}  B. Csik\'{o}s, L. Pyber, E.Szab\'{o},
{\it Diffeomorphism Groups of Compact 4-manifolds are not always Jordan},
arXiv:1411.7524. 

\bibitem{DG} G.  Dethloff,  H. Grauert,  {\it Seminormal Complex Spaces}. In: Encyclopedia of Mathematical Sciences, vol. {\bf 74}, Several Complex variables, VII,  Sheaf-Theoretical methods in Complex Analysis, Springer Verlag, Berlin, 1984, pp. 206--219.




\bibitem {Fi}   G.  Fischer,  Complex Analytic Geometry. Lecture Notes in Math., {\bf 538}, Springer-Verlag, New York, 1976.

\bibitem {FG}  W.  Fischer, H.  Grauert,
{\it  Lokal-triviale Familien kompakter komplexer Mannigfaltigkeiten}.
Nachr. Akad. Wiss. Gottingen Math.-Phys. Kl. II (1965), 89--94.



\bibitem{Fu78}  A. Fujiki, {\it On Automorphism Groups of Compact Complex Manifolds}. 
Inventiones Math. {\bf 44} (1978), 225-258.

\bibitem{Fu80}  A. Fujiki, {\it On the Minimal models of Complex Manifolds}. Math. Ann. {\bf 253} (1980), 111--128.

\bibitem{Fu81}  A. Fujiki, {\it Deformation of Uniruled manifolds}. Publ. RIMS, Kyoto Univ. {\bf 17}(1981), 687-702.



\bibitem{Golota}  A.  Golota,
{\it Jordan property for groups of bimeromorphic automorphisms of compact Kähler threefolds}. Mat. Sbornik, to appear;    arXiv:2112.02673. 



\bibitem{Graf}
P. Graf, M. Schwald,  {\it On the Kodaira problem for uniruled Kähler spaces}. Ark. Mat. {\bf 58}    (2020), no. 2, 267--284. 


\bibitem{GR} R.C. Gunning, H. Rossi,  Analytic Functions of Several complex variables. Prenitce-Hall, 1965.

\bibitem {HP} A. H\"oring, Th. Peternell, {\sl  Minimal models for Kähler threefolds}. 
Invent. Math.   {\bf 203},    (2016), p. 217--264.

\bibitem {Hu}    Fei Hu, {\sl
Jordan property for algebraic groups and automorphism groups of projective
varieties in arbitrary characteristic}. 
Indiana Univ. Math. J. {\bf 69} (2020), no. 7, 2493--2504.



\bibitem{Huy}  D. Hyubrechts, {\it Complex Geometry.  An Introduction}, Springer-Verlag,  Berlin, Heidelberg,  2005.




\bibitem{Jordan} C. Jordan,  {\it  M\'{e}moire  sur des \'{e}quations differentielles  lin\'{e}ares   \`{a} int\'{e}grale alg\'{e}brique}.  Crelle's
journal  {\bf  84}, (1878)   89-215: \OE uvres II, 13-140.


\bibitem{Kempf} G.R. Kempf, {\it
  Complex Abelian Varieties and Theta functions},  Springer-Verlag,  Berlin Heidelberg New York, 1991.

\bibitem{Kim} Jin Hong   Kim,  {\it Jordan property and automorphism groups of normal compact Kähler varieties}.  Commun. Contemp. Math. {\bf 20} (2018), no. 3, 1750024-1-- 1750024-9.

\bibitem{Kollar}  J. Kollar,  {\it Rational curves  on algebraic varieties}, Springer-Verlag, 1996.

\bibitem {Kuz}  A. Kuznetsova, {\it  Automorphisms of quasi-projective surfaces over fields of finite characteristic}.  J.  Algebra,
{\bf 595}, 1, (2022),  271--278.

\bibitem{Lang} S. Lang, Algebra, Revised 3rd edition.  GTM {\bf 211}, Springer Science, 2002.

\bibitem{LP} M.J. Larsen, R. Pink, {\it Finite subgroups of algebraic groups}. J. Amer. Math. Soc. {\bf 24} (2011), 1105--1158.

\bibitem{Levi}  E.E. Levi,  {\it   Studii sui punti singolari essenziali delle
funzioni analitiche di due 0 piu variabili complesse}.
Annali di mat. pur. apple, {\bf 17:3} (1910), 61--87.


\bibitem{Lie}  D. Lieberman, {\sl Compactness of the Chow scheme: applications to automorphisms and deformations of     K\"{a}hler     manifolds}. In: Fonctions de plusieurs variables complexes, III (Sém. François Norguet, 1975–1977), pp. 140–186, Lecture Notes in Math. {\bf 670}, Springer, Berlin, 1978.


\bibitem{MS}  L. N. Mann, J. C. Su, {\sl Actions of elementary p-groups on manifolds}.Trans. Amer. Math. Soc. {\bf 106}
(1963), 115--126.

	
\bibitem{MengZhang} Sh. Meng, D.-Q. Zhang, {\sl Jordan property for non-linear algebraic groups and projective varieties}.  Amer. J. Math.  {\bf 140:4} (2018), 1133--1145.

\bibitem{MPZ}Sh. Meng, F. Perroni, D.-Q. Zhang, 
{\sl Jordan property for automorphism groups of compact spaces in Fujiki's class $\mathrsfs C$}.
J. Topology {\bf 15:2} (2022), 806--814.

	
	

\bibitem{Mi88} Y. Miyaoka, {\sl On the Kodaira dimension of minimal threefolds}. Math.
Ann. {\bf 281}   (1988) 325--332.
\bibitem{Mobuchi}  T. Mobuchi, {\it Invariant $\beta$  and uniruled threefolds}.
J. Math. Kyoto Univ., 
{\bf 22},  3, (1982) , 503--554.


\bibitem{Mum66}   D. Mumford, {\sl On the equations defining abelian varieties I}. Invent. Math. {\bf 1} (1966),
287--354.



\bibitem{MR10}  I. Mundet i Riera,
{\sl  Jordan's theorem for the diffeomorphism group of some manifolds}. 
Proc. Amer. Math. Soc. {\bf  138},   (2010), no. 6, 2253-2262.


\bibitem{MR13}  I. Mundet i Riera,
{\sl Finite group actions on manifolds without odd cohomology}.
arXiv:1310.6565.



\bibitem{MR15}  I. Mundet i Riera, A. Turull, {\sl Boosting an analogue of Jordan's theorem for finite groups}. Adv. Math. {\bf 272} (2015), 820--836.


\bibitem{MR16}  I. Mundet i Riera,
{\sl Finite group actions on 4-manifolds with nonzero Euler characteristic}. 
Math. Z. {\bf 282} (2016), no. 1-2, 25--42.


\bibitem{MR17}  I. Mundet i Riera,
{\sl  Finite groups acting symplectically on  $ T^2\times S^2.$} 
Trans. Amer. Math. Soc. {\bf 369} (2017), no. 6, 4457--4483.

\bibitem{MR17b}  I. Mundet i Riera,
{\sl Non Jordan groups of diffeomorphisms and actions of compact Lie groups on
manifolds.}
Transform. Groups {\bf 22} (2017), no. 2, 487--501.


\bibitem{MR18}  I. Mundet i Riera,
{\sl  Finite group actions on homology spheres and manifolds with nonzero Euler
characteristic.} 
J. Topology {\bf 12} (2019), no. 3, 744--758.


 

\bibitem{Narasimhan}  R. Narasimhan,   Introduction to the theory of analytic spaces.
Lecture Notes in mathematics {\bf 25}, Springer-Verlag, 1966. 



\bibitem{OV} È. B.   Vinberg, A. L.    Onishchik,  {\it  Lie Groups and Lie Algebras I: Foundations of Lie Theory}. In:Encyclopaedia of Mathematical Sciences Vol. {\bf 20}	
Springer-Verlag, Berlin, 1993,  pp. 1-94.



\bibitem{Pe} Th. Peternell, {\it Modifications}.  In: Encyclopaedia of Mathematical Sciences, vol. {\bf 74}, Several Complex variables, VII,  Sheaf-Theoretical methods in Complex Analysis, Springer Verlag,  Berlin, 1984, pp. 286-319.

\bibitem{PR} Th. Peternell, P.Remmert, 
{\it  Differential calculus, Holomorphic maps and linear structures on complex spaces} In: Encyclopaedia of Mathematical Sciences, vol. {\bf 74}, Several Complex variables, VII,  Sheaf-Theoretical methods in Complex Analysis, Springer Verlag,  Berlin, 1984, pp. 286-319.

\bibitem{Pop} V.L.  Popov,  {\it On the Makar-Limanov, Derksen invariants, and finite automorphism groups
of algebraic varieties}. In: Affine algebraic geometry, 289--311, CRM Proc. Lecture Notes {\bf 54}, Amer. Math. Soc., Providence, RI, 2011.

\bibitem{Pop14} V.L.  Popov,  
{\sl Jordan groups and automorphism groups of algebraic varieties.} 
Automorphisms in birational and affne geometry, 185--213, Springer Proc. Math. Stat.,
{\bf 79}, Springer, Cham, 2014.


\bibitem{Pop15} V.L.  Popov,  
{\sl Finite subgroups of diffeomorphism groups.}  Tr. Mat. Inst. Steklova {\bf   289} (2015), 235--241.
Proc. Steklov Inst. Math. {\bf   289} (2015), no. 1, 221--226.




\bibitem{Pop18} V.L.  Popov,  {\it  The Jordan property for Lie groups and automorphism groups of complex spaces.} Math. Notes. {\bf 103} (2018), no. 5--6, 811--819. 



\bibitem{PS14}  Yu. Prokhorov, C. Shramov, {\it Jordan property for groups of birational selfmaps}. Compositio
Math. {\bf 150} (2014),  2054--2072.



\bibitem{PS16}  Yu. Prokhorov, C. Shramov,
{\sl Jordan property for Cremona groups.} 
Amer. J. Math. 1{\bf 38 } (2016), no. 2, 403--418.

\bibitem{PS17}  Yu. Prokhorov, C. Shramov, {\sl
Jordan constant for Cremona group of rank 3}. 
Mosc. Math. J. {\bf 17} (2017), no. 3, 457--509.



\bibitem{PS18} Yu.  Prokhorov, C. Shramov, {\it Finite groups of birational selfmaps of threefolds}. Math. Res. Lett.  {\bf 25}  (2018), no. 3, 957--972.




\bibitem{PS19} Yu.  Prokhorov, C. Shramov, 
{\sl  Automorphism groups of Moishezon threefolds.} 
Mat. Zametki {\bf 106} (2019), no.  4, 636--640;  Math. Notes {\bf 106}  (2019), no.
3-4, 651--655.

\bibitem{PS20} Yu.  Prokhorov, C. Shramov, {\it Bounded  automorphism  groups of complex compact surfaces.}
Mat. Sb. 
{\bf 211} (2020), no. 9, 105--118; Sb. Math., {\bf 211:9} (2020), 1310--1322.

\bibitem{PS19-2} Yu.  Prokhorov, C. Shramov, 
{\it  Finite groups of bimeromorphic selfmaps of uniruled  K\"ahler threefolds}.
Izv. Ross. Akad. Nauk Ser. Mat.  {\bf 84} (2020), no. 5, 169--196; Izvestiya: Mathematics {\bf 84:5} (2020), 978--1001.



\bibitem{PS21-1} Yu.  Prokhorov,  C. Shramov, 
{\it Automorphism groups of compact complex surfaces. }
Int. Math. Res. Not. IMRN, 2021, no. 14, 10490--10520.

\bibitem{PS21} Yu.  Prokhorov, C. Shramov, 
{\it  Finite groups of bimeromorphic selfmaps of non-uniruled Kähler threefolds}, Mat. Sbornik, {\bf 213},  n. 12,  (2022),  86-108.
arXiv:2110.05825 

\bibitem{PS21-c} Yu.  Prokhorov,  C. Shramov, 
{\it Jordan property for Cremona group over a finite field},  
Proceedings of the Steklov Inst. of Math. {\bf 320} (2023), to appear; arXiv:2111.13367.



\bibitem{Ra}  C. P. Ramanujam,  {\it  On a certain purity theorem}, J. Indian Math. Soc. (N.S.) {\bf 34}
(1970), 1–9 (1971).

\bibitem{Re} R. Remmert, {\it  Holomorphe und meromorphe Abbildungen komplexer Räume}.  Math. Ann.  {\bf  133} (1957), 328-370.


\bibitem{Serre56}   J-P. Serre, {\sl G\'{e}ometrie alg\'{e}brique et g\'{e}ometrie  analytique}. Ann, Inst. Fourier {\bf 6}  (1956), 175-189.

\bibitem{Serre20}   J.-P. Serre,   {\sl
 Bounds for the orders of the finite subgroups of G(k)}.  In:
Group Representation Theory, eds. M. Geck, D. Testerman, J. Th\'{e}venaz, EPFL Press, Lausanne 2006.

\bibitem{Serre1} J-P. Serre, {\sl A Minkowski-style bound for the orders of the finite subgroups of
the Cremona group of rank 2 over an arbitrary field}. Moscow Math. J. {\bf 9}
(2009), no. 1, 183--198.
\bibitem{Serre} J.-P. Serre,   {\sl Finite Groups: An Introduction}, 2nd edition.  International Press, Sommerville, MA,  2022.




\bibitem{kostya} C. Shramov, {\sl Fiberwise bimeromorphic maps of conic bundles}.  Internat. J. Math. {\bf 30} (2019), no. 11, 1950059, 12 pp.



\bibitem{ShV} C. Shramov, V. Vologodsky,  {\sl Automorphisms of pointless surfaces}. Arxiv:1807.06477.


\bibitem{ShVb} C. Shramov, V. Vologodsky, {\it  Boundedness for finite subgroups of linear algebraic groups. }
Trans. Amer. Math. Soc. {\bf 374} (2021), no. 12, 9029--9046.


\bibitem{Siu}  {\sl Extension of meromorphic maps  into K\"ahler manifolds}, Annals of Math.  {\bf 102}, \ (1975), 421--462.


\bibitem{Voisin}  C.  Voisin, {\it  Hodge theory and complex algebraic geometry. I}. Translated from the French original by Leila Schneps. Cambridge Studies in Advanced Mathematics {\bf 76}.  Cambridge University Press, Cambridge, 2002. 



\bibitem {W}   J. Winkelmann, {\it Realizing countable groups as automorphism groups of Riemann surfaces.}  Doc. Math. {\bf 6} (2001), 413--417. 

\bibitem{Yasinsky}  E.  Yasinsky,   {\it  The Jordan constant   for Cremona group of rank 2}.
Bull. Korean Math. Soc. {\bf    54}  (2017), No. 5, pp. 1859--1871.





\bibitem{Zar14} Yu.G. Zarhin, {\sl Theta groups and products of abelian and rational varieties}. Proc. Edinburgh Math. Soc. {\bf 57:1} (2014), 299--304.



\bibitem{Zar15} Yu.G. Zarhin, {\sl Jordan groups and elliptic ruled surfaces}. 
Transformation Groups {\bf 20} (2015), no. 2, 557--572.


\bibitem{Zar19}Yu.G. Zarhin, {\it Complex tori, theta groups and their Jordan properties}.
In: Algebra, Number Theory and Algebraic Geometry.
Tr. Mat. Inst. Steklova, 307 (2019),  32–62;    
   Proc. Steklov Inst. Math., 307 (2019), 22–50.



\bibitem{Zim}  B. Zimmermann, {\it On Jordan type bounds for finite groups acting on compact 3-manifolds}.
Arch. Math. (Basel) {\bf 103} (2014), no. 2, 195-200.
\end{thebibliography}
 \end{document}